\documentclass[10pt]{article}
\usepackage{amsmath,amsfonts,amssymb,latexsym,epic,eepic}
\usepackage[overload]{empheq}
\usepackage[dvips]{graphicx,graphics,epsfig,color}
\usepackage{ifthen}
\usepackage{geometry}
\usepackage{amsthm}
\usepackage{mathrsfs}
\usepackage{tikz,pstricks,fp}
\usepackage{multirow}
\usepackage{array}
\usepackage{dsfont}
\newtheorem{thrm}{Theorem}[section]
\newtheorem{lem}[thrm]{Lemma}

\newtheorem{prop}[thrm]{Proposition}
\newtheorem{defi}{Definition}[section]
\newtheorem{rmrk}{Remark}[section]
\usepackage[colorlinks=true, pdfstartview=FitV, linkcolor=blue, citecolor=red, urlcolor=blue]{hyperref}
\def\dsp{\displaystyle}

\newcommand{\mcal}[1]{\mathcal{#1}}

\definecolor{vertf}{rgb}{0,0.55,0.1}
\definecolor{bclairf}{rgb}{0.40,0.65,0.89}
\definecolor{or}{rgb}{0.98,0.6,.1}
\definecolor{light-gray}{gray}{0.8}

\newcommand{\disc}{\mathcal{D}}
\newcommand{\edge}{\sigma}

\newcommand{\edges}{{\mathcal E}}
\newcommand{\edgesd}{\tilde {\edges}}
\newcommand{\edgesint}{{\mathcal E}_{{\rm int}}}
\newcommand{\edgesintn}{{\mathcal E}_{n,{\rm int}}}
\newcommand{\edgesext}{{\mathcal E}_{{\rm ext}}}
\newcommand{\edgesdint}{{\edgesd}_{{\rm int}}}
\newcommand{\edgesdext}{{\edgesd}_{{\rm ext}}}

\newcommand{\mesh}{{\mathcal M}}
\newcommand{\edged}{\epsilon}

\newcommand{\fluxd}{F_{\edge,\edged}}

\def\xR{{\mathbb R}}

\def\xN{{\mathbb N}}
\def\xH{{\rm H}}
\def\xbfH{{\rm \bf H}}
\def\xbfW{{\rm \bf W}}
\def\xbfL{{\rm \bf L}}
\def\xW{{\rm W}}
\def\xL{{\rm L}} 
 
\newcommand{\bff}{{\boldsymbol f}}
\newcommand{\bfn}{{\boldsymbol n}}
\newcommand{\bfu}{{\boldsymbol u}}
\newcommand{\bfa}{{\boldsymbol a}}
\newcommand{\bfb}{{\boldsymbol b}}
\newcommand{\bfc}{{\boldsymbol c}}
\newcommand{\bfd}{{\boldsymbol d}}
\newcommand{\bfQ}{{\boldsymbol Q}}

\newcommand{\bfpsi}{{\boldsymbol \psi}}
\newcommand{\bfdelta}{{\boldsymbol \delta}}
\newcommand{\bfv}{{\boldsymbol v}}
\newcommand{\bfw}{{\boldsymbol w}}
\newcommand{\bfx}{{\boldsymbol x}}

\newcommand{\bfy}{{\boldsymbol y}}

\newcommand{\dv}{\partial}

\newcommand{\gradi}{{\boldsymbol \nabla}}
\newcommand{\lapi}{{\boldsymbol \Delta}}
\newcommand{\dive}{{\rm div}}
\newcommand{\rot}{{\rm curl}}
\newcommand{\divv}{\boldsymbol{\rm div}}
\newcommand{\norm}[1]{{\lVert #1 \rVert}}

\newcommand{\snorm}[1]{{| #1 |}}

\newcommand{\xPi}{\Pi}
\newcommand{\xP}{\mcal{P}}
\newcommand{\n}{_{n}}
\newcommand{\dx}{\,{\rm d}\bfx}
\newcommand{\dy}{\,{\rm d}\bfy}

\newcommand{\dedge}{\,{\rm d}\sigma}

\newcommand{\dt}{\,{\rm d}t}

\newcommand{\eg}{\emph{e.g.}}
\newcommand{\ie}{\emph{i.e.}}
\newcommand{\Ind}{\mathcal{X}}

\newcommand{\Ds}{{\scalebox{0.6}{$D_\edge$}}}

\newcommand{\bfomega}{{\boldsymbol \omega}}

\numberwithin{equation}{section} 
\title
{\large
\textbf{A CONVERGENT FV -- FEM SCHEME FOR THE STATIONARY COMPRESSIBLE NAVIER-STOKES EQUATIONS}}

\author{Charlotte Perrin\footnote{Aix Marseille Univ, CNRS, Centrale Marseille, I2M, Marseille, France. Email: charlotte.perrin@univ-amu.fr} \hspace{2ex} Khaled Saleh\footnote{Universit\'e de Lyon, CNRS UMR 5208, Universit\'e Lyon 1, Institut Camille Jordan, 43 bd 11 novembre 1918; F-69622 Villeurbanne cedex, France. Email: saleh@math.univ-lyon1.fr}}

\begin{document}

\maketitle

\begin{small}
\begin{center}
{\bf Abstract}
\end{center}
In this paper, we propose a discretization of the multi-dimensional stationary compressible Navier-Stokes equations combining finite element and finite volume techniques.
As the mesh size tends to $0$, the numerical solutions are shown to converge (up to a subsequence) towards a weak solution of the continuous problem for ideal gas pressure laws $p(\rho)=a \rho^\gamma$, with $\gamma > 3/2$ in the three-dimensional case.  
It is the first convergence result for a numerical method with adiabatic exponents $\gamma$ less than $3$ when the space dimension is three.
The present convergence result can be seen as a discrete counterpart of the construction of weak solutions established by P.-L. Lions and by S. Novo, A. Novotn\'y. 


\end{small}

\small
\tableofcontents
\normalsize
\section{Introduction}
Let $\Omega$ be an open bounded connected subset of $\xR^d$, with $d=2$ or $3$, with Lipschitz boundary.
We consider the system of stationary isentropic Navier-Stokes equations, posed for $\bfx\in\Omega$: 
\begin{subequations} \label{eq:pb}
\begin{align}\label{eq:pb_mass} &
\dive( \rho \bfu) = 0,
\\[1ex] \label{eq:pb_mom} &
\divv(\rho \bfu \otimes \bfu) 
 - \mu \lapi \bfu- (\mu+\lambda)\gradi (\dive \,\bfu) + a\gradi \rho^\gamma = \bff. 
\end{align}
\end{subequations}
The quantities $\rho \geq 0$ and $\bfu=(u_1,..,u_d)^T$ are respectively the density and velocity of the fluid, while $\bff$ is an external force.
The pressure satisfies the ideal gas law with $ a > 0$ and $\gamma >1$.
Equation \eqref{eq:pb_mass} expresses the local conservation of the mass of the fluid while equation \eqref{eq:pb_mom} expresses the local balance between momentum and forces. The viscosity coefficients $\mu$ and $\lambda$ are such that $\mu>0$ and $\mu+\lambda>0$.
System \eqref{eq:pb} is complemented with homogeneous Dirichlet boundary conditions on the velocity:
\begin{equation} \label{eq:pb_CL}
\bfu|_{\dv\Omega}=0,
\end{equation}
and the following average density constraint (up to the normalization by $|\Omega|$ it is the same as prescribing the total mass)
\begin{equation}\label{hyp:mass}
\dfrac{1}{|\Omega|}\int_{\Omega}{\rho} \dx = \rho^\star > 0.
\end{equation}

\medskip
Regarding the theoretical results on these equations, the existence of weak solutions has been first proved by Lions in \cite{lio-98-comp} for adiabatic exponents $\gamma > \frac{5}{3}$ in dimension $d=3$, a result which has then been extended to coefficients $\gamma \in (\frac{3}{2},\frac{5}{3})$ by Novo and Novotn{\' y} in \cite{novo2002}.
%
%
\medskip
From the numerical viewpoint, compressible fluid equations have been intensively studied and several approximations have been designed in the last few years. In this paper, we consider a stabilized version of a numerical scheme implemented in the industrial software CALIF$^3$S \cite{califs} developed by the French \emph{Institut de Radioprotection et de S\^uret\'e Nucl\'eaire} (IRSN, a research center devoted to nuclear safety). This scheme falls in the class of \emph{staggered} discretizations in the sense that the scalar variables (density, pressure) are associated with the cells of a primal mesh $\mesh$ while the vectorial variables (velocity, external force) are associated with the set $\edges$ of faces of the primal mesh. Such decoupling, associated here with a Crouzeix-Raviart finite element discretization \cite{cro-73-con} (but other non-conforming finite elements are possible, such as the Rannacher-Turek discretization \cite{ran-92-sim}) of the viscous stress tensor, provides a discrete 
pressure estimate, thanks to the so-called discrete \emph{inf-sup} stability condition (see for instance \cite{girault2012}). This condition, which is also satisfied by the MAC scheme (see \cite{har-65-num}, \cite{har-68-num}, \cite{har-71-num}) on structured grids, ensures the unconditional stability of the scheme in almost incompressible regimes (for instance in the low Mach regime, see \cite{gal-08-unc} and \cite{herbin2018}).
Let us mention that, contrary to the MAC scheme (where the domain $\Omega$ is assumed to be a finite union of orthogonal parallelepipeds, and the mesh is composed by a structured partition of rectangular parallelepipeds with cell faces normal to the coordinate axis), the scheme considered in this paper is able to cope with unstructured meshes. 

\medskip
In its reduced form, our numerical scheme reads
\begin{align*}
 & 
	\dive_\mesh(\rho \bfu)  + T_{\rm stab}^1 + T_{\rm stab}^2 = 0, \\[2ex]  
&
	\divv_\edges(\rho \bfu\otimes \bfu) - \mu \lapi_\edges \bfu - (\mu+\lambda) (\gradi \circ \dive)_\edges \bfu + a\gradi_\edges (\rho^\gamma) + T_{\rm stab}^3   = \widetilde{\Pi}_\edges\bff,
\end{align*}
where, as suggested by the notations used for the discrete differential operators, the (scalar) mass equation is discretized on the primal mesh $\mesh$, whereas the (vectorial) momentum equation is discretized on a dual mesh associated with the set of faces $\edges$. 

\medskip
The finite element discretization for the viscous stress tensor is here coupled with finite volume discretizations of the convective terms which allow, thanks to standard techniques, to obtain discrete convection operators satisfying maximum principles (\eg\ \cite{lar-91-how}).
The discrete mass convection operator is a standard finite volume operator defined on the cells of the primal mesh $\mesh$ while the discrete momentum convection operator is also a finite volume operator written on {\em dual cells}, \ie\ cells centered at the location of the velocity unknowns, namely the faces $\edges$.
A difficulty implied by such staggered discretization lies in the fact that, as in the continuous case, the derivation of the energy inequality needs that a mass balance equation be satisfied on the same (dual) cells, while the mass balance in the scheme is naturally written on the primal cells. A procedure has therefore been developed to define the density on the dual mesh cells and the mass fluxes through the dual faces from the primal cell density and the primal faces mass fluxes, which ensures a discrete mass balance on dual cells.

\medskip
Compared to the continuous problem \eqref{eq:pb}, the discrete equations contain three additional ``stabilization'' terms $T_{\rm stab}^i$ that ensure the convergence (up to extracting a subsequence) of the numerical solutions towards weak solutions of \eqref{eq:pb}-\eqref{eq:pb_CL}-\eqref{hyp:mass} as the mesh size tends to $0$.
The first stabilization term $T_{\rm stab}^1$ guarantees the total mass constraint \eqref{hyp:mass} at the discrete level. 
The second stabilization term $T_{\rm stab}^2$, which is a discrete counterpart of a diffusion term for the density, provides an additional (mesh dependent) estimate on the discrete gradient of the density.
As we will explain in details in the core of the paper, this artificial discrete diffusion is used to show the crucial convergence property satisfied by the effective viscous flux.
The last stabilization term $T_{\rm stab}^3$ is an artificial pressure gradient which is necessary only if $\gamma \leq 3$.
The precise definitions of the discrete operators and stabilization terms are given in Section \ref{sec:discrete}.

\medskip
There exist in the literature several recent convergence results for finite element or mixed finite volume - finite element schemes. 
In \cite{eym-10-conv-isen}, Eymard \emph{et al.} (see also \cite{gal-09-conv} for the particular case $\gamma=1$, {\it i.e.} a linear pressure term) study the compressible Stokes equations, that correspond to \eqref{eq:pb} where the nonlinear convective term $\dive(\rho \bfu \otimes \bfu)$ is neglected. 
At the discrete level, two stabilization terms, namely, $T_{\rm stab}^1$ and a term similar to $T_{\rm stab}^2$, are introduced for the convergence analysis of the numerical scheme. 
In the case of Equations \eqref{eq:pb}-\eqref{eq:pb_CL}-\eqref{hyp:mass}, {\it i.e.} with the additional convective term, Gallou\"et \emph{et al.} prove in the recent paper \cite{gal-18-conv} the convergence of the MAC scheme under the condition $\gamma > 3$, with only one stabilization term $T_{\rm stab}^1$ ensuring the mass constraint \eqref{hyp:mass}
(we refer to Remark \ref{rmrk:MAC} below which explains why $T_{\rm stab}^2$ is unnecessary for the MAC scheme). 
Finally, Karper proved in \cite{karper2013} (see also the recent book \cite{feireisl2016}) a convergence result in the evolution case, again for $\gamma > 3$, and an equivalent of the artificial diffusion term $T_{\rm stab}^2$ is also introduced (note that in the evolutionary case there is no additional mass constraint and thus no need for $T_{\rm stab}^1$). Let us mention that for the evolutionary case, error estimates are available in \cite{gallouet2015} for the whole range $\gamma > \frac{3}{2}$, and that convergence results have been obtained in \cite{feireisl2018} for $1 < \gamma < 2$ within the framework of dissipative measure-valued solutions, a ``weaker'' framework than ours.

\medskip
To the best of our knowledge, our result is the first convergence result in the three-dimensional case for values $\gamma \in (\frac{3}{2},3]$ within the framework of weak solutions with finite energy (see Definition \ref{df:c-weak-sol}). 
It provides an alternative proof of the existence result obtained by Lions or by Novo and Novotn\'y. 
Compared to the previous numerical studies dealing with coefficients $\gamma > 3$, it requires the introduction of a third stabilization term $T_{\rm stab}^3$, an artificial pressure term weighted by some power of the mesh size: $h^\xi \gradi_\edges( \rho^\Gamma)$ with $\Gamma > 3$. Note that the stabilization terms $T_{\rm stab}^2$ and $T_{\rm stab}^3$ are not implemented in practice and are introduced here for the convergence analysis.

\medskip
Let us emphasize that the evolution case is beyond the scope of this paper and left for future work.

\medskip
The paper is organized as follows: in Section \ref{sec:continuous}, we present the main ingredients for the analysis of the continuous problem \eqref{eq:pb}-\eqref{eq:pb_CL}-\eqref{hyp:mass}. This section does not present any substantial novelty compared to the work of Novo and Novotn\`y~\cite{novo2002}, and the reader already familiar with the analysis of the compressible Navier-Stokes equations can directly pass to the next sections concerning the discrete problem. Then, in Section \ref{sec:discrete}, we introduce our numerical scheme and state precisely our main convergence result. 
We derive in Section \ref{sec:estimates} mesh independent estimates and show the existence of solutions to the numerical scheme.
Finally, Section \ref{sec:pass-limit} is devoted to the proof of convergence of the numerical method as the mesh size tends to $0$.
We provide in the Appendix additional material and proofs.


\section{The continuous setting}
{\label{sec:continuous}}

The aim of this section is to present the main ingredients involved in the analysis of the continuous problem \eqref{eq:pb}-\eqref{eq:pb_CL}-\eqref{hyp:mass} for readers who are not familiar with the compressible Navier-Stokes equations.
Although the existence theory of weak solutions to these equations is now well understood since the works of Lions \cite{lio-98-comp} and Feireisl \cite{feireisl2001} (see also \cite{novotny2004}), the analysis developed there involves advanced tools (such as weak compactness methods based on energy estimates, renormalized solutions, effective viscous flux, etc.) that are to our opinion worth recalling.
Especially as these tools will be also crucial in the convergence analysis of our numerical scheme.

\medskip
It turns out that the estimates and compactness arguments differ significantly according to the value of the adiabatic exponent $\gamma$ appearing in the pressure law.
For $d=3$ and $\gamma > 3$, the case treated in previous numerical studies, a sketch of the proof of the stability of weak solutions can be found for instance in \cite{gal-18-conv}.
We focus here, as in the other sections, on the case $d=3$ and $\gamma \in (\frac 32,3]$ which is the case covered by the study of Novo and Novotn{\' y}.\\
Essentially, the minimal value $\gamma^* = 3$ is the one that ensures a control of the pressure $\rho^\gamma$ and of the convective term $\rho \bfu \otimes \bfu$ in $\xL^2(\Omega)$.
The value $\gamma^* = \frac{5}{3}$ exhibited by Lions corresponds to the minimal exponent guaranteeing that $\rho$ is controlled in $\xL^2(\Omega)$. 
As we will explain later on (see Remark \ref{rmk:renorm} below), this control is required to prove that weak solutions are renormalized solutions.
This constraint on $\gamma$ has been relaxed by Novo, Novotn{\' y} \cite{novo2002} (and Feireisl \cite{feireisl2001} in the evolutionary case) to reach $\gamma > \frac 32$ which corresponds to the minimal exponent ensuring that $\rho \bfu \otimes \bfu$ is controlled in $\xL^p(\Omega)$, with $p > 1$.


\medskip
This section is organized as follows: in the first subsection we recall the classical definition of weak solutions to problem \eqref{eq:pb}-\eqref{eq:pb_CL}-\eqref{hyp:mass} and state a stability result on these solutions. 
We present in the second subsection the main arguments for the key step of the proof of the stability, that is the proof of the strong convergence of the density.
In the last subsection we explain how to approximate \eqref{eq:pb} in order to construct effectively weak solutions.

\subsection{Definition of weak solutions, stability result}

\begin{defi}{\label{df:c-weak-sol}}
	Let $\Omega$ be a Lipschitz bounded domain of $\xR^3$. Let $\gamma>\frac 32$. Let $\bff\in\xbfL^2(\Omega)$ and $\rho^\star>0$.
	A pair $(\rho,\bfu) \in \xL^{3(\gamma-1)}(\Omega) \times \xbfH^1_0(\Omega) $ is said to be a weak solution to Problem~\eqref{eq:pb}-\eqref{eq:pb_CL}-\eqref{hyp:mass} if it satisfies:
	
	\medskip
	\noindent
	Positivity of the density and global mass constraint:
	    \begin{equation}
	      \rho \geq 0 ~ \text{a.e. in } \ \Omega \quad \text{and} \qquad \frac{1}{|\Omega|}\int_\Omega \rho \dx = \rho^\star.
	    \end{equation}
		Equations~\eqref{eq:pb_mass}--\eqref{eq:pb_mom} are satisfied in the weak sense:
	\begin{equation}\label{eq:c-weak-mass}
	\int_{\Omega}{\rho \bfu \cdot \gradi \phi \dx} = 0 \qquad \forall \phi \in \mcal{C}_c^{\infty}(\Omega),
	\end{equation} 
	\begin{multline}\label{eq:c-weak-momentum}
	 -\int_{\Omega}{\rho \bfu \otimes \bfu : \gradi \bfv \dx} 
	  -a\int_{\Omega}{\rho^\gamma \ \dive\,\bfv \dx}
	  + \mu \int_{\Omega}{\gradi \bfu : \gradi\bfv \dx} \\
	 + (\lambda + \mu)\int_{\Omega}{\dive\,\bfu \ \dive\,\bfv \dx}
	 = \int_\Omega{\bff \cdot \bfv \dx},
	  \qquad \forall \bfv \in \mcal{C}_c^{\infty}(\Omega)^3.
	\end{multline}
	The pair $(\rho,\bfu) \in  L^{3(\gamma-1)}(\Omega) \times \xbfH^1_0(\Omega) $ is said to be a weak solution with bounded energy if, in addition to the previous conditions, it satisfies the energy inequality
	\begin{equation}\label{eq:c-energy}
	\mu \int_\Omega{|\gradi \bfu|^2 \dx} + (\lambda+\mu)\int_{\Omega}{(\dive\,\bfu)^2 \dx}
	\leq \int_\Omega \bff\cdot \bfu \dx.
	\end{equation}
	Finally the pair $(\rho,\bfu) \in  L^{3(\gamma-1)}(\Omega) \times \xbfH^1_0(\Omega) $ is said to be a weak renormalized solution if, in addition to the previous conditions and for any $b \in \mathcal{C}^0([0,+\infty)) \cap \mathcal{C}^1((0,+\infty))$ such that 
	\begin{equation}\label{eq:cond-renorm}
	|b'(t)| \leq 
	\begin{cases} 
	\ c t^{-\lambda_0}, ~\lambda_0 < 1 \quad &\text{if}~ t < 1,\\
	\ c t^{\lambda_1}, ~\lambda_1 + 1 \leq \frac{3(\gamma-1)}{2} \quad &\text{if}~ t \geq 1,
	\end{cases}
	\end{equation}
	the pair $(\rho,\bfu)$ satisfies
	\begin{equation}\label{eq:c-renormalized-eq}
	\dive(b(\rho) \bfu) + \big(b'(\rho)\rho - b(\rho)\big) \dive\,\bfu = 0 \quad \text{in}~\mathcal{D}'(\mathbb{R}^3),
	\end{equation}
	where $\rho$ and $\bfu$ have been extended by $0$ outside $\Omega$.
\end{defi}

\medskip
\begin{rmrk}
 In the whole paper, we adopt the following notations:
 \[
  \xbfH^1_0(\Omega) := \xH^1_0(\Omega)^d, \qquad \xbfW^{1,p}_0(\Omega) := \xW^{1,p}_0(\Omega)^d, \qquad \xbfL^p(\Omega):=\xL^p(\Omega)^d, \ p\in[1,+\infty].
 \]
\end{rmrk}

\medskip
\begin{rmrk}
 Since $\gamma > \frac{3}{2}$, we have $\rho\bfu\in\xbfL^{\frac 65}(\Omega)$ and by density \eqref{eq:c-weak-mass} is valid for all $\phi\in\xW^{1,6}_0(\Omega)$. In addition, $\rho\bfu\otimes\bfu\in\xbfL^{1+\eta}(\Omega)^3$ and $\rho^\gamma\in\xL^{1+\eta}(\Omega)$ for some $\eta > 0$ so that \eqref{eq:c-weak-momentum} is valid for all $\bfv\in\xbfW^{1,q}_0(\Omega)$ for all $q\in[1,+\infty)$.
\end{rmrk}

\medskip
\begin{rmrk}
For $d=3$, $\gamma > 3$, and $d=2$, $\gamma>1$, we would get better integrability on $\rho$. 
Precisely, we would have $\rho\in \xL^{2\gamma}(\Omega)$.
\end{rmrk}

\medskip

\begin{rmrk}{\label{rmk:renorm}} \ 
	\begin{itemize}
		\item 
	When $\gamma$ is large enough, namely $\gamma \geq \frac{5}{3}$, the following lemma, initially proved by Di~Perna and Lions \cite{diperna1989}, shows that any weak solution with finite energy of \eqref{eq:pb} is a renormalized weak solution.
	
	\begin{lem}[\cite{novotny2004}~Lemma 3.3]{\label{lem:diperna}}
	Assume that $\gamma \geq \frac{5}{3}$ and let $\rho \in \xL^{3(\gamma-1)}_{\rm loc}(\mathbb{R}^3)$, $\bfu \in \xbfH^1_{\rm loc}(\mathbb{R}^3)$ satisfying the continuity equation
	\[
	\dive(\rho \bfu) = 0 \quad \text{in}~~\mathcal{D}'(\mathbb{R}^3).
	\]
	Then, equation \eqref{eq:c-renormalized-eq} holds for any $b \in \mathcal{C}^0([0,+\infty)) \cap \mathcal{C}^1((0,+\infty))$ satisfying \eqref{eq:cond-renorm}.
	\end{lem}
	
	\medskip
	More precisely, the justification of the renormalized equation requires a preliminary regularization of the density. 
	The commutator term resulting from this regularization involves in particular products such as $\rho\, \dive\, \bfu$ which are then controlled precisely under the condition that $\rho \in \xL^2_{\rm loc}(\mathbb{R}^3)$ since $\dive \, \bfu \in \xL^2_{\rm loc}(\mathbb{R}^3)$ (see for instance \cite{novotny2004} Lemma 3.1).
	This condition is achieved as soon as $3(\gamma-1) \geq 2$, i.e. $\gamma \geq \frac{5}{3}$.
	The interested reader is also referred to \cite{fettah2013} (Appendix B) for a discussion on the criticality of the assumption on $\rho$.
	
	\item If the pair $(\rho,\bfu) \in \xL^{3(\gamma-1)}(\Omega) \times \xbfH^1_0(\Omega)$ is a renormalized solution which satisfies, instead of the continuity equation~\eqref{eq:c-weak-mass},
	\[
	\dive(\rho \bfu) = g ~~\text{in} ~\mathcal{D}'(\Omega) ~~\text{for some}~g \in \xL^{1}_{\rm loc}(\mathbb{R}^3), 
	\]
	then, extending $\rho$, $\bfu$ by zero outside $\Omega$ (denoting again $\rho,\bfu,g$ the extended functions), the previous equation also holds in $\mathcal{D}'(\xR^3)$.
	Moreover, for any $b \in \mathcal{C}^1([0,+\infty))$ satisfying~\eqref{eq:cond-renorm}, denoting $b_M$ the truncated function such that
	\[
	b_M(t) = \begin{cases}
			\ b(t) \quad & \text{if} ~ t < M, \\
			\ b(M) \quad & \text{if} ~ t \geq M,
	\end{cases}
	\]
	then we have
	\begin{equation}
	\label{eq:c-renormalized-eq-trunc}
	 \dive(b_M(\rho) \bfu) + \big([b_M]'_+(\rho)\rho - b_M(\rho)\big) \dive \, \bfu = g\, [b_M]'_+(\rho)
	\quad \text{in}~\mathcal{D}'(\xR^3)
	\end{equation}
    where
		\[
		[b_M]'_+(t) = \begin{cases}
		\ b'(t) \quad & \text{if} ~ t < M, \\
		\ 0 \quad & \text{if} ~ t \geq M.
		\end{cases}
		\]	
	\end{itemize}
\end{rmrk}

\medskip
We now focus on the stability of weak solutions the proof of which is essential for the analysis of the numerical scheme in the next sections. In Section \ref{sec:elements:construct}, some elements are given for the approximation procedure that allows to construct such weak solutions.

%

\medskip

\begin{thrm}\label{thm:stab}
	Let $\Omega$ be a Lipschitz bounded domain of $\xR^3$. Assume that $\gamma \in (\frac 32,3]$.
	Consider sequences of external forces $(\bff\n)_{n\in\xN}\subset\xbfL^2(\Omega)$ and masses $(\rho^\star\n)_{n\in\xN}\subset\xR_+^*$, and an associated sequence $(\rho\n,\bfu\n)_{n\in\xN}$ of renormalized weak solutions with bounded energy.
	Assume that $\rho^\star\n \rightarrow \rho^\star>0$ and that $(\bff\n)_{n\in\xN}$ converges strongly in $\xbfL^2(\Omega)$ to $\bff$.
	Then, there exist $(\rho,\bfu)\in\xL^{3(\gamma-1)}(\Omega)\times \xbfH^{1}_0(\Omega)$ and a subsequence of $(\rho\n,\bfu\n)_{n\in\xN}$, still denoted $(\rho\n,\bfu\n)_{n\in\xN}$ such that:
\begin{itemize}
\item The sequence $(\bfu\n)_{n\in\xN}$ converges to $\bfu$ in $\xbfL^q(\Omega)$ for all $q\in[1,6)$,
\item The sequence $(\rho\n)_{n\in\xN}$ converges to $\rho$ in $\xL^q(\Omega)$ for all $q\in[1,3(\gamma-1))$ and weakly in $\xL^{3(\gamma-1)}(\Omega)$,
\item The sequence $(\rho\n^\gamma)_{n\in\xN}$ converges to $\rho^\gamma$ in $\xL^q(\Omega)$ for all $q\in[1,\frac{3(\gamma-1)}{\gamma})$ and weakly in $\xL^{\frac{3(\gamma-1)}{\gamma}}(\Omega)$,
\item The pair $(\rho,\bfu)$ is a weak solution of Problem \eqref{eq:pb}-\eqref{eq:pb_CL}-\eqref{hyp:mass} with finite energy.
\end{itemize}
\end{thrm}


\medskip
For the sake of brevity, we shall not present the complete proof of Theorem \ref{thm:stab}. 
Let us recall the methodology: first we derive the basic uniform estimates which enable us in the next step to derive compactness results on the sequence  $(\rho\n,\bfu\n)_{n\in\xN}$ and to pass to the limit in the mass and momentum equations as $n\to+\infty$. 
These steps do not present any remarkable difficulty and are just sketched below. 
They will be performed in details in the next sections concerning the discrete case. 
We prefer to focus here on the final step which consists in proving the strong convergence of the density, convergence which is essential to pass to the limit in 
the equation of state ({\it i.e.} in the non linear pressure law). 

\medskip
The basic uniform estimates are recalled in the next proposition.
\begin{prop}[Uniform controls]
	\label{unif:est:cont}
	Let $\Omega$ be a Lipschitz bounded domain of $\xR^3$. Assume that $\gamma \in (\frac 32,3]$. Let $(\rho\n,\bfu\n)_{n\in\xN}$ be the sequence defined in Theorem \ref{thm:stab}. Then, we have the following \emph{a priori} controls:
	\begin{equation}\label{eq:c-control}
	\norm{\bfu\n}_{\xbfH^1_0(\Omega)} + \norm{\rho\n^\gamma}_{\xL^{1+\eta}(\Omega)} + \norm{\rho\n}_{\xL^{3(\gamma-1)}(\Omega)}  \leq C(\Omega,(\norm{\bff_n}_{\xbfL^2(\Omega)})_{n\in\xN}), \qquad \forall n\in\xN,
	\end{equation}
	where $\eta = \frac{2\gamma -3}{\gamma} >0$.
\end{prop}
The control of the velocity easily follows from the energy inequality \eqref{eq:c-energy}, while the control of the pressure (and the density) derives from an estimate linked to the so-called Bogovskii operator $\mathcal{B}$ defined in the next Lemma.

%
%
%
%

\medskip
\begin{lem}{\label{lem:bogovskii}}
Let $\Omega$ be a bounded Lipschitz domain of $\xR^d$, $d \geq 1$.
Then, there exists a linear operator $\mcal{B}$ depending only on $\Omega$ with the following properties:
\begin{itemize}
 \item[(i)] For all $q\in(1,+\infty)$,
 \[
  \mcal{B}: \xL^q_0(\Omega) = \big \lbrace p\in\xL^q(\Omega), \ \text{s.t.} \ <p>=\frac{1}{|\Omega|}\int_{\Omega}{p \dx}=0\big \rbrace \rightarrow \xbfW^{1,q}_0(\Omega).
 \]
 \item[(ii)] For all $q\in(1,+\infty)$ and $p\in \xL^q_0(\Omega)$,
 \[
    \dive (\mcal{B} p ) = p, \ \text{a.e. in $\Omega$}.
 \]
 \item[(iii)] For all $q \in (1,+\infty)$, there exists $C=C(q,\Omega)$, such that for any $p \in \xL^q_0(\Omega)$:
 \[
  \snorm{\mcal{B} p}_{\xbfW^{1,q}(\Omega)} \leq C \, \norm{p}_{\xL^q(\Omega)}.
 \]
\end{itemize}
\end{lem}
The interested reader is referred to~\cite{novotny2004} (Chapter 3.3) for a proof and additional properties on this operator. 
Applying Lemma~\ref{lem:bogovskii} to $P_n=a\rho\n^\gamma$ and using the resulting field $\bfv_n= \mcal{B} (P_n-<P_n>)$ as a test function in~\eqref{eq:c-weak-momentum}, one gets the desired control on the pressure and thus on the density.

\medskip
Thanks to the previous estimates we deduce that up to the extraction of a subsequence, the sequence $(\rho\n, \bfu\n, \rho\n^\gamma)$ weakly converges to a triple $(\rho,\bfu,\overline{\rho^\gamma})\in\xL^{3(\gamma-1)}(\Omega)\times\xbfH^1_0(\Omega)\times\xL^{ \frac{3(\gamma-1)}{\gamma}}(\Omega)$.
In particular, the limit $(\rho,\bfu, \overline{\rho^\gamma})$ satisfies the momentum equation in the following sense: 
\begin{multline}\label{eq:c-weak-momentum-lim}
 -\int_{\Omega}{\rho \bfu \otimes \bfu : \gradi \bfv \dx} 
- a\int_{\Omega}{\overline{\rho^\gamma} \, \dive\,\bfv \dx}
+ \mu \int_{\Omega}{\gradi \bfu : \gradi\bfv \dx} \\
 \quad + (\lambda + \mu)\int_{\Omega}{\dive\,\bfu \ \dive\,\bfv \dx}
= \int_\Omega{\bff \cdot \bfv \dx},
\qquad \forall \bfv \in \mcal{C}_c^{\infty}(\Omega)^3.
\end{multline}

\medskip
To complete the proof of Theorem \ref{thm:stab}, it remains to identify the limit pressure $\overline{\rho^\gamma}$ in~\eqref{eq:c-weak-momentum-lim}. 
Namely we have to pass to the limit in the equation of state and prove that
\begin{equation}
\overline{\rho^\gamma} = \rho^\gamma \quad \text{a.e. in $\Omega$}
\end{equation}
which is equivalent to proving the strong convergence of the density towards its weak limit.

\subsection{Passing to the limit in the equation of state}

This is classically obtained in two steps: first by proving some weak compactness property satisfied by the so-called effective viscous flux defined as $(2\mu+\lambda)\dive \, \bfu - a \rho^\gamma$, and then, by using the monotonicity of the pressure to deduce the strong convergence of the sequence of densities $(\rho\n)_{n\in\xN}$.

\subsubsection{Weak compactness of the effective viscous flux}\label{sec:c-effectiveflux}

Let us first recall the definition of the $\rot$ operator, and a useful identity linked to this operator.
\begin{lem}[A differential identity]
Let $\Omega$ be a Lipschitz bounded domain of $\xR^3$.
For $\bfa=(a_1,a_2,a_3)^T$ and $\bfb=(b_1,b_2,b_3)^T$ in $\xR^3$ we denote $\bfa\wedge\bfb = ( a_2 b_3-a_3 b_2, a_3 b_1-a_1 b_3, a_1 b_2-a_2 b_1  )^T\in\xR^3$. For a vector valued function $\bfv=(v_1,v_2,v_3)^T$, denote $\rot\, \bfv = \gradi \wedge \bfv$ where $\gradi = (\dv_1,\dv_2,\dv_3)^T$. With these notations, if $\bfu\in\xbfH^1(\Omega)$ and $\bfv\in\xbfH^1(\Omega)$, the following identity holds:
\begin{multline}
\label{curlcurl1}
\int_\Omega \gradi \bfu : \gradi \bfv \dx =  \int_\Omega \dive \, \bfu \, \dive\, \bfv \,\dx + \int_\Omega \rot \, \bfu \cdot \rot \, \bfv \,\dx 
\\ + \int_{\dv\Omega} (\gradi \bfv . \bfn)\cdot\bfu\,\dedge(\bfx) + \int_{\dv\Omega} \rot\,\bfv \cdot (\bfu \wedge \bfn )\,\dedge(\bfx) - \int_{\dv\Omega}\dive\,\bfv \, (\bfu\cdot\bfn)\,\dedge(\bfx).
\end{multline}
If $\bfu\in\xbfH^1_0(\Omega)$ and $\bfv\in\xbfH^1(\Omega)$, this identity simplifies to:
\begin{equation}
\label{curlcurl2}
\int_\Omega \gradi \bfu : \gradi \bfv \dx=  \int_\Omega \dive \, \bfu \, \dive\, \bfv \,\dx + \int_\Omega \rot \, \bfu \cdot \rot \, \bfv \,\dx.
\end{equation}
\end{lem}
\medskip
We shall also need the next result.
\begin{lem}{\label{lem:inv-div}}
	Let $\Omega$ be a bounded open set of $\xR^d$.
	Then, there exists a linear operator $\mcal{A}$ with the following properties:
\begin{itemize}
 \item[(i)] For all $q\in(1,+\infty)$,
 \[
  \mcal{A}: \xL^q(\Omega) \rightarrow \xbfW^{1,q}(\Omega).
 \]
 \item[(ii)] For all $q\in(1,+\infty)$ and $p\in \xL^q(\Omega)$,
 \[
    \dive (\mcal{A} p ) = p, \ \text{and} \ \ \rot (\mcal{A} p ) = 0, \ \text{a.e. in $\Omega$}.
 \]
 \item[(iii)] For all $q \in (1,+\infty)$, there exists $C=C(q,\Omega)$, such that for any $p \in \xL^q(\Omega)$:
 \[
  \snorm{\mcal{A} p}_{\xbfW^{1,q}(\Omega)} \leq C \, \norm{p}_{\xL^q(\Omega)}.
 \]
\end{itemize}
%
%
\end{lem}
\medskip
\begin{proof}
A solution is given by $\mcal{A}p:=\gradi \Delta^{-1} (p)$, where $\Delta^{-1}$ is defined as the inverse of the Laplacian on $\xR^3$, here applied to $p$ extended by $0$ outside $\Omega$. The reader is referred to~\cite{novotny2004} Section 4.4.1 for properties of the operator $\mathcal{A}$.
In particular the operator $\mcal A$ does not depend on $q$.
\end{proof}

\medskip
\begin{prop}{\label{prop:c-eff-flux}} 
Let $\Omega$ be a Lipschitz bounded domain of $\xR^3$. Assume that $\gamma \in (\frac 32,3]$.
Let $(\rho\n,\bfu\n)_{n\in\xN}$ be the sequence defined in Theorem \ref{thm:stab}.
For $k\in\xN^*$, define
\begin{equation}\label{df:Tk}
T_k(t)= \begin{cases}
		\ t \quad & \text{if} ~ t \in [0,k), \\
		\ k \quad & \text{if} ~ t \in [k,+\infty). 
\end{cases}
\end{equation}
The sequence $(T_k(\rho\n))_{n\in\xN}$ is bounded in $\xL^{\infty}(\Omega)$ and up to extracting a subsequence, it converges for the weak-* topology in $\xL^\infty(\Omega)$ towards some function denoted $\overline{T_k(\rho)}$.
Then, up to extracting a subsequence, the following identity holds:
\begin{equation}\label{eq:c-eff-flux}
\lim_{n\rightarrow +\infty} \int_{\Omega}{\big((2\mu +\lambda)\ \dive\, \bfu\n - a\rho\n^\gamma \big) T_k(\rho\n) \phi \dx}
= \int_{\Omega}{\big((2\mu +\lambda)\, \dive \, \bfu - a\overline{\rho^\gamma}\big) \overline{T_k(\rho)} \phi\dx}, \quad \forall \phi \in \mcal{C}_c^{\infty}(\Omega). 
\end{equation}
\end{prop}



\medskip
\begin{proof}
Let $k\in\xN^*$. For $n\in\xN$, let $\bfw_n=\mcal{A}T_k(\rho\n)$ be the field associated with $T_k(\rho\n)$ through Lemma~\ref{lem:inv-div}. 
We have
\[
\dive \ \bfw_n = T_k(\rho\n), \quad \rot \ \bfw_n = 0, \quad (\bfw_n)_{n\in\xN} \text{ \ is bounded in \ }\xbfW^{1,q}(\Omega) \quad \forall q \in (1,+\infty).
\]
Moreover, $(\bfw_n)_{n\in\xN}$ is bounded in $\xbfL^{\infty}(\Omega)$ and 
up to extracting a subsequence, as $n\to+\infty$, it strongly converges in $\xbfL^q(\Omega)$ and weakly in $\xbfW^{1,q}(\Omega)$ for all $q \in(1,+\infty)$ towards some function $\bfw$ satisfying:
\begin{equation}
\dive \ \bfw = \overline{T_k(\rho)} \quad \text{and} \quad \rot \ \bfw = 0.
\end{equation}
Let $\phi \in \mathcal{C}^\infty_c(\Omega)$.
Considering in~\eqref{eq:c-weak-momentum} the test function $\bfv_n = \phi \bfw_n $ we get:
\begin{multline*}
	 -\int_{\Omega}{\rho\n \bfu\n \otimes \bfu\n : \gradi (\phi \bfw\n) \dx} 
	  -a\int_{\Omega}{\rho\n^\gamma \ \dive(\phi\bfw\n) \dx}
	  + \mu \int_{\Omega}{\gradi \bfu\n : \gradi(\phi\bfw\n) \dx} \\
	 + (\lambda + \mu)\int_{\Omega}{\dive\,\bfu\n \ \dive(\phi\bfw\n) \dx}
	 = \int_\Omega{\bff \cdot (\phi\bfw\n) \dx}.
\end{multline*}
Using the formula~\eqref{curlcurl2} and the fact that $\dive(\phi\bfw\n)=T_k(\rho\n)\phi+\bfw\n\cdot\gradi\phi$ and $\rot(\phi\bfw\n)=L(\phi)\bfw\n$ where $L(\phi)$ is a matrix involving first order derivatives of $\phi$,
we obtain:
\begin{align*}
& \int_{\Omega}{\big(a\rho\n^\gamma - (2\mu + \lambda) \dive \, \bfu_{n} \big) T_k(\rho\n) \phi \dx} \\
& = -\int_{\Omega}{\big(a\rho\n^\gamma - (2\mu + \lambda)\dive \, \bfu_{n} \big ) \bfw_n \cdot \gradi \phi \dx} 
+ \mu \int_{\Omega}{\rot \, \bfu\n \cdot L(\phi) \bfw_n \dx} \\
& \quad - \int_{\Omega}{\rho\n \bfu\n \otimes \bfu\n : \gradi(\phi \bfw_n) \dx}
- \int_{\Omega}{\bff_n \cdot (\phi \bfw_n)\dx}        
\end{align*}
Thanks to the previous estimates and convergences 
, we are allowed to pass to the limit as $n \to +\infty$
\begin{align*}
& \lim_{n \rightarrow + \infty}\int_{\Omega}{\big(a\rho\n^\gamma - (2\mu + \lambda)\dive \, \bfu_{n} \big)T_k(\rho\n) \phi \dx} \\
& = - \int_{\Omega}{\big(a\overline{\rho^\gamma} - (2\mu + \lambda)\dive\,\bfu \big) \bfw \cdot \gradi \phi \dx} 
+ \mu \int_{\Omega}{\rot\,\bfu \cdot L(\phi) \bfw \dx} \\
& \quad - \lim_{n \rightarrow + \infty} \int_{\Omega}{\rho\n \bfu\n \otimes \bfu\n : \gradi(\phi \bfw\n) \dx}
- \int_{\Omega}{\bff \cdot (\phi \bfw) \dx}  .     
\end{align*}

An analogous equation can be obtained from the limit momentum equation~\eqref{eq:c-weak-momentum-lim} with the test function $\phi \bfw$.
It reads:
\begin{align*}
& \int_{\Omega}{\big(a\overline{\rho^\gamma} - (2\mu + \lambda)\dive\, \bfu \big)\ \overline{T_k(\rho)} \phi \dx } \\
& = - \int_{\Omega}{\big(a\overline{\rho^\gamma} - (2\mu + \lambda)\dive\,\bfu \big) \bfw \cdot \nabla \phi \dx} 
+ \mu \int_{\Omega}{\rot(\bfu) \cdot L(\phi) \bfw \dx} \\
& \quad - \int_{\Omega}{\rho \bfu \otimes \bfu : \gradi(\phi \bfw)\dx}
- \int_{\Omega}{\bff \cdot (\phi \bfw)\dx} .      
\end{align*}
Comparing the two expressions, we get
\begin{align*}
& \lim_{n \rightarrow + \infty}\int_{\Omega}{ \big(a\rho\n^\gamma - (2\mu + \lambda)\dive\, \bfu_{n} \big) T_k(\rho\n)\phi  \dx} \\
& \quad  =  \int_{\Omega}{ \big(a\overline{\rho^\gamma} - (2\mu + \lambda)\dive \,\bfu  \big)  \overline{T_k(\rho)} \phi \dx} \\
& \qquad - \lim_{n \rightarrow + \infty} \int_{\Omega}{\rho\n \bfu\n \otimes \bfu\n : \gradi(\phi \bfw\n) \dx} 
		+ \int_{\Omega}{\rho \bfu \otimes \bfu : \gradi(\phi \bfw)\dx}.
\end{align*}
Hence it remains to show the two last integrals are equal, which is not direct
since we have only weak convergence on $(\rho\n \bfu_n)_{n\in\xN}$ and $(\gradi \bfw\n)_{n\in\xN}$.

\medskip
This convective term is usually treated with compensated compactness tools by means of Div-Curl and commutator lemmas (see \cite{novotny2004} Section 4.4).
In the case $\gamma > 3$, a simpler proof is presented in \cite{gal-18-conv} which enables to bypass the use of these tools.
We adapt this simple proof to our case $\gamma \in (\frac{3}{2},3]$ by using a regularization of the velocity $\bfu_n$. 
Let us first extend $\bfu\n$ and $\bfu$ by $0$ on $\xR^3\setminus \Omega$ and introduce the regularized velocities $\bfu_{n,\delta} = \bfu_n * \bfomega_\delta$ and $\bfu_\delta = \bfu * \bfomega_\delta$, where $(\bfomega_\delta)_{\delta>0}$ is a mollifying sequence. 
By standard properties of the convolution and our \emph{a priori} control of the velocity $\bfu\n$, the following convergences hold (see for instance \cite{feireisl2016} Lemma 5 p.75 where a regularization of the velocity is also used)
\begin{align}
\bfu_{n,\delta} &\underset{n \to +\infty}{\longrightarrow} \bfu_\delta \quad \text{strongly in} ~ \xbfL^q_{\rm loc}(\xR^3) ~\forall q\in[1,6) ~\text{uniformly in} ~\delta, \label{eq:cvg-reg-u-1}\\
\bfu_{n,\delta} &\underset{\delta \to 0}{\longrightarrow} \bfu\n \quad \text{strongly in} ~ \xbfL^q_{\rm loc}(\xR^3) ~\forall q\in[1,6) ~(\text{uniformly in} ~n)  \label{eq:cvg-reg-u-2}, \\
\bfu_\delta &\underset{\delta \to 0}{\longrightarrow} \bfu \quad \text{strongly in} ~\xbfL^6_{\rm loc}(\xR^3)  \label{eq:cvg-reg-u-3}.
\end{align}
Since $\dive(\rho\n\bfu\n)=0$, we then have, for any $\delta>0$:
\begin{align*}
\int_{\Omega}{\rho\n \bfu\n \otimes \bfu\n : \gradi(\phi \bfw_n)\dx} 
& = \int_{\mathbb{R}^3}{ \bfu_{n,\delta} \otimes (\rho\n \bfu\n)  : \gradi(\phi \bfw_n)\dx}+ R_1^{n,\delta} \\
& = -\int_{\mathbb{R}^3}{\divv(\bfu_{n,\delta} \otimes \rho\n \bfu\n)  \cdot (\phi \bfw_n)\dx} +  R_1^{n,\delta} \\
& = -\int_{\mathbb{R}^3}{(\rho\n \bfu\n \cdot \gradi) \bfu_{n,\delta} \cdot (\phi \bfw_n)\dx} +  R_1^{n,\delta}
\end{align*}
where
\[
 R_1^{n,\delta}= \int_{\mathbb{R}^3}{(\bfu\n-\bfu_{n,\delta}) \otimes (\rho\n \bfu\n)  : \gradi(\phi \bfw_n)}\dx.
\]
Since $(\rho_n \bfu_n)_{n\in\xN}$ is bounded in $\xbfL^p(\Omega)$ for some $p > \frac 65$, $(\gradi \bfw_n)_{n\in\xN}$ is bounded in $\xbfL^s(\Omega)^3$ for any $s \in (1,+\infty)$, then  the following inequality holds, for some triple $(p,q,s)$, such that $p>\frac 65$, $s>1$, $q<6$ and $\frac{1}{p}+ \frac{1}{q} + \frac{1}{s}=1$:
\begin{align*}
|R_1^{n,\delta}| 
&\leq C \norm{\rho\n \bfu\n}_{\xbfL^{p}(\Omega)}  \norm{\gradi (\phi \bfw_n)}_{\xbfL^s(\Omega)^{3}} \norm{\bfu\n -\bfu_{n,\delta}}_{\xbfL^q(\xR^3)} \\
&\leq C \norm{\bfu\n -\bfu_{n,\delta}}_{\xbfL^q(\mathbb{R}^3)}.
\end{align*}
As a consequence:
\begin{equation}
\label{control:R1}
 \limsup\limits_{n\to+\infty} |R_1^{n,\delta}| \leq C \limsup\limits_{n\to+\infty} \norm{\bfu\n -\bfu_{n,\delta}}_{\xbfL^q(\mathbb{R}^3)}.
\end{equation}
Thanks to the regularization of the velocity, we ensure that $(\gradi \bfu_{n,\delta})_{n\in\xN}$ is bounded in $\xbfL^6_{\rm loc}(\mathbb{R}^3)^3$.
The sequence $\big((\rho_n \bfu_n \cdot \gradi) \bfu_{n,\delta})_{n\in\xN} = (\bfQ_{n,\delta})_{n\in\xN}$ is then bounded in $\xbfL^r(\Omega)$, for some $r > 1$ and up to the extraction of a subsequence, it weakly converges in $\xbfL^r(\Omega)$ towards some function $\bfQ_\delta \in \xbfL^r(\Omega)$. We have
\begin{align}
\label{reg1}
\int_{\Omega}{\rho\n \bfu\n \otimes \bfu\n : \gradi(\phi \bfw_n)\dx} 
& = -\int_{\Omega}{(\rho\n \bfu\n \cdot \gradi) \bfu_{n,\delta} \cdot (\phi \bfw_n)\dx} +  R_1^{n,\delta} \nonumber\\
& = -\int_{\Omega}{\bfQ_\delta \cdot (\phi \bfw)\dx} +  R_1^{n,\delta} +  R_2^{n,\delta}
\end{align}
where 
\begin{align*}
R_2^{n,\delta} 
& = -\int_{\Omega}{\bfQ_{n,\delta} \cdot (\phi \bfw_n)\dx} + \int_{\Omega}{\bfQ_\delta \cdot (\phi \bfw)\dx} \\
& = \int_{\Omega}{\big(\bfQ_\delta -\bfQ_{n,\delta}\big) \cdot (\phi \bfw)\dx} + \int_{\Omega}{\phi \, \bfQ_{n,\delta}\cdot\big(\bfw-\bfw_n\big) \dx}.
\end{align*}
Since $(\bfw_n)_{n\in\xN}$ converges strongly to $\bfw$ in $\xbfL^q(\Omega)$, for all $q \in(1, +\infty)$, we have
\begin{equation}
\label{control:R2}
|R_{2}^{n,\delta}|\to 0 \quad \text{ as $n\to+\infty$ for any fixed $\delta>0$}.
\end{equation}
We now want to show that 
$\bfQ_\delta = (\rho \bfu \cdot \gradi)\bfu_\delta$.
To that end, let us consider a fixed test function $\bfv \in \xbfW^{1,\infty}_0(\Omega)$ and write (again thanks to the fact that $\dive(\rho\n\bfu\n)=0$)
\begin{align*}
\int_{\Omega}{(\rho_n \bfu_n \cdot \gradi)\bfu_{n,\delta} \cdot \bfv \dx}
& = -\int_{\Omega}{ \bfu_{n,\delta} \otimes (\rho_n \bfu_n) : \gradi \bfv \dx} \\
& = -\int_{\Omega}{ \bfu_{\delta} \otimes (\rho_n \bfu_n) : \gradi \bfv \dx} + \tilde{R}_2^{n,\delta}\\
& = \int_{\Omega}{\divv(\bfu_{\delta} \otimes \rho_n \bfu_n) \cdot \bfv \dx} + \tilde{R}_2^{n,\delta}\\
& = \int_{\Omega}{(\rho \bfu \cdot \gradi)\bfu_\delta \cdot \bfv \dx} + \tilde{R}_2^{n,\delta}
\end{align*}
with 
\[
\tilde{R}_2^{n,\delta} = \int_{\Omega}{(\bfu_\delta - \bfu_{n,\delta}) \otimes (\rho_n \bfu_n ) : \gradi \bfv \dx} 
\]
which tends to $0$ (uniformly with respect to $\delta$) as $n \to +\infty$, since $(\rho_n \bfu_n)_{n\in\xN}$ converges weakly to $\rho \bfu$ in $\xbfL^{q_1}(\Omega)$ for some $q_1 > \frac 65$, and $(\bfu_{n,\delta})_{n\in\xN}$ converges strongly to $\bfu_\delta$ (uniformly with respect to $\delta$) in $\xbfL^{q_2}(\Omega)$ for any $q_2 < 6$.
As a consequence, we identify $\bfQ_\delta = (\rho \bfu \cdot \gradi) \bfu_\delta$.
Now, back to \eqref{reg1}, since at the limit $\dive(\rho\bfu)=0$, we have:
\begin{align*}
\int_{\Omega}{\rho\n \bfu\n \otimes \bfu\n : \gradi(\phi \bfw_n)\dx} 
& = -\int_{\Omega}{(\rho \bfu \cdot \gradi) \bfu_\delta \cdot (\phi \bfw)\dx} + R_1^{n,\delta} + R_2^{n,\delta} \\
& = \int_{\Omega}{\rho \bfu_\delta \otimes \bfu : \gradi(\phi \bfw)\dx} + R_1^{n,\delta} + R_2^{n,\delta} \\
& = \int_{\Omega}{\rho \bfu \otimes \bfu : \gradi(\phi \bfw)\dx} + R_1^{n,\delta} + R_2^{n,\delta} + R_3^\delta
\end{align*}
where
\[
 R_3^\delta = \int_{\Omega}{ (\bfu_\delta-\bfu) \otimes (\rho \bfu) : \gradi(\phi \bfw)\dx}.
\]
Combining \eqref{control:R1} and \eqref{control:R2}, we get that for any fixed $\delta>0$:
\[
\limsup_{n\rightarrow +\infty}  \Big | \int_{\Omega}{\rho\n \bfu\n \otimes \bfu\n : \gradi(\phi \bfw_n)\dx} - \int_{\Omega}{\rho \bfu \otimes \bfu : \gradi(\phi \bfw)\dx } \Big | 
 \leq C \limsup_{n\rightarrow +\infty} \norm{\bfu\n -\bfu_{n,\delta}}_{\xbfL^q(\xR^3)}  +|R_3^\delta|,
\]
for some $q<6$.
By \eqref{eq:cvg-reg-u-3}, we have $R_3^\delta \to 0$ as $\delta\to 0$. Hence, by the uniform in $n$ convergence of $(\bfu_{n,\delta})_{\delta >0}$ towards $\bfu\n$ as $\delta\to 0$ \eqref{eq:cvg-reg-u-2}, letting $\delta$ tend to $0$ yields:
\[
 \lim_{n\rightarrow +\infty} \int_{\Omega}{\rho\n \bfu\n \otimes \bfu\n : \gradi(\phi \bfw_n)\dx} =  \int_{\Omega}{\rho \bfu \otimes \bfu : \gradi(\phi \bfw)\dx}.
\]
We finally conclude that
\begin{equation*}
\lim_{n \rightarrow + \infty}\int_{\Omega}{ \big(a\rho\n^\gamma - (2\mu + \lambda)\dive\, \bfu_{n} \big) T_k(\rho_{n})\phi  \dx}
=  \int_{\Omega}{ \big(a\overline{\rho^\gamma} - (2\mu + \lambda)\dive \,\bfu  \big)  \overline{T_k(\rho)} \phi \dx} .
\end{equation*}
\end{proof}

\subsubsection{Strong convergence of the density}

Let us begin this subsection with a brief sketch of the general strategy that we will employ. 
Concentration phenomena being excluded, the only mechanism which can prevent the strong convergence is the presence of oscillations.
We need to prove that we control these oscillations.
For large values of $\gamma$, namely $\gamma \geq 2$, one can show an ``improved'' version of the weak compactness of the effective viscous flux \eqref{eq:c-eff-flux} where $T_k(\rho\n)$ (resp. $\overline{T_k(\rho)}$) is replaced by $\rho\n$ (resp $\rho$). Then, passing to the limit $n \rightarrow +\infty$ in the renormalized continuity equation
\[
\dive( (\rho\n \ln\rho\n) \, \bfu\n) = -\rho\n \dive\, \bfu\n \quad \text{in}~ \mathcal{D}'(\mathbb{R}^3),
\]
we get
\begin{equation*}
\dive(\overline{\rho \ln\rho} \ \bfu) = - \overline{\rho \dive\, \bfu} \quad \text{in}~ \mathcal{D}'(\mathbb{R}^3),
\end{equation*}
where all the functions have been extended by zero outside $\Omega$.
On the other hand, applying the renormalization theory of Di Perna-Lions (Lemma \ref{lem:diperna}) on the limit $\rho \in \xL^{3(\gamma-1)}(\Omega)$, $\bfu\in \xbfH^1_0(\Omega)$ we also have
\begin{equation}\label{eq:c-limit-renorm0}
\dive((\rho \ln\rho)\, \bfu) = - \rho \dive\, \bfu \quad \text{in}~ \mathcal{D}'(\mathbb{R}^3).
\end{equation}
Subtracting this equation from the previous one, we arrive at 
\begin{equation*}
\dive\big((\overline{\rho \ln\rho}-\rho \ln \rho) \ \bfu\big) = \rho \dive\, \bfu - \overline{\rho \dive\, \bfu}\quad \text{in}~ \mathcal{D}'(\mathbb{R}^3).
\end{equation*} 
The weak compactness property of the effective viscous flux yields
\begin{equation*}
\dive\big((\overline{\rho \ln\rho}-\rho \ln \rho) \ \bfu\big) 
= \dfrac{a}{2\mu+\lambda} \Big(\rho \overline{\rho^\gamma} - \overline{\rho^{\gamma+1}}\Big) \quad \text{in}~ \mathcal{D}'(\mathbb{R}^3).
\end{equation*} 
Integrating in space, we end up with the identity $\rho \overline{\rho^\gamma} = \overline{\rho^{\gamma+1}}$ a.e., from which the strong convergence of the density follows by invoking the monotonicity of the pressure (Minty's trick).

\medskip
In the arguments presented above, one of the key point is to write \eqref{eq:c-limit-renorm0} which requires from the theory of Di Perna and Lions that $\rho \in \xL^2(\Omega)$ (see Remark \ref{rmk:renorm}).
The previous proof can be adapted in the case $\gamma \in (\frac{5}{3},2)$ using a ``weaker'' version of the effective viscous flux identity where $\rho\n$ is essentially replaced by $\rho\n^\alpha$ for some $\alpha\in(0,1)$. 
This is the case initially demonstrated by Lions in \cite{lio-98-comp}. 
For smaller values of $\gamma$, i.e. $\frac{3}{2} <\gamma \leq \frac{5}{3}$, we do not ensure \emph{a priori} \eqref{eq:c-limit-renorm0} since $\rho$ does not belong to $\xL^2(\Omega)$. 
The idea of Feireisl \cite{feireisl2001} (adapted then by Novo and Novotn{\'y} in the stationary case) is to work on the truncated variable $T_k(\rho)$, defined in \eqref{df:Tk} which is bounded for fixed $k$ (and thus in $\xL^2(\Omega)$).
With similar arguments as before, one may then show the strong convergence of $(T_k(\rho\n))_{n\in \mathbb{N}}$ to $T_k(\rho)$ (uniformly with respect to $n$ in $\xL^{\gamma+1}(\Omega)$). 
Combining finally this result with the strong convergence of the truncated variables as $k \rightarrow +\infty$ (see Lemma \ref{lem:cvg-Tk} below), we will get the strong convergence of $(\rho\n)_{n\in\mathbb{N}}$.

\paragraph{Properties of the truncation operators $T_k$.}

\begin{lem}\label{lem:cvg-Tk}
	Under the assumptions of Proposition \ref{prop:c-eff-flux},
	there exists a constant $C$ such that the following inequality holds for all $1 \leq q < 3(\gamma-1)$, $n\in\xN$ and $k \in\xN^*$:
	\begin{align}
	\norm{\overline{T_k(\rho)} - \rho}_{\xL^q(\Omega)} + \norm{T_k(\rho) - \rho}_{\xL^q(\Omega)} + \norm{T_k(\rho\n) - \rho\n}_{\xL^q(\Omega)}
	\leq C k^{\frac{1}{3(\gamma-1)} - \frac{1}{q}}.
	\end{align}	
	Consequently, as $k \rightarrow +\infty$, the sequences $(\overline{T_k(\rho)})_{k \in \mathbb{N}^*}$ and $(T_k(\rho))_{k \in \mathbb{N}^*}$ both converge strongly to $\rho$ in $\xL^q(\Omega)$ for all $q \in [1,3(\gamma-1))$.
\end{lem}

\medskip
\begin{proof}
	As a consequence of the inequality
	\[
	|\{\rho\n \geq k\}| \leq \dfrac{1}{k} \int_{\Omega}{\rho\n \dx} = |\Omega|\dfrac{\rho^*_n}{k} \leq\dfrac{C}{k}
	\]
	we deduce by H\"older's inequality that for any $q \in[1, 3(\gamma-1))$
	\begin{align*}
	\norm{T_k(\rho\n) - \rho\n}_{\xL^q(\Omega)} 
	& = \norm{(T_k(\rho\n) - \rho\n) \mcal{X}_{\{\rho\n \geq k\}} }_{\xL^q(\Omega)} \\
	& \leq \norm{\rho\n  \mcal{X}_{\{\rho\n \geq k\}} }_{\xL^q(\Omega)} 
	 \leq C k^{\frac{1}{3(\gamma-1)} - \frac{1}{q}} \norm{ \rho\n}_{\xL^{3(\gamma-1)}(\Omega)}
	 \leq C k^{\frac{1}{3(\gamma-1)} - \frac{1}{q}},
	\end{align*}
	where $\mcal{X}_\omega$ is the characteristic function of a set $\omega$ and where the constant $C$ only depends on $q$ and the uniform bounds on the sequences $(\rho\n^\star)_{n\in\xN}$ and $(\norm{\rho\n}_{\xL^{3(\gamma-1)}})_{n\in\xN}$.
	Doing the same with the limit density $\rho$, we get
	\begin{align*}
	\norm{T_k(\rho) - \rho}_{\xL^q(\Omega)} 
	& \leq C k^{\frac{1}{3(\gamma-1)} - \frac{1}{q}}.
	\end{align*}
	Finally, we have:
	\[
	\norm{\overline{T_k(\rho)} - \rho}_{\xL^q(\Omega)} 
	 \leq \liminf_{n\to+\infty} \norm{T_k(\rho\n) - \rho\n}_{\xL^q(\Omega)} 
	 \leq \limsup_{n\to+\infty} \norm{T_k(\rho\n) - \rho\n}_{\xL^q(\Omega)} 
	 \leq C k^{\frac{1}{3(\gamma-1)} - \frac{1}{q}}.
	\]
	which ends the proof.
\end{proof}

\medskip
\begin{lem}\label{lem:osc}
	Under the assumptions of Proposition \ref{prop:c-eff-flux},
	there exists a constant $C$ such that the following estimate holds:
	\begin{equation}
	\sup_{k > 1} ~\limsup_{n \rightarrow +\infty} \norm{T_k(\rho\n) - T_k(\rho)}_{\xL^{\gamma+1}(\Omega)} \leq C.
	\end{equation}	
\end{lem}

\medskip
\begin{proof}
	First of all, observe that for all $r_1,r_2 \geq 0$,
	$
	|T_k(r_1) - T_k(r_2)|^{\gamma+1} \leq (r_1^\gamma - r_2^\gamma)(T_k(r_1) - T_k(r_2)),
	$ 
	and thus
	\begin{align*}
	\limsup_{n\to+\infty} \int_{\Omega} |T_k(\rho\n) - T_k(\rho)|^{\gamma+1}  \dx
	& \leq \limsup_{n\to+\infty} \int_{\Omega} (\rho\n^\gamma - \rho^\gamma)(T_k(\rho\n) - T_k(\rho))\dx \\
	& \leq \int_{\Omega}\big(\overline{\rho^\gamma T_k(\rho)} - \overline{\rho^\gamma}\ \overline{T_k(\rho)} \big)\dx
	+ \int_{\Omega}\big(\overline{\rho^\gamma} - \rho^\gamma\big) \big(\overline{T_k(\rho)} - T_k(\rho) \big)\dx.
	\end{align*}
	Invoking the convexity of the functions $t \mapsto t^\gamma$ and $t \mapsto - T_k(t)$, we have $\overline{\rho^\gamma} \geq \rho^\gamma$ and $\overline{T_k(\rho)} \leq T_k(\rho)$
	so that
	\begin{align*}
	\limsup_{n\to+\infty} \int_{\Omega} |T_k(\rho\n) - T_k(\rho)|^{\gamma+1} \dx
	& \leq \int_{\Omega}\big(\overline{\rho^\gamma T_k(\rho)} - \overline{\rho^\gamma}\ \overline{T_k(\rho)} \big)\dx.
	\end{align*}
	We can now use the weak compactness property satisfied by the effective viscous flux \eqref{eq:c-eff-flux}:
	\begin{align}\label{eq:c-appl-eff-flux}
	& \limsup_{n\to+\infty} \int_{\Omega} |T_k(\rho\n) - T_k(\rho)|^{\gamma+1} \dx\nonumber\\
	& \leq \dfrac{2\mu+\lambda}{a} \limsup_{n\to+\infty} \int_{\Omega}\big(T_k(\rho\n) - \overline{T_k(\rho)} \big)\dive \, \bfu\n \dx\\
	& \leq \dfrac{2\mu+\lambda}{a} \limsup_{n\to+\infty} \int_{\Omega}\big(T_k(\rho\n) - T_k(\rho) \big) \dive \, \bfu\n\dx
	+ \dfrac{2\mu+\lambda}{a} \limsup_{n\to+\infty} \int_{\Omega}\big(T_k(\rho) - \overline{T_k(\rho)} \big)\dive \, \bfu\n \dx\nonumber \\
	& \leq C \limsup_{n\to+\infty} \norm{T_k(\rho\n) - T_k(\rho)}_{\xL^2(\Omega)}.\nonumber
	\end{align}
	Where $C$ depends on the uniform bound on the sequence $(\norm{\dive \, \bfu\n}_{\xL^2(\Omega)})_{n\in\xN}$.
	Since $\gamma+1 > 2$, we obtain thanks to H\"older and Young inequalities
	\begin{align*}
	\limsup_{n\to+\infty} \int_{\Omega} |T_k(\rho\n) - T_k(\rho)|^{\gamma+1} \dx
	& \leq C + \dfrac{1}{2} \limsup_{n\to+\infty}  \norm{T_k(\rho\n) - T_k(\rho)}_{\xL^{\gamma+1}(\Omega)}^{\gamma+1}
	\end{align*}
	which achieves the proof.
\end{proof}

\paragraph{Renormalization equation associated with $T_k$.}
For any $k\in\xN^*$, define
\begin{equation}
L_k(t) = \begin{cases} 
\ t (\ln t - \ln k - 1), & \quad \text{if} ~ t \in [0,k), \\
\ -k,  & \quad \text{if} ~ t \in [k,+\infty),
\end{cases}
\end{equation}
which belongs to $\mathcal{C}^0([0,+\infty)) \cap \mathcal{C}^1((0,+\infty))$ and which is such that
\[
t L_k'(t) - L_k(t) = T_k(t) \quad \forall t \in [0,+\infty).
\]

\medskip
\begin{rmrk}
Note that function $L_k$ can be seen as a truncated version of the function $L(t) = t\ln t$ used in \eqref{eq:c-limit-renorm0} for large $\gamma$, up to the addition of the linear function $t \mapsto -(\ln k+1)t$.
\end{rmrk}
%

\medskip
\begin{prop}\label{prop:c-renorm-Lk}
Let $\Omega$ be a Lipschitz bounded domain of $\xR^3$. Assume that $\gamma \in (\frac 32,3]$.
Let $(\rho\n,\bfu\n)_{n\in\xN}$ be the sequence defined in Theorem \ref{thm:stab} and let
$(\rho,\bfu)\in\xL^{3(\gamma-1)}(\Omega)\times\xbfH^1_0(\Omega)$ be its limit.
Then, for all $k\in\xN^*$, the following equations hold:
\begin{align}
&\dive(L_k(\rho\n) \bfu\n) + T_k(\rho\n)\dive\, \bfu\n = 0,  \quad &\text{in} ~\mathcal{D}'(\mathbb{R}^3), &\quad \forall n\in\xN. \label{eq:c-Lk-n} \\
&\dive(L_k(\rho) \bfu) + T_k(\rho)\dive\, \bfu = 0 \quad &\text{in} ~\mathcal{D}'(\mathbb{R}^3). \label{eq:c-Lk}
\end{align}
\end{prop}

\medskip

\begin{proof}
Since $L_k\in \mathcal{C}^0([0,+\infty)) \cap \mathcal{C}^1((0,+\infty))$ satisfies \eqref{eq:cond-renorm}, 
and since $(\rho\n,\bfu\n)$ is a renormalized solution of \eqref{eq:pb}-\eqref{eq:pb_CL}-\eqref{hyp:mass}, in the sense of Definition \ref{df:c-weak-sol}, equation \eqref{eq:c-Lk-n} holds true.

\medskip
Let us prove that a similar equation is also satisfied for the limit couple $(\rho, \bfu)$.
Using the renormalization property~\eqref{eq:c-renormalized-eq-trunc} satisfied by $(\rho\n,\bfu\n)$ for the truncated function $T_M$, $M\in\xN^*$, we obtain:
 \[
 \dive\big(T_M(\rho\n)\bfu\n\big) 
 = -\big[\rho\n [T_M]'_+(\rho\n) - T_M(\rho\n)\big]\dive \, \bfu\n 
 \quad \text{in}~ \mathcal{D}'(\xR^3).
 \]
which yields as $n\rightarrow +\infty$ 
\begin{equation}
\label{eq:c-limit-renorm}
\dive\big(\overline{T_M(\rho)}\bfu\big) 
= -\overline{\big[\rho [T_M]'_+(\rho) - T_M(\rho)\big]\dive \, \bfu}
\quad \text{in}~ \mathcal{D}'(\xR^3).
\end{equation}
For $k\in\xN^*$ and $\delta > 0$, we introduce the regularized function $L_{k,\delta}$ defined as
\begin{equation}\label{eq:c-reg-Lk}
L_{k,\delta}(t) = L_k(t+\delta),
\end{equation}
the derivative of which is bounded close to $0$ unlike $L_k$.
Applying Lemma \ref{lem:diperna} (and the second part of Remark \ref{rmk:renorm}) to the pair $(\overline{T_M(\rho)},\bfu)$ (justified since $\overline{T_M(\rho)} \in L^\infty(\Omega)$ for $M$ fixed) with the function
$L_{k,\delta}$ and the source term $g = -\overline{\big[\rho [T_M]'_+(\rho) - T_M(\rho)\big]\dive \, \bfu} \in \xL^1_{\rm loc}(\mathbb{R}^3)$, we get:
\begin{equation}
\label{eq:c-limit-renorm-Mk}
\dive\big(L_{k,\delta}\big(\overline{T_M(\rho)}\big)\bfu\big) + T_{k,\delta}\big(\overline{T_M(\rho)}\big) \dive\, \bfu  
= -L_{k,\delta}'\big(\overline{T_M(\rho)}\big) \overline{\big[\rho [T_M]'_+(\rho) - T_M(\rho)\big]\dive \, \bfu}
\quad \text{in}~ \mathcal{D}'(\xR^3)
\end{equation}
where
\[
T_{k,\delta}(t) = t L_{k,\delta}'(t)  - L_{k,\delta}(t). 
\]
We now have to pass to the limits $M \rightarrow +\infty$, $\delta \rightarrow 0^+$.
Lemma \ref{lem:cvg-Tk} yields the strong convergence of $\overline{T_M(\rho)}$ to $\rho$ as $M\rightarrow +\infty$.
As a consequence, the left-hand side of \eqref{eq:c-limit-renorm-Mk} converges in $\mathcal{D}'(\mathbb{R}^3)$ to
\[
\dive\big(L_{k,\delta}(\rho) \bfu\big) + T_{k,\delta}(\rho) \dive\, \bfu. 
\]
Regarding the right-hand side
\[
-L_{k,\delta}'\big(\overline{T_M(\rho)}\big) \overline{\big[\rho [T_M]'_+(\rho) - T_M(\rho)\big]\dive \, \bfu},
\]
since $L_{k,\delta}'(t) = 0$ for $t > k$, we estimate its $\xL^1$ norm as follows
\[
\left|\int_{\Omega}
L_{k,\delta}'\big(\overline{T_M(\rho)}\big) \overline{\big[\rho [T_M]'_+(\rho) - T_M(\rho)\big]\dive \, \bfu} \dx
\right| 
\leq \max_{t \in [0,k]}|L_{k,\delta}'(t)| \int_{\Omega}{\left|  \overline{\big[\rho [T_M]'_+(\rho) - T_M(\rho)\big]\dive \, \bfu}\right| \mcal{X}_{\Omega_{M,k}} \dx }
\]
where $\Omega_{M,k} = \left \lbrace \bfx\in\Omega,\, \overline{T_M(\rho)}(\bfx) \leq k \right \rbrace$.
We then have
\begin{align}
 \max_{t \in [0,k]}|L_{k,\delta}'(t)|  & \int_{\Omega}{\left|  \overline{\big[\rho [T_M]'_+(\rho) - T_M(\rho)\big]\dive \, \bfu}\right| \mcal{X}_{\Omega_{M,k}} \dx } \nonumber \\
& \leq \max_{t \in [0,k]}|L_{k,\delta}'(t)| \ \liminf_{n \rightarrow +\infty} \int_{\Omega}{\left|  \big[\rho\n [T_M]'_+(\rho\n) - T_M(\rho\n)\big]\dive \, \bfu\n \right| \mcal{X}_{\Omega_{M,k}}  \dx} \nonumber \\
& \leq \max_{t \in [0,k]}|L_{k,\delta}'(t)| \ \limsup_{n \rightarrow +\infty} \int_{\Omega}{\left|  \big[\rho\n [T_M]'_+(\rho\n) - T_M(\rho\n)\big]\dive \, \bfu\n \right| \mcal{X}_{\Omega_{M,k}}  \dx} \nonumber \\
& \leq C \max_{t \in [0,k]}|L_{k,\delta}'(t)| \ \limsup_{n \rightarrow +\infty} \norm{T_M(\rho\n)\mcal{X}_{\Omega_{M,k} \cap \{\rho\n \geq M\} } }_{\xL^2(\Omega)},
\label{eq:c-ineg-Tk}
\end{align}
since the sequence $(\norm{\dive \, \bfu\n}_{\xL^2(\Omega)} )_{n\in\xN}$ is bounded.
We already know that $T_M(\rho\n)\mcal{X}_{\Omega_{M,k} \cap \{\rho\n \geq M\}}$ is controlled in $\xL^1(\Omega)$ since
\[
\norm{T_M(\rho\n)\mcal{X}_{\Omega_{M,k} \cap \{\rho\n \geq M\}}}_{\xL^1(\Omega)} 
 \leq \norm{\rho\n \mcal{X}_{\{\rho\n \geq M\}} }_{\xL^1(\Omega)} 
 \leq C M^{\frac{1}{3(\gamma-1)} - 1} \norm{ \rho\n}_{\xL^{3(\gamma-1)}(\Omega)}
 \leq C M^{\frac{1}{3(\gamma-1)} - 1},
\]
where the constant $C$ only depends on the uniform bounds on the sequences $(\rho\n^\star)_{n\in\xN}$ and $(\norm{\rho\n}_{\xL^{3(\gamma-1)}})_{n\in\xN}$. Therefore, by an interpolation inequality, we obtain:
\begin{align*}
\limsup_{n\to+\infty} & \norm{T_M(\rho\n)\mcal{X}_{\Omega_{M,k} \cap \{\rho\n \geq M\} } }_{\xL^2(\Omega)} \\
& \leq C \limsup_{n\to+\infty} \norm{T_M(\rho\n)\mcal{X}_{\Omega_{M,k} \cap \{\rho\n \geq M\} } }_{\xL^1(\Omega)}^{\frac{\gamma-1}{2\gamma}} \norm{T_M(\rho\n)\mcal{X}_{\Omega_{M,k} \cap \{\rho\n \geq M\} } }_{\xL^{\gamma+1}(\Omega)}^{\frac{\gamma+1}{2\gamma}} \\
& \leq  C M^{\frac{\gamma-1}{2\gamma} \big (\frac{1}{3(\gamma-1)}-1 \big )} \limsup_{n\to+\infty} \Big(\norm{\big(T_M(\rho\n) - \overline{T_M(\rho)}\big)\mcal{X}_{\Omega_{M,k} } }_{\xL^{\gamma+1}(\Omega)} 
+\norm{\overline{T_M(\rho)}\mcal{X}_{\Omega_{M,k} } }_{\xL^{\gamma+1}(\Omega)}
 \Big)^{\frac{\gamma+1}{2\gamma}} \\
& \leq  C M^{\frac{\gamma-1}{2\gamma} \big (\frac{1}{3(\gamma-1)}-1 \big )} \limsup_{n\to+\infty}
\Big(\norm{\big(T_M(\rho\n) - \overline{T_M(\rho)}\big) }_{\xL^{\gamma+1}(\Omega)}
	+ k|\Omega|^{\frac{1}{\gamma+1}} \Big)^{\frac{\gamma+1}{2\gamma}}\\
\end{align*}
Thanks to Lemma \ref{lem:osc}, we deduce that
\[
 \limsup_{n\to+\infty}  \norm{T_M(\rho\n)\mcal{X}_{\Omega_{M,k} \cap \{\rho\n \geq M\} } }_{\xL^2(\Omega)} \leq  C(k,\Omega) M^{\frac{\gamma-1}{2\gamma} \big (\frac{1}{3(\gamma-1)}-1 \big )}  \rightarrow 0 \quad \text{as} ~ M \rightarrow +\infty.
\]
Injecting in \eqref{eq:c-ineg-Tk}, we get
\begin{align*}
\max_{t \in [0,k]}|L_{k,\delta}'(t)| \int_{\Omega}{\left|  \overline{\big[\rho [T_M]'_+(\rho) - T_M(\rho)\big]\dive \, \bfu} \right| \mcal{X}_{\Omega_{M,k}} \dx  } \ 
\rightarrow 0 \quad \text{as} ~ M \rightarrow +\infty.
\end{align*}
Note that to get this result, we have been forced to regularize the function $L_k$ (see \eqref{eq:c-reg-Lk}) in order to control its derivative close to $0$.
Hence, passing to the limit $M\rightarrow +\infty$ in \eqref{eq:c-limit-renorm-Mk} we get
\begin{equation*}
\dive\big(L_{k,\delta}(\rho)\bfu) + T_{k,\delta}(\rho)\dive \, \bfu = 0, \quad \text{in $\mathcal{D}'(\xR^3)$}, \quad \forall \, k\in\xN^*,\delta > 0.
\end{equation*}
Now, observe that for all $t \in [0,+\infty)$
\begin{align*}
L_{k,\delta}(t) & \underset{\delta \rightarrow 0^+}{\longrightarrow} L_k(t), \\
T_{k,\delta}(t)=t L_{k,\delta}'(t) - L_{k,\delta}(t) & \underset{\delta \rightarrow 0^+}{\longrightarrow} t L_{k}'(t) - L_{k}(t) = T_k(t).
\end{align*}
Moreover, since for all $t\in[0,+\infty)$ and $\delta\in(0,1)$, we have
$|L_{k,\delta}(t)| \leq  k $ and
\begin{align*}
|T_{k,\delta}(t)| 
&= |T_{k}(t+\delta)  -  \delta L_k'(t+\delta)| \\
& \leq k + \delta |\ln(t+\delta)-\ln\,k| \mcal{X}_{\lbrace t+\delta\leq k \rbrace}
 \leq k + \delta \Big( |\ln\, \delta| + \frac{t}{\delta} + \ln\,k \Big) \mcal{X}_{\lbrace t+\delta\leq k \rbrace} 
 \leq C(k),
\end{align*}
we can pass to the limit $\delta \rightarrow 0^+$ thanks to the Lebesgue Dominated Convergence Theorem to get
\[
\dive\big(L_{k}(\rho)\bfu) + T_k(\rho) \dive \, \bfu = 0, \quad \text{in $\mathcal{D}'(\xR^3)$},
\] 
for all $k\in\xN^*$ which concludes the proof.
\end{proof}

\medskip
\paragraph{Strong convergence of the density}
\begin{prop}{\label{prop:c-cvg-rho}}
Let $\Omega$ be a Lipschitz bounded domain of $\xR^3$. Assume that $\gamma \in (\frac 32,3]$.
Let $(\rho\n,\bfu\n)_{n\in\xN}$ be the sequence defined in Theorem \ref{thm:stab} and let
$(\rho,\bfu)\in\xL^{3(\gamma-1)}(\Omega)\times\xbfH^1_0(\Omega)$ be its limit.
Up to extraction, the sequence $(\rho\n)_{n\in\xN}$ strongly converges towards $\rho$ in $\xL^q(\Omega)$ for all $q\in[1,3(\gamma-1))$. 
\end{prop}

\begin{proof}
Integrating the renormalized continuity equations \eqref{eq:c-Lk-n} and \eqref{eq:c-Lk} and summing, one obtains:
\[
 \int_{\Omega}T_k(\rho\n) \dive\, \bfu\n \dx - \int_{\Omega}T_k(\rho) \dive\, \bfu \dx = 0, \qquad \forall n\in\xN.
\]
We then use this identity in inequality \eqref{eq:c-appl-eff-flux} to deduce that
\begin{align*}
\limsup_{n\to+\infty} & \int_{\Omega} |T_k(\rho\n) - T_k(\rho)|^{\gamma+1}\dx \\ 
& \leq \dfrac{2\mu+\lambda}{a} \limsup_{n\to+\infty} \int_{\Omega}\big(T_k(\rho\n) - \overline{T_k(\rho)} \big)\dive \, \bfu\n \dx \\
&  = \dfrac{2\mu+\lambda}{a}  \int_{\Omega}(T_k(\rho)-\overline{T_k(\rho)}) \, \dive \, \bfu \dx
+ \limsup_{n\to+\infty} \Big ( \int_{\Omega}T_k(\rho\n) \dive\, \bfu\n \dx - \int_{\Omega}T_k(\rho) \dive\, \bfu \dx \Big ) \\
&  = \dfrac{2\mu+\lambda}{a}  \int_{\Omega}(T_k(\rho)-\overline{T_k(\rho)}) \, \dive \, \bfu \dx.
\end{align*}
Using Lemma \ref{lem:cvg-Tk} we infer that
\begin{align*}
\left| \int_{\Omega}(T_k(\rho)-\overline{T_k(\rho)}) \, \dive \, \bfu \dx \right| 
& \leq \norm{\dive\, \bfu}_{\xL^2(\Omega)} \norm{T_k(\rho)-\overline{T_k(\rho)}}_{\xL^2(\Omega)} \\
& \leq C \norm{T_k(\rho)-\overline{T_k(\rho)}}_{\xL^1(\Omega)}^{\frac{\gamma-1}{2\gamma}} \norm{T_k(\rho)-\overline{T_k(\rho)}}_{\xL^{\gamma+1}(\Omega)}^{\frac{\gamma+1}{2\gamma}} \\
& \leq   C \Big(\norm{T_k(\rho)-\rho}_{\xL^1(\Omega)} +
\norm{\overline{T_k(\rho)}-\rho}_{\xL^1(\Omega)}\Big)^{\frac{\gamma-1}{2\gamma}}  \norm{T_k(\rho)-\overline{T_k(\rho)}}_{\xL^{\gamma+1}(\Omega)}^{\frac{\gamma+1}{2\gamma}} \\
& \leq  C k^{\frac{\gamma-1}{2\gamma} \big (\frac{1}{3(\gamma-1)}-1 \big )}  \norm{T_k(\rho)-\overline{T_k(\rho)}}_{\xL^{\gamma+1}(\Omega)}^{\frac{\gamma+1}{2\gamma}} \\
& \leq C k^{\frac{\gamma-1}{2\gamma} \big (\frac{1}{3(\gamma-1)}-1 \big )} \Big( \limsup_{n \rightarrow +\infty} \norm{T_k(\rho\n) - T_k(\rho)}_{\xL^{\gamma+1}(\Omega)}\Big )^{\frac{\gamma+1}{2\gamma}}.
\end{align*}
As a consequence of Lemma \ref{lem:osc}, we have
\begin{equation*}
\lim_{k \rightarrow +\infty} \limsup_{n\to+\infty} \int_{\Omega} |T_k(\rho\n) - T_k(\rho)|^{\gamma+1}\dx = 0,
\end{equation*}
and thus
\begin{equation}\label{eq:est-Tk}
\lim_{k \rightarrow +\infty} \limsup_{n\to+\infty}  \norm{T_k(\rho\n) - T_k(\rho)}_{\xL^1(\Omega)} = 0.
\end{equation}
We conclude to the strong convergence of the density by writing
\[
\norm{\rho - \rho\n}_{\xL^1(\Omega)} 
\leq \norm{\rho\n - T_k(\rho\n)}_{\xL^1(\Omega)} + \norm{T_k(\rho\n) - T_k(\rho)}_{\xL^1(\Omega)} + \norm{T_k(\rho) - \rho}_{\xL^1(\Omega)}.
\]
Passing to the limit superior $n \rightarrow +\infty$ in this inequality, using again Lemma \ref{lem:cvg-Tk} and the previous estimate \eqref{eq:est-Tk} to then pass to the limit $k \rightarrow +\infty$ yields the strong convergence of the density in $\xL^1(\Omega)$ and therefore in $\xL^q(\Omega)$ for all $q\in[1,3(\gamma-1))$. 
\end{proof}

\subsection{Elements for the construction of weak solutions}
\label{sec:elements:construct}
The previous subsections were concerned with the stability of weak solutions of Problem \eqref{eq:pb}-\eqref{eq:pb_CL}-\eqref{hyp:mass}.
An important step is the effective construction of such weak solutions by means of successive approximations.
This construction is sketched by Lions in \cite{lio-98-comp} for the case $\gamma > \frac{5}{3}$ and detailed by Novo and Novotn{\' y} in \cite{novo2002} for the case $\gamma > \frac{3}{2}$.
The approximation procedure is usually decomposed into three steps:
\begin{itemize}
	\item addition in the momentum equation of an \emph{artificial pressure} term $\delta \gradi \rho^{\Gamma}$ with $\Gamma$ sufficiently large, namely $\Gamma > 3$;
	
	\item addition of a \emph{relaxation term} $\alpha (\rho -\rho^\star)$ in the mass equation in order to ensure that the total mass constraint \eqref{hyp:mass} is satisfied at the approximate level;
	
	\item addition of a diffusion or \emph{regularization term} ({\it e.g.} $-\varepsilon \Delta \rho$) in the mass equation which regularizes the density.
\end{itemize}
As a consequence of the modification of the mass equation, the momentum equation can also involve additional perturbation terms depending on $\alpha$ and $\varepsilon$, in order to ensure the preservation of the energy inequality (see details in \cite{novo2002} or \cite{novotny2004}). 
The parameters $\delta, \alpha, \varepsilon$ being fixed, the existence of weak solutions is obtained by a fixed point argument. 
Then, the proof consists in passing to the limit successively with respect to $\varepsilon$, $\alpha$ and then finally with respect to $\delta$. 

\medskip
In the next section, we present our numerical scheme which essentially reproduces at the discrete level the previous three approximation terms. 
Nevertheless, the parameters $\varepsilon, \, \alpha,\, \delta$ are no more independent in the discrete case,
they shall all depend on the mesh size and converge to $0$ as the mesh size tends to $0$.
In that sense, the convergence result that we obtain in the next sections can be seen as an alternative proof of existence of weak solutions to the stationary problem \eqref{eq:pb}-\eqref{eq:pb_CL}-\eqref{hyp:mass}.


\section{The discrete setting, presentation of the numerical scheme}\label{sec:discrete}

\subsection{Meshes and discretization spaces}\label{sec:space_disc}

Let $\Omega$ be an open bounded connected subset of $\xR^d$, $d=2$ or $3$.
We assume that $\Omega$ is polygonal if $d=2$ and polyhedral if $d=3$.

\begin{defi}[Staggered mesh] \label{def:disc}
A staggered discretization of $\Omega$, denoted by $\disc$, is given by a pair $\disc=(\mesh,\edges)$, where:
\begin{itemize}
\item 
$\mesh$, the so-called primal mesh, is a finite family composed of non empty simplices.
The primal mesh $\mesh$ is assumed to form a partition of $\Omega$ : $\overline{\Omega}= \displaystyle{\cup_{K \in \mesh} \overline K}$.
For any simplex $K\in\mesh$, let $\dv K  = \overline K\setminus K$ be the boundary of $K$, which is the union of cell faces.
We denote by $\edges$ the set of faces of the mesh, and we suppose that two neighboring cells share a whole face: for all $\edge\in\edges$, either $\edge\subset \dv\Omega$ or there exists $(K,L)\in \mesh^2$ with $K \neq L$ such that $\overline K \cap \overline L  = \edge$; we denote in the latter case $\edge = K|L$.
We denote by $\edgesext$ and $\edgesint$ the set of external and internal faces: $\edgesext=\lbrace \edge \in \edges, \edge \subset \dv \Omega \rbrace$ and $\edgesint=\edges \setminus \edgesext$.
For $K \in \mesh$, $\edges(K)$ stands for the set of faces of $K$. The unit vector normal to $\edge \in \edges(K)$ outward $K$ is denoted by $\bfn_{K,\edge}$.
In the following, the notation $|K|$ or $|\edge|$ stands indifferently for the $d$-dimensional or the $(d-1)$-dimensional measure of the subset $K$ of $\xR^d$ or $\edge$ of $\xR^{d-1}$ respectively.
\medskip
\item We define a dual mesh associated with the faces $\edge\in\edges$ as follows.
For $K\in\mesh$ and $\edge \in \edges(K)$, we define $D_{K,\edge}$ as the cone with basis $\edge$ and with vertex the mass center of $K$ (see Figure \ref{fig:mesh}).
We thus obtain a partition of $K$ in $m$ sub-volumes, where $m$ is the number of faces of $K$, each sub-volume having the same measure $| D_{K,\edge}|= |K|/m$.
The volume $D_{K,\edge}$ is referred to as the half-diamond cell associated with $K$ and $\edge$.
For $\edge \in \edgesint$, $\edge=K|L$, we now define the diamond cell $D_\edge$ associated with $\edge$ by $D_\edge=D_{K,\edge} \cup D_{L,\edge}$. For $\edge\in\edgesext\cap\edges(K)$, we define $D_\edge=D_{K,\edge}$.
We denote by $\edgesd(D_\edge)$ the set of faces of $D_\edge$, and by $\edged=D_\edge|D_{\edge'}$ the face separating two diamond cells $D_\edge$ and $D_{\edge'}$.
As for the primal mesh, we denote by $\edgesdint$ the set of dual faces included in the domain $\Omega$ and by $\edgesdext$ the set of dual faces lying on the boundary $\dv \Omega$.
In this latter case, there exists $\edge\in\edgesext$ such that $\edged=\edge$.
\end{itemize}
\end{defi}

\begin{figure}[tb]
\begin{center}
\includegraphics[scale=1]{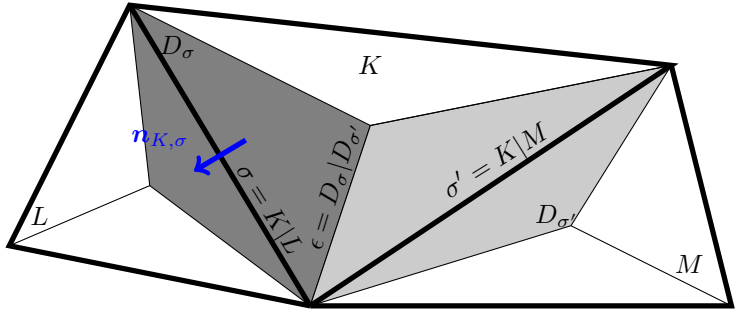}
\end{center}
\caption{\label{fig:mesh}Notations for primal and dual cells. Primal cells are delimited with bold lines, dual cells are in grey.}
\end{figure}

\medskip
\begin{defi}[Size of the discretization] 
\label{defi:pas}
Let $\disc=(\mesh,\edges)$ be a staggered discretization of $\Omega$.
For every $K\in\mesh$, we denote $h_K$ the diameter of $K$ (\emph{i.e.} the 1D measure of the largest line segment included in $K$).
The size of the discretization is defined by:
\[
 h_\mesh = \max\limits_{K\in\mesh} \ h_K.
\]
\end{defi}

\medskip
\begin{defi}[Regularity of the discretization] 
Let $\disc=(\mesh,\edges)$ be a staggered discretization of $\Omega$.
For every $K\in\mesh$, denote $\varrho_K$ the radius of the largest ball included in $K$. The regularity parameter of the discretization is defined by:
\begin{equation}
\label{eq:reg}
 \theta_\mesh = \max\limits \Big \lbrace \frac{h_K}{\varrho_K}, \, K\in\mesh \Big \rbrace \cup \Big \lbrace \frac{h_K}{h_L},\, \frac{h_L}{h_K}, \, \edge=K|L\in\edgesint \Big \rbrace.
\end{equation}
\end{defi}

\medskip
Relying on Definition \ref{def:disc}, we now define a staggered space discretization.
The degrees of freedom for the density (\ie\ the discrete density unknowns) are associated with the cells of the mesh $\mesh$:
\[
\big\{ \rho_K,\ K \in \mesh \big\}.
\]
The discrete density unknowns are associated with piecewise constant functions on the cells of the primal mesh.
\begin{defi}
\label{def:space-p}
Let $\disc=(\mesh,\edges)$ be a staggered discretization of $\Omega$.
We denote $\xL_\mesh(\Omega)$ the space of scalar functions that are constant on each primal cell $K\in\mesh$. For $\rho\in \xL_\mesh(\Omega)$ and $K\in\mesh$, we denote $\rho_K$ the constant value of $\rho$ on $K$. We denote $\xL_{\mesh,0}(\Omega)$ the subspace of $\xL_\mesh(\Omega)$ composed of zero average functions over $\Omega$.
\end{defi}

\medskip
The degrees of freedom for the velocity are associated with the faces of the mesh $\mesh$ or equivalently with the cells of the dual mesh $D_\edge$, $\edge\in\edges$ so the set of discrete velocity unknowns reads:
\[
\lbrace \bfu_\edge \in \xR^d,\ \edge \in \edges\rbrace.
\]
The discrete velocity unknowns are associated with the \emph{Crouzeix-Raviart} finite element.
For all $K\in\mesh$, the restriction of the discrete velocity belongs to $P_1(K)$ the space of polynomials of degree less than one defined on $K$.

\medskip
The space of discrete velocities is given in the following definition.

\medskip
\begin{defi}
\label{defi:space-u:M}
Let $\disc=(\mesh,\edges)$ be a staggered discretization of $\Omega$ as defined in Definition \ref{def:disc}.
We denote $\xH_{\mesh}(\Omega)$ the space of functions $u$ such that $u_{|K} \in P_1(K)$ for all $K\in\mesh$ and such that:
\begin{equation}
\label{cont:bfu}
\frac{1}{|\edge|}\int_\edge [u]_\edge \dedge(\bfx)= 0, \qquad \forall \edge\in\edgesint, 
\end{equation}
where $[u]_\edge$ is the jump of $u$ through $\edge$ which is defined on $\edge=K|L$ by $[u]_\edge= u_{|L}- u_{|K}$.
We define $\xH_{\mesh,0}(\Omega)\subset\xH_{\mesh}(\Omega)$ the subspace of $\xH_{\mesh}(\Omega)$ composed of functions the degrees of freedom of which are zero over $\dv \Omega$, \emph{i.e.} the functions $u\in\xH_{\mesh}(\Omega)$ such that $\frac{1}{|\edge|}\int_\edge u \dedge(\bfx) = 0$ for all $\edge\in\edgesext$.
Finally, we denote $\xbfH_{\mesh}(\Omega):=\xH_{\mesh}(\Omega)^d$ and $\xbfH_{\mesh,0}(\Omega):=\xH_{\mesh,0}(\Omega)^d$.
\end{defi}

\medskip
For a discrete velocity field $\bfu\in\xbfH_{\mesh}(\Omega)$ and $\edge\in\edges$, the degree of freedom associated with $\edge$ is given by:
\begin{equation}
 \label{DOF:u}
\bfu_\edge=\frac{1}{|\edge|}\int_\edge \bfu \dedge(\bfx). 
\end{equation}
Although $\bfu\in\xbfH_\mesh(\Omega)$ is discontinuous across an internal face $\edge\in\edgesint$, the definition of $\bfu_\edge$ is unambiguous thanks to \eqref{cont:bfu}.

\subsection{The numerical scheme} \label{sec:scheme}
Let $\Omega$ be a polyhedral domain of $\xR^d$. Let $\disc=(\mesh,\edges)$ be a staggered discretization of $\Omega$ as defined in Definition \ref{def:disc}. The continuity equation is discretized on the primal mesh, while the momentum balance is discretized on the dual mesh.
The scheme reads as follows:

\medskip
\emph{Solve for $\rho\in\xL_\mesh(\Omega)$ and $\bfu\in\xbfH_{\mesh,0}(\Omega)$:}
\begin{subequations}\label{eq:sch}
\begin{align}
\label{eq:sch_mass}  & 
\dive_\mesh(\rho \bfu)  + h_\mesh^{\xi_1} \, (\rho-\rho^\star) 
-
h_\mesh^{\xi_2} \, \Delta_{\frac{1+\eta}{\eta},\mesh}(\rho)= 0, \\[2ex] 
\label{eq:sch_mom}  &
\divv_\edges(\rho \bfu\otimes \bfu) - \mu \lapi_\edges \bfu - (\mu+\lambda) (\gradi \circ \dive)_\edges \bfu + a\gradi_\edges (\rho^\gamma)   + h_\mesh^{\xi_3} \,  \gradi_\edges (\rho^\Gamma)   = \widetilde{\Pi}_\edges\bff,
\end{align}
\end{subequations}
where $\eta=\frac{2\gamma-3}{\gamma}$. 

\medskip
The discrete space differential operators involved in \eqref{eq:sch_mass} and \eqref{eq:sch_mom}, as well as their main properties, are described in the following lines. The positive constants $\Gamma$ and $(\xi_1,\xi_2,\xi_3)$ will be determined so as to ensure the convergence of the numerical solution towards a weak solution of \eqref{eq:pb}-\eqref{eq:pb_CL}-\eqref{hyp:mass}.

\medskip

\paragraph{Mass convection operator --}
Given discrete density and velocity fields $\rho\in\xL_\mesh(\Omega)$ and $\bfu\in\xbfH_{\mesh,0}(\Omega)$, the discretization of the mass convection term is given by:
\begin{equation}
\label{div_mass}
\dive_\mesh(\rho \bfu)(\bfx) = \sum_{K\in\mesh} \, \frac 1 {|K|} \Big (\sum_{\edge \in\edges(K)} F_{K,\edge}(\rho,\bfu) \Big )\, \Ind_K(\bfx),
\end{equation}
where $\Ind_K$ is the characteristic function of the subset $K$ of $\Omega$. 
The quantity $F_{K,\edge}(\rho,\bfu)$ stands for the mass flux across $\edge$ outward $K$.
By the impermeability boundary conditions, it vanishes on external faces and is given on internal faces by:
\begin{equation}
\label{eq:def_FKedge}
F_{K,\edge}(\rho,\bfu)= 
|\edge|\, \rho_\edge\, \bfu_\edge\cdot \bfn_{K,\edge} 
, \qquad \forall \edge\in\edgesint,\, \edge=K|L.
\end{equation}
The density at the face $\edge=K|L$ is approximated by the upwind technique, \ie
\begin{equation}\label{df:rho-upwind}
\rho_\edge = \begin{cases}
 \rho_K  \quad & \text{if} \quad  \bfu_\edge\cdot \bfn_{K,\edge} \geq 0, \\
 \rho_L  \quad & \text{otherwise}.
\end{cases}
\end{equation}

%
\paragraph{Stabilization terms in the mass equation --}
The discrete mass equation involves two stabilization terms. The first stabilization term is there to ensure the total mass constraint at the discrete level \eqref{hyp:mass}:
\[
h_\mesh^{\xi_1} \, (\rho(\bfx)-\rho^\star) = h_\mesh^{\xi_1} \, \sum_{K\in\mesh} (\rho_K-\rho^\star) \Ind_K(\bfx).
\]

\medskip
The second stabilization term in the discrete mass equation \eqref{eq:sch_mass} is defined as follows:
\begin{equation}
\label{Tstab2}
- \Delta_{\frac{1+\eta}{\eta},\mesh}(\rho)(\bfx) =  \sum_{K\in\mesh} \, \frac 1 {|K|} \Big (\sum_{\substack{\edge \in\edges(K)\cap \edgesint \\ \edge= K|L}} |\edge| \,\Big(\dfrac{|\edge|}{|D_\edge|}\Big)^\frac{1}{\eta}\,|\rho_K-\rho_L|^{\frac{1}{\eta}-1}\,(\rho_K-\rho_L) \Big )\, \Ind_K(\bfx).
\end{equation}
Its aim is to provide a control on a discrete analogue of the $\xW^{1,\frac{1+\eta}{\eta}}(\Omega)$ semi-norm of $\rho$ by some (negative) power of the discretization parameter $h_\mesh$. 
This control appears to be necessary in the convergence analysis, when passing to the limit in the equation of state, see Remark \ref{rmrk:MAC}.


%
\paragraph{Velocity convection operator --} 
\label{sec:vel_conv_op}

Given discrete density and velocity fields $\rho\in\xL_\mesh(\Omega)$ and $\bfu\in\xbfH_{\mesh,0}(\Omega)$, the discretization of the mass convection term is given by:
\begin{equation} \label{eq:div_conv}
\divv_\edges(\rho \bfu\otimes \bfu)(\bfx) = \sum_{\edge\in\edgesint}\frac 1 {|D_\edge|} \Big ( \sum_{\edged\in\edgesd(D_\edge)} \fluxd(\rho,\bfu)\ \bfu_\edged \Big ) \, \Ind_{\Ds}(\bfx).
\end{equation}
$\fluxd(\rho,\bfu)$ is the mass flux across the edge $\edged$ of the dual cell $D_\edge$. Its value is zero if $\edged\in\edgesdext$. Otherwise, it is defined as a linear combination, with constant coefficients, of the primal mass fluxes at the neighboring faces. 
For $K \in \mesh$ and $\edge \in \edges(K)$, let $\xi_K^\edge$ be given by:
\begin{equation}
\label{eq:xiksigma}
\xi_K^\edge=\frac{|D_{K,\edge}|}{|K|}=\frac{1}{d+1},
\end{equation}
so that $\sum_{\edge \in \edges(K)} \xi_K^\edge=1$.
Then the mass fluxes through the inner dual faces are calculated from the primal mass fluxes $F_{K,\edge}(\rho,\bfu)$ as follows. 
We first incorporate the second stabilization term (see \eqref{Tstab2}) into the primal mass fluxes by 
defining $\overline{F}_{K,\edge}(\rho,\bfu)$ as follows:
\begin{equation}
\label{eq:def_FKedgebar}
\overline{F}_{K,\edge}(\rho,\bfu)=  
F_{K,\edge}(\rho,\bfu) + h_\mesh^{\xi_2} \, |\edge| \,\Big(\dfrac{|\edge|}{|D_\edge|}\Big)^\frac{1}{\eta}\,|\rho_K-\rho_L|^{\frac{1}{\eta}-1}\,(\rho_K-\rho_L)
, \quad \forall \edge\in\edgesint,\, \edge=K|L.
\end{equation}
The dual mass fluxes $\fluxd(\rho,\bfu)$ are then computed to as to satisfy the following three conditions:
\begin{itemize}
\item[(H1)] The discrete mass balance over the half-diamond cells is satisfied, in the following sense.
For all primal cell $K$ in $\mesh$, the set $(\fluxd(\rho,\bfu))_{\edged\subset K}$ of dual fluxes included in $K$ solves the following linear system
\begin{equation}\label{eq:F_syst}
\overline{F}_{K,\edge}(\rho,\bfu) + \sum_{\edged \in \edgesd(D_\sigma),\ \edged \subset K} F_{\edge,\edged}(\rho,\bfu)=
\xi_K^\edge \sum_{\edge' \in \edges(K)} \overline{F}_{K,\edge'}(\rho,\bfu), \quad \edge \in \edges(K).
\end{equation}
\item[(H2)] The dual fluxes are conservative, \ie\ for any dual face $\edged=D_\edge|D_{\edge'}$, we have $F_{\edge,\edged}(\rho,\bfu)=-F_{\edge',\edged}(\rho,\bfu)$.
\item[(H3)] The dual fluxes are bounded with respect to the primal fluxes $(\overline{F}_{K,\edge}(\rho,\bfu))_{\edge \in \edges(K)}$, in the sense that
\begin{equation}\label{eq:F_bounded}
|F_{\edge,\epsilon}(\rho,\bfu)| \leq \ \max \,\left \lbrace |\overline{F}_{K,\edge'}(\rho,\bfu)|,\ \edge' \in \edges(K) \right \rbrace,
\end{equation}
for $K \in \mesh$, $\edge \in \edges(K)$, $\epsilon \in \edgesd(D_\sigma)$ with $\edged\subset K$.
\end{itemize}
The system of equations \eqref{eq:F_syst} 
does not depend on the particular cell $K$ since it only depends on the coefficient $\xi_K^\edge=1/(d+1)$.
It has an infinite number of solutions, which makes necessary to impose in addition the constraint \eqref{eq:F_bounded}; however, assumptions (H1)-(H3) are sufficient for the subsequent developments, in the sense that any choice for the expression of the dual fluxes satisfying these assumptions yields stable and consistent schemes (see \cite{lat-18-conv, lat-18-disc}).

\medskip
This convection operator is built so that a discrete mass conservation equation similar to \eqref{eq:sch_mass} is also satisfied on the cells of the dual mesh. Indeed, let $(\rho,\bfu)\in\xL_\mesh(\Omega)\times\xbfH_\mesh(\Omega)$ and define a constant density on the dual cells $\rho_\Ds$ as follows: 
\[
|D_\edge| \rho_\Ds= |D_{K,\edge}| \rho_K + |D_{L,\edge}| \rho_L \quad \text{for} ~ \edge=K|L.
\]
Then if $(\rho,\bfu)$ satisfy \eqref{eq:sch_mass}, one has:
\begin{equation}
\label{eq:sch_mass:dual}
 \sum_{\edged\in\edgesd(D_\edge)} F_{\edge,\edged}(\rho,\bfu) + h_\mesh^{\xi_1} |D_\edge|\, (\rho_{\Ds}-\rho^\star) = 0, \qquad \forall\edge\in\edgesint,
\end{equation}
which is an analogue of \eqref{eq:sch_mass} where the stabilization diffusion term is hidden in the dual fluxes.

\medskip
To complete the definition of the momentum convective term, we must give the expression of the velocity $\bfu_\edged$ at the dual face. As already said, a dual face lying on the boundary is also a primal face, and the flux across that face is zero.
Therefore, the values $\bfu_\edged$ are only needed at the internal dual faces; we choose them to be centered:
\[
\bfu_\edged = \frac 1 2 ( \bfu_\edge + \bfu_{\edge'}), \qquad \mbox{for } \edged=D_\edge|D_\edge'.
\]

%
%
\paragraph{Diffusion operator -- }

Let us define the shape functions associated with the Crouzeix-Raviart
finite element. These are the functions $(\zeta_\edge)_{\edge\in\edges}$ where for all $\edge\in\edges$, $\zeta_\edge$ is the element of $\xH_\mesh(\Omega)$ which satisfies:
\[
\dfrac{1}{|\edge'|}\int_{\edge'} \zeta_\edge \dedge(\bfx) =
\left \lbrace
\begin{array}{ll}
 1, \quad \text{if $\edge'=\edge$,}\\
 0, \quad \text{if $\edge'\neq\edge$.}\\
\end{array}
\right.
\]
Given a discrete velocity field $\bfu\in\xbfH_{\mesh,0}(\Omega)$, the discretization of the diffusion terms is given by:
\begin{equation}\label{eq:def_diff}
\begin{aligned}
& \lapi_\edges \bfu (\bfx)= 
 \sum_{\edge \in \edgesint} \frac{1}{|D_\edge|} \Big (\sum_{K \in \mesh} \int_K \gradi\bfu \, . \gradi \zeta_\edge \dx \Big ) \, \Ind_{\Ds}(\bfx), \\[2ex]
& (\gradi\circ\dive)_\edges\bfu(\bfx)= 
 \sum_{\edge \in \edgesint} \frac{1}{|D_\edge|} \Big ( \sum_{K \in \mesh} \int_K \dive \, \bfu \gradi \zeta_\edge \dx \Big ) \, \Ind_{\Ds}(\bfx).
\end{aligned}
\end{equation}

\paragraph{Pressure gradient operator --}
Given a discrete density field $\rho\in\xL_\mesh(\Omega)$, 
the pressure gradient term is discretized as follows:
\begin{equation}
\label{eq:def_grad_p}
\gradi_\edges (\rho^\gamma) (\bfx)= \sum_{\substack{\edge\in\edgesint \\ \edge=K|L}} \Big ( \frac{|\edge|}{|D_\edge|} \ (\rho_L^\gamma-\rho_K^\gamma)\ \bfn_{K,\edge} \Big ) \, \Ind_{\Ds}(\bfx).
\end{equation}
The discrete momentum equation \eqref{eq:sch_mom} also involves a third stabilization term, an artificial pressure term, which reads: 
\[
h_\mesh^{\xi_3}\,\gradi_\edges (\rho^\Gamma) (\bfx)=  h_\mesh^{\xi_3}\,\sum_{\substack{\edge\in\edgesint \\ \edge=K|L}} \Big ( \frac{|\edge|}{|D_\edge|} \ (\rho_L^\Gamma-\rho_K^\Gamma)\ \bfn_{K,\edge} \Big ) \, \Ind_{\Ds}(\bfx),
\]
where $\Gamma>\gamma$ is chosen large enough to ensure a control on the discrete weak formulation of the convective term in the momentum equation when $\gamma\in(\frac 32,3]$. Note that, if $d=3$ and $\gamma>3$ or $d=2$ and $\gamma>2$, this term is not needed.

\paragraph{Source term --} 

The source term $\bff \in \xbfL^2(\Omega)$ is discretized with the following projection operator:
\begin{equation}
\widetilde{\Pi}_\edges\bff(\bfx) = \dsp \sum_{\edge\in\edgesint} \Big(\dfrac{1}{|D_\edge|}\int_{D_\edge} \bff \dx \Big) \Ind_{\Ds}(\bfx).
\end{equation}

\subsection{Main result: convergence of the scheme}

\begin{defi}[Regular sequence of discretizations]\label{def:reg_disc}
A sequence $(\disc\n)_{n\in\xN}$ of staggered discretizations is said to be regular if:
\begin{itemize}
\item[$(i)$] there exists $\theta_0>0$ such that $\theta_{\mesh\n} \leq \theta_0$ for all $n\in\xN$,
\item[$(ii)$] the sequence of space steps $(h_{\mesh\n})_{n\in\xN}$ tends to zero as $n$ tends to $+ \infty$.
\end{itemize}
\end{defi}

\medskip
For the clarity of the presentation, we state our convergence result in the same setting as for the continuous problem, namely for $d=3$ and $\gamma \in (\frac 32,3]$.
We refer to the remark below for the ``simpler'' cases $d=2$, and $d=3$ with $\gamma > 3$.
 
\begin{thrm}[Convergence of the scheme]
\label{main_thrm}
Let $\Omega$ be a polyhedral connected open subset of $\xR^3$. 
Let $\bff\in\xbfL^2(\Omega)$ and $\rho^\star >0$. 
Assume that $\gamma \in (\frac 32, 3]$.
Denoting $\eta = \frac{2\gamma-3}{\gamma}\in(0,1]$, assume that $\Gamma$ and $(\xi_1,\xi_2,\xi_3)$ satisfy:
\begin{align}
\label{thm:cond_xi1}
(i) \qquad &\xi_1 > 1,\\[2ex]
\label{thm:cond_xi3}
(ii) \qquad & \frac{5}{4\Gamma}\left(\frac{3}{1+\eta}+\xi_3\right) < \frac{\eta}{1+\eta},\\[2ex]
\label{thm:cond_xi2}
(iii) \qquad & \frac 1\eta+\frac{5}{4\eta\Gamma}\left(\frac{3}{1+\eta}+\xi_3\right) 
< \xi_2 < \frac{1+\eta}{\eta} - \frac{5}{4\Gamma}\left(\frac{3}{1+\eta}+\xi_3\right).
\end{align}
Let $(\disc\n)_{n\in\xN}$ be a regular sequence of staggered discretizations of $\Omega$ as defined in Definition \ref{def:reg_disc}.
Then there exists $N\in\xN$ such that for all $n\geq N$, there exists a solution $(\rho\n,\bfu\n)\in\xL_{\mesh\n}(\Omega)\times\xbfH_{\mesh\n,0}(\Omega)$ to the numerical scheme \eqref{eq:sch} with the discretization $\disc\n$
and the obtained density $\rho\n$ is positive on $\Omega$.
Moreover, there exist $(\rho,\bfu)\in\xL^{3(\gamma-1)}(\Omega)\times \xbfH^{1}_0(\Omega)$ and a subsequence of $(\rho\n,\bfu\n)_{n\geq N}$, denoted $(\rho\n,\bfu\n)_{n\in\xN}$ such that:
\begin{itemize}
\item The sequence $(\bfu\n)_{n\in\xN}$ converges to $\bfu$ in $\xbfL^q(\Omega)$ for all $q\in[1,6)$,
\item The sequence $(\rho\n)_{n\in\xN}$ converges to $\rho$ in $\xL^q(\Omega)$ for all $q\in[1,3(\gamma-1))$ and weakly in $\xL^{3(\gamma-1)}(\Omega)$,
\item The sequence $(\rho\n^\gamma)_{n\in\xN}$ converges to $\rho^\gamma$ in $\xL^q(\Omega)$ for all $q\in[1,\frac{3(\gamma-1)}{\gamma})$ and weakly in $\xL^{\frac{3(\gamma-1)}{\gamma}}(\Omega)$,
\item The pair $(\rho,\bfu)$ is a weak solution of Problem \eqref{eq:pb}-\eqref{eq:pb_CL}-\eqref{hyp:mass} with finite energy.
\end{itemize}
\end{thrm}

\medskip
\begin{rmrk}[Some remarks on Theorem \ref{main_thrm}] \ 

\begin{itemize}
\item 
Let us mention that the convergence result of Theorem \ref{main_thrm} can be extended to the cases $d=3$, $\gamma>3$ and $d=2$, $\gamma>2$ with the mass stabilization term defined as 
	\begin{equation*}
	- h_\mesh^{\xi_2} \, \Delta_\mesh \rho (\bfx)  
	= h_\mesh^{\xi_2} \, \sum_{K\in\mesh} \, \frac 1 {|K|} \Big (\sum_{\substack{\edge \in\edges(K)\cap \edgesint \\ \edge= K|L}} |\edge| \,\dfrac{|\edge|}{|D_\edge|}\,(\rho_K-\rho_L) \Big )\, \Ind_K(\bfx),
	\end{equation*}
 and without the artificial pressure term. The required constraints on $(\xi_1,\xi_2)$ are the following:
\[
\begin{aligned}
 & \xi_1 > 1 \text{\ if $d=3$} \qquad \text{and} \qquad \xi_1 > 0 \text{\ if $d=2$},\\ 
 &\frac 32 < \xi_2 <2.
\end{aligned}
\]
In the case $d=2$, $1<\gamma\leq 2$, we expect a convergence result with the stabilization term proposed in \cite{eym-10-conv-isen} combined with an artificial pressure term.


\medskip
\item The upper bound on $\xi_2$ is required when passing to the limit in the effective viscous flux at the discrete level (see Subsection \ref{sec:cvf-eff-flux-d} and \eqref{eq:csq-hyp-xi2}). 
The lower bound on $\xi_2$ is required for the control on the momentum convective term when deriving the discrete estimate on the density (see \eqref{eq:convconv:3D}), which explains why this constraint was not introduced in \cite{eym-10-conv-isen} for the Stokes equations. 
\end{itemize}

\end{rmrk}

\medskip
The following sections are devoted to the proof of Theorem \ref{main_thrm}. In Section \ref{sec:norms-prop}, we introduce some notations and properties of the discretization. In Sections \ref{sec:renorm} to \ref{sec:sch:exist}, we derive \emph{a priori} estimates on the solution of the scheme and prove its existence provided a small enough space step $h_\mesh$. Finally, in Section \ref{sec:pass-limit}, we prove Theorem \ref{main_thrm} by successively passing to the limit in the discrete mass and momentum equations, and then in the equation of state.

\section{Mesh independent estimates and existence of a discrete solution}
\label{sec:estimates}

\subsection{Discrete norms and properties}
\label{sec:norms-prop}

We gather in this section some preliminary mathematical results which are useful for the analysis of the numerical scheme. 
Similar results have been previously used by Gallou\"et \emph{et al.} in their study \cite{gal-09-conv} which also relies on a mixed FV-FE discretization. 
The interested reader is also referred to the books \cite{ern2013}, \cite{eym-00-book}, \cite{dro-18-gra} and to the appendix of \cite{gallouet2015}.

\medskip

We start with defining the piecewise smooth first order differential operators associated with the Crouzeix-Raviart non-conforming finite element representation of velocities $\bfu\in\xbfH_{\mesh}(\Omega)$ :
\begin{align}
\label{disc:grad} &
\gradi_\mesh \bfu(\bfx) =  \sum_{K\in\mesh} \gradi \bfu(\bfx)\Ind_K(\bfx), \\
\label{disc:div}  &
\dive_\mesh\,\bfu(\bfx) =  \sum_{K\in\mesh} \dive\, \bfu(\bfx)\Ind_K(\bfx),\\
\label{disc:rot}  &
\rot_\mesh \bfu(\bfx) =  \sum_{K\in\mesh} \rot\, \bfu(\bfx)\Ind_K(\bfx).
 \end{align}
Note that on each element $K\in\mesh$, $\gradi \bfu_{|K}\in \xR^d$ is actually a constant and the  divergence defined in \eqref{disc:div} matches the finite volume divergence defined in \eqref{div_mass} for $\rho \equiv 1$.

\medskip

We then define for $q\in [1,\infty)$ the broken Sobolev $\xW^{1,q}$ semi-norm $\norm{.}_{1,q,\mesh}$ associated with the Crouzeix-Raviart finite element representation of the discrete velocities. For any $\bfu \in \xbfH_{\mesh}(\Omega)$ it is given by:
\[
\norm{\bfu}_{1,q,\mesh}^q= \int_\Omega |\gradi_\mesh \bfu|^q\dx.
\]

\medskip

\medskip
\begin{lem}[Discrete Sobolev inequality] \label{lmm:injsobolev_br}
Let $\disc=(\mesh,\edges)$ be a staggered discretization of $\Omega$ such that $\theta_\mesh \leq \theta_0$ (where $\theta_\mesh$ is defined by \eqref{eq:reg}) for some positive constant $\theta_0$.
Then, for all $q\in[1,+\infty)$ if $d=2$ and for all $q\in [1,6]$ if $d=3$, there exists $C=C(q,d,\theta_0)>0$ such that:
\[
\norm{\bfu}_{\xbfL^q(\Omega)} \leq C\, \norm{\bfu}_{1,2,\mesh}, \quad \forall \bfu \in \xbfH_{\mesh,0}(\Omega).
\]
\end{lem}


\medskip
A consequence of this Sobolev embedding is a discrete Poincar\'e inequality. Note that the semi-norm $\norm{\bfu}_{1,2,\mesh}$ is in fact a norm on the space $\xbfH_{\mesh,0}(\Omega)$.

\medskip
\begin{lem}[Discrete Poincar\'e inequality] \label{lmm:poincare_brok}
Let $\disc=(\mesh,\edges)$ be a staggered discretization of $\Omega$ such that $\theta_\mesh \leq \theta_0$ (where $\theta_\mesh$ is defined by \eqref{eq:reg}) for some positive constant $\theta_0$.
Then there exists $C=C(d,\theta_0)$ such that
\[
\norm{\bfu}_{\xbfL^2(\Omega)} \leq C\, \norm{\bfu}_{1,2,\mesh}, \quad \forall \bfu \in \xbfH_{\mesh,0}(\Omega).
\]
\end{lem}

\medskip
It will be convenient in the analysis of the scheme to handle several representations of the discrete velocities. We define an interpolation operator $\xPi_\edges$ which associates a piecewise constant function over the cells of the dual mesh to any
function $\bfu\in\xbfH_{\mesh}(\Omega)$ as follows:
\begin{equation}
\label{def:piepsu}
\xPi_\edges \bfu (\bfx)= 
 \sum_{\edge \in \edgesint} \bfu_\edge \, \Ind_{\Ds}(\bfx). 
\end{equation}
The constant value of $\xPi_\edges \bfu$ over the cell $D_\edge$ is $\bfu_\edge$ defined in \eqref{DOF:u}.
The mapping $\bfu \mapsto \xPi_\edges \bfu $ is a one-to-one mapping which is continuous with respect to the $\xL^q$-norm, for all $q\in[1,+\infty]$.
Indeed, we have the following result.

\medskip

\begin{lem} \label{lmm:Lpeq}
Let $\disc=(\mesh,\edges)$ be a staggered discretization of $\Omega$ such that $\theta_\mesh \leq \theta_0$ (where $\theta_\mesh$ is defined by \eqref{eq:reg}) for some positive constant $\theta_0$.
Then, for all $q\in[1,+\infty]$, there exists a constant $C=C(q,d,\theta_0)$ such that:
\[
  \norm{\xPi_\edges \bfu}_{\xbfL^q(\Omega)} \leq C\, \norm{\bfu}_{\xbfL^q(\Omega)}.
\]
\end{lem}


\medskip
We also define a finite-volume type gradient for the velocities associated with the dual mesh. This gradient is somehow a vector version of the gradient $\gradi_\edges$ defined in \eqref{eq:def_grad_p} for scalar function in $\xL_\mesh(\Omega)$. For $\bfu\in\xbfH_{\mesh}(\Omega)$ and $K\in\mesh$, denote $\bfu_{K}=\sum_{\edge\in\edges(K)}\xi_K^\edge\,\bfu_\edge$ where $\xi_K^\edge$ is defined in \eqref{eq:xiksigma}. The finite-volume gradient of $\bfu$ is defined by:
\begin{equation}
\label{eq:def_grad_mesh_u}
\gradi_\edges\bfu (\bfx) = \sum_{\substack{\edge\in\edgesint \\ \edge=K|L}} \Big( \frac{|\edge|}{|D_\edge|} (\bfu_L-\bfu_K) \otimes \bfn_{K,\edge}\, \Big ) \Ind_{\Ds}(\bfx).
\end{equation}

\medskip
We also introduce the following other discrete $\xW^{1,q}$ semi-norm given for $\bfu \in \xbfH_{\mesh}(\Omega)$ by:
\[
\norm{\bfu}_{1,q,\edges}^q = \sum_{K \in \mesh} h_K ^{d-q} \sum_{\edge,\edge' \in \edges(K)}\ |\bfu_\edge-\bfu_{\edge'}|^q.
\]
This semi-norm may be shown to be equivalent, over a regular sequence of discretizations, to the usual finite volume $\xW^{1,q}$ semi-norm associated with the piecewise constant function $\xPi_\edges \bfu$. It is possible to prove that this semi-norm, as well as the semi-norm defined by the $\xL^q$ norm of $\gradi_\edges\bfu$ are controlled on a regular discretization by the finite-element $\xW^{1,q}$ semi-norm. Indeed, we have the following lemma.

\medskip

\begin{lem} \label{lmm:H1ns}
Let $\disc=(\mesh,\edges)$ be a staggered discretization of $\Omega$ such that $\theta_\mesh \leq \theta_0$ (where $\theta_\mesh$ is defined by \eqref{eq:reg}) for some positive constant $\theta_0$.
Then for all $q\in[1,+\infty)$ there exist two constants $C_1=C_1(q,d,\theta_0)$ and $C_2=C_2(q,d,\theta_0)$ such that:
\[
\norm{\gradi_\edges\bfu}_{\xbfL^q(\Omega)^d} \leq C_1\, \norm{\bfu}_{1,q,\edges}\leq C_2 \, \norm{\bfu}_{1,q,\mesh}, \quad \forall \bfu \in \xbfH_{\mesh}(\Omega).
\]
\end{lem}


\medskip
\begin{lem}[Inverse inequalities] \label{lmm:inv-ineq}
Let $\disc=(\mesh,\edges)$ be a staggered discretization of $\Omega$ such that $\theta_\mesh \leq \theta_0$ (where $\theta_\mesh$ is defined by \eqref{eq:reg}) for some positive constant $\theta_0$. Let $u$ be a function defined on $\Omega$ such that for all $K\in\mesh$, $u_{|K}$ belongs to a finite dimensional space of functions which is stable by affine transformation. Then, for all $q,p\in [1,+\infty]$, there exists $C=C(q,p,\theta_0,d)$ such that (with $1/\infty=0$):
\begin{equation}\label{eq:invineq1}
\norm{u}_{\xL^q(K)}\leq C\, h_K^{d\left(\frac 1q- \frac 1p \right)} \, \norm{u}_{\xL^p(K)}, \qquad \forall K\in\mesh.
\end{equation}
Hence, for all $p\in [1,+\infty)$, there exists $C=C(p,\theta_0)$ such that :
\begin{equation}\label{eq:invineq2}
\norm{u}_{\xL^\infty(\Omega)}\leq C\, h_\mesh^{- \frac dp} \, \norm{u}_{\xL^p(\Omega)}. 
\end{equation}
\end{lem}

%

\medskip
For $q \in [1,+\infty)$, we introduce a discrete semi-norm on $\xL_\mesh(\Omega)$ similar to the usual $\xW^{1,q}$ semi-norm used in the finite volume context:
\[
 \snorm{\rho}_{q,\mesh}^q := \norm{\gradi_\edges (\rho)}_{\xbfL^q(\Omega)}^q= \sum_{\substack{\edge\in\edgesint \\ \edge=K|L}}|D_\edge|\, \Big(\frac{|\edge|}{|D_\edge|}\Big )^q \, |\rho_K-\rho_L|^q.
\]

\medskip
It will be convenient in the analysis of the scheme to handle another representation of the discrete densities associated with the upwind discretization of the mass flux. For $\bfu\in\xbfH_{\mesh,0}(\Omega)$ we define an interpolation operator $\xP_\edges$ which associates a piecewise constant function over the cells of the dual mesh to any function $\rho\in\xL_\mesh(\Omega)$ as follows:
\begin{equation}
\label{def:pepsrho}
 \xP_{\edges} \rho (\bfx)= 
 \sum_{\edge \in \edgesint} \rho_\edge \, \Ind_{\Ds}(\bfx).
\end{equation}
The constant value of $\xP_\edges \rho$ over the cell $D_\edge$, $\edge=K|L\in\edgesint$ is $\rho_\edge$ the upwind value with respect to $\bfu_\edge$, \emph{i.e.} $\rho_\edge=\rho_K$ is $\bfu_\edge\cdot\bfn_{K,\edge} \geq 0$ and $\rho_\edge=\rho_L$ otherwise.

\medskip

\begin{lem} \label{lmm:Pedges}
Let $\disc=(\mesh,\edges)$ be a staggered discretization of $\Omega$ such that $\theta_\mesh \leq \theta_0$ (where $\theta_\mesh$ is defined by \eqref{eq:reg}) for some positive constant $\theta_0$. For all $q\in[1,+\infty]$, there exists a constant $C=C(q,\theta_0)$ such that:
\[
\norm{\xP_\edges\rho}_{\xL^q(\Omega)} \leq C \norm{\rho}_{\xL^q(\Omega)}, \quad \forall \rho \in \xL_{\mesh}(\Omega). 
\]
\end{lem}

\subsection{Positivity of the density and discrete renormalization property}
\label{sec:renorm}

We begin this subsection with the next proposition which states the positivity of the discrete density $\rho\in\xL_\mesh(\Omega)$ if $(\rho,\bfu)$ is a solution of the discrete mass balance \eqref{eq:sch_mass}. 

\medskip

\begin{prop}[Positivity of the density]
\label{prop:rho-pos}
Let $\disc=(\mesh,\edges)$ be a staggered discretization of $\Omega$.
Let $(\rho,\bfu)\in\xL_\mesh(\Omega)\times\xbfH_\mesh(\Omega)$ be a solution of the discrete mass balance \eqref{eq:sch_mass}.
For $K\in\mesh$, denote $\dive(\bfu)_K$ the constant value of $\dive_\mesh\,\bfu$ over $K$.
Then
\begin{equation*}
\rho_K \geq \bar \rho := \frac{\rho^\star}{1+ h_\mesh^{-\xi_1} \, \max \Big( 0, \max\limits_{K\in\mesh}\dive(\bfu)_K  \Big )} > 0 \quad \forall K \in \mesh.
\end{equation*}
\end{prop}

We skip the proof. A similar proof can be found in \cite{gal-18-conv} (Appendix A).

\medskip

Next, we state a discrete analogue of the renormalization property \eqref{eq:c-renormalized-eq} satisfied at the continuous level. The proof is given in Appendix \ref{sec:prop:renorm}.

\medskip
\begin{prop}[Discrete renormalization property]
\label{prop:renorm}
Let $\disc=(\mesh,\edges)$ be a staggered discretization of $\Omega$.
Let $(\rho,\bfu)\in\xL_\mesh(\Omega)\times\xbfH_{\mesh,0}(\Omega)$ satisfy the discrete mass balance \eqref{eq:sch_mass}. We have $\rho>0$ \emph{a.e.} in $\Omega$ (\emph{i.e.} $\rho_K>0$, $\forall K\in\mesh$). 
Then, for any $b \in \mathcal{C}^1([0,+\infty))$:
\begin{align}\label{eq:discr-renorm-loc}
\dive \big(b(\rho) \bfu \big)_K + \big(b'(\rho_K)\rho_K - b(\rho_K) \big)\dive (\bfu)_K +R^1_K + R^2_K + R^3_K = 0  \quad \forall K \in \mesh,
\end{align}
where
\[
\dive \big(b(\rho) \bfu \big)_K 
= \dfrac{1}{|K|} \sum_{\edge \in \edges(K)}{|\edge| \ b(\rho_\edge) \bfu_\edge \cdot \bfn_{K,\edge}} ,
\]
and
\begin{align*}
R^1_K & = \frac{1}{|K|}\sum_{\edge \in \edges(K)}|\edge| r_{K,\edge} \bfu_\edge \cdot n_{K,\edge} \quad \text{and} \quad 
		r_{K,\edge} = b'(\rho_K)(\rho_\edge-\rho_K) + b(\rho_K) - b(\rho_\edge), \\
R^2_K & =  h_\mesh^{\xi_2} \  b'(\rho_K) \frac{1}{|K|}\sum_{\edge\in\edges(K)}|\edge| \,\Big(\dfrac{|\edge|}{|D_\edge|}\Big)^\frac{1}{\eta}\,|\rho_K-\rho_L|^{\frac{1}{\eta}-1}\,(\rho_K-\rho_L), \\
R^3_K & = h_\mesh^{\xi_1} b'(\rho_K)(\rho_K - \rho^\star). 
\end{align*}

\medskip
Multiplying by $|K|$ and summing over $K \in \mesh$, it holds
\begin{align}\label{eq:discr-renorm-glob}
\int_{\Omega}{\big(b'(\rho)\rho - b(\rho) \big)\dive_\mesh\,\bfu \dx} + R^1_\edges + R^2_\edges + R^3_\mesh = 0,
\end{align}
with
\begin{align*}
R^1_\edges & =\sum_{\substack{\edge\in\edgesint \\ \edge=K|L}} |\edge| (r_{K,\edge}-r_{L,\edge}) \bfu_\edge \cdot \bfn_{K,\edge}, \\
R^2_\edges & =  h_\mesh^{\xi_2} \sum_{\substack{\edge\in\edgesint \\ \edge=K|L}} |\edge| \,\Big(\dfrac{|\edge|}{|D_\edge|}\Big)^\frac{1}{\eta}\,|\rho_K-\rho_L|^{\frac{1}{\eta}-1}\,(\rho_K-\rho_L)(b'(\rho_K)-b'(\rho_L)), \\
R^3_\mesh & = h_\mesh^{\xi_1} \sum_{K \in \mesh} |K| b'(\rho_K)(\rho_K - \rho^\star), 
\end{align*}
and if $b$ is convex then $R^{1,2}_\edges \geq 0$ and $R^3_\mesh \geq 0$.
\end{prop}

\medskip
\begin{rmrk}
\label{rmrk:renorm}
For $b(\rho) = \dfrac{1}{\beta-1} \rho^\beta$ with $\beta > 1$, the previous result gives
	\begin{equation}
	\label{eq:renorm}
	\int_\Omega \rho^\beta \, \dive_\mesh\,\bfu \, \dx +  R_{\edges,1}(\rho,\bfu)  +  R_{\edges,2}(\rho,\bfu)  \leq 0,
	\end{equation}
	with
	\[
	\begin{aligned}
	&R^1_\edges(\rho,\bfu) = \frac{\beta}{2} \sum_{\substack{\edge\in\edgesint \\ \edge=K|L}} |\edge|\,\min(\rho_K^{\beta-2},\rho_L^{\beta-2})\, (\rho_L-\rho_K)^2 \,|\bfu_\edge\cdot\bfn_{K,\edge}|
	\geq 0 ,\\[2ex]
	&R^2_\edges(\rho,\bfu) =
	h_\mesh^{\xi_2}\,  \frac{\beta}{\beta-1} \sum_{\substack{\edge\in\edgesint \\ \edge=K|L}}  |\edge| \,\Big(\dfrac{|\edge|}{|D_\edge|}\Big)^\frac{1}{\eta}\,|\rho_K-\rho_L|^{\frac{1}{\eta}-1}\,(\rho_K-\rho_L)(\rho_K^{\beta-1}-\rho_L^{\beta-1})
	\geq 0.
	\end{aligned}
	\]
\end{rmrk}

\medskip
\begin{rmrk}{\label{rmrk:cond-b}}
	As explained in the continuous case, we can extend the renormalization result of Proposition \ref{prop:renorm} to functions $b \in \mathcal{C}^0([0,+\infty)) \cap \mathcal{C}^1((0,+\infty))$, under the additional assumption that for all $t \leq 1$:
	\[
	|b'(t)| \leq C t^{-\lambda_0} \quad \text{for some} ~ \lambda_0 < 1.
	\]
\end{rmrk}


\subsection{Estimate on the discrete velocity}
\label{sec:est:u}

In order to derive estimates on the discrete velocity and density, we
begin with writing a discrete counterpart of the weak formulation of the momentum balance. 
The following lemma states discrete counterparts to classical Stokes formulas. We refer to Sections \ref{sec:scheme} and \ref{sec:norms-prop} for the definitions of the operators.

\medskip
\begin{lem}
Let $\disc=(\mesh,\edges)$ be a staggered discretization of $\Omega$. The following discrete integration by parts formulas are satisfied for all $(p,\bfu)\in\xL_\mesh(\Omega)\times\xbfH_{\mesh,0}(\Omega)$. One has for all $\bfv \in \xbfH_{\mesh,0}(\Omega)$:
\begin{align}
\label{ipp:diff1}  - \int_{\Omega}\lapi_\edges \bfu \cdot \xPi_\edges \bfv \dx  & = \int_\Omega\gradi_\mesh \bfu:\gradi_\mesh \bfv \dx, \\
\label{ipp:diff2} - \int_{\Omega}(\gradi\circ\dive)_\edges\bfu \cdot \xPi_\edges \bfv \dx & = \int_\Omega \dive_\mesh\,\bfu \, \dive_\mesh\,\bfv \dx,
\\
\label{eq:dual_divgrad}
 \int_\Omega \gradi_\edges (p) \cdot \Pi_\edges \bfv \dx & = - \int_\Omega p \, \dive_\mesh\,\bfv \dx. 
\end{align}
\end{lem}

\medskip
Thanks to these formulas we easily show the next lemma which corresponds to a discrete counterpart of the weak formulation of the momentum equation.

\begin{lem}[Weak formulation of the momentum balance - first form]
Let $\disc=(\mesh,\edges)$ be a staggered discretization of $\Omega$. 
A pair $(\rho,\bfu)\in\xL_\mesh(\Omega)\times\xbfH_\mesh(\Omega)$ satisfies the discrete momentum balance \eqref{eq:sch_mom} if and only if:
\begin{multline}
\label{eq:sch_mom:weak}
\int_\Omega \divv_\edges(\rho \bfu\otimes \bfu)\cdot \xPi_\edges \bfv\dx +\mu \int_\Omega\gradi_\mesh \bfu:\gradi_\mesh \bfv \dx + (\mu+\lambda)\int_\Omega \dive_\mesh\, \bfu \, \dive_\mesh \,\bfv \dx \\[2ex]
- a \int_\Omega \rho^\gamma\, \dive_\mesh\,\bfv \dx - h_\mesh^{\xi_3}\int_\Omega \rho^\Gamma\, \dive_\mesh\,\bfv \dx = \int_\Omega \bff\cdot\xPi_\edges \bfv \dx, \qquad \forall \bfv\in\xbfH_{\mesh,0}(\Omega).
\end{multline}
\end{lem}


\medskip
From this point, we assume that $\Gamma$ and $(\xi_1,\xi_2,\xi_3)$ satisfy the conditions \eqref{thm:cond_xi1}-\eqref{thm:cond_xi3}-\eqref{thm:cond_xi2}.


\medskip
\begin{prop}[Estimate on the discrete velocity]
\label{prop:estimate:u}
Let $(\rho,\bfu)\in\xL_\mesh(\Omega)\times\xbfH_{\mesh,0}(\Omega)$ be a solution of the numerical scheme \eqref{eq:sch}. Then, we have $\rho>0$ \emph{a.e.} in $\Omega$ (\emph{i.e.} $\rho_K>0$, $\forall K\in\mesh$),
and if $h_\mesh \leq h_0$ (with $h_0$ depending on $\mu,\rho^\star,\Omega,\theta_0$), there exists $C_1=C_1(\bff,\mu,\rho^\star,\Omega,\xi_1,\theta_0)$ such that:
\begin{equation}
\label{estimates:u}
\norm{\bfu}_{1,2,\mesh}^2\leq C_1.
\end{equation}
\end{prop}

\medskip

\begin{proof}
We take $\bfv=\bfu$ as a test function in \eqref{eq:sch_mom:weak}:
\begin{multline*}
\int_\Omega \divv_\edges(\rho \bfu\otimes \bfu)\cdot \xPi_\edges \bfu\dx + \mu \norm{\bfu}^2_{1,2,\mesh}+(\mu+\lambda) \norm{\dive_\mesh\,\bfu}_{\xL^2(\Omega)}^2 \\[2ex]
- a\int_\Omega \rho^\gamma\, \dive_\mesh\,\bfu \dx - h_\mesh^{\xi_3}\int_\Omega \rho^\Gamma\, \dive_\mesh\,\bfu \dx
= \int_\Omega \bff \cdot \xPi_\edges \bfu \dx.
\end{multline*}
Applying Remark \ref{rmrk:renorm} on discrete renormalization with $\beta=\gamma>1$ and $\beta=\Gamma>1$, the last two terms in the left hand side of this equality are seen to be non-negative. We thus obtain:
\begin{equation}
\label{estimates:u:proof1}
\int_\Omega \divv_\edges(\rho \bfu\otimes \bfu)\cdot \xPi_\edges \bfu\dx + \mu \norm{\bfu}^2_{1,2,\mesh}\leq \norm{\bff}_{\xbfL^2(\Omega)}\,\norm{\xPi_\edges \bfu}_{\xbfL^2(\Omega)}. 
\end{equation}
Recalling that in the definition of the convection term, $\bfu_\edged=\frac 12(\bfu_\edge+\bfu_{\edge'})$ for $\edged=D_\edge|D_{\edge'}$, we get:
\begin{multline*}
 \int_\Omega \divv_\edges(\rho \bfu\otimes \bfu)\cdot \xPi_\edges \bfu\dx = \\ \frac 12 \sum_{\edge\in\edgesint} \,\Big ( \sum_{\edged\in\edgesd(D_\edge)} \fluxd(\rho,\bfu) \Big) |\bfu_\edge|^2 
+ \frac 12 \sum_{\edge\in\edgesint} \, \Big ( \sum_{\substack{\edged\in\edgesd(D_\edge)\\ \edged=D_\edge|D_{\edge'}}} \fluxd(\rho,\bfu)\ \bfu_\edge\cdot\bfu_{\edge'} \Big ),
\end{multline*}
and the last term in the right hand side vanishes thanks to the conservativity of the dual fluxes (assumption (H2)). Using the mass conservation equation satisfied on the dual mesh \eqref{eq:sch_mass:dual} in the first term, we get (denoting $\tilde \rho$ the piecewise constant scalar function which is equal to $\rho_\Ds$ on every dual cell $D_\edge$, and which satisfies $\tilde \rho>0$ (because $\rho>0$) and $\int_\Omega \tilde \rho\dx=\int_\Omega \rho\dx = |\Omega|\rho^\star$):
\[
 \Big | \int_\Omega \divv_\edges(\rho \bfu\otimes \bfu)\cdot \xPi_\edges \bfu\dx  \Big | =  \frac12 \, h_\mesh^{\xi_1} \,  \Big |  \int_{\Omega}(\tilde \rho-\rho^\star)\, |\xPi_\edges \bfu|^2 \dx \Big | \leq h_\mesh^{\xi_1} \, |\Omega|\,\rho^\star \, \norm{\xPi_\edges \bfu}_{\xbfL^\infty(\Omega)}^2.
\]
Injecting in \eqref{estimates:u:proof1} yields:
\begin{equation*}
\mu \norm{\bfu}^2_{1,2,\mesh}
\leq \norm{\bff}_{\xbfL^2(\Omega)}\,\norm{\xPi_\edges \bfu}_{\xbfL^2(\Omega)} + h_\mesh^{\xi_1} |\Omega|\,\rho^\star \norm{\xPi_\edges \bfu}_{\xbfL^\infty(\Omega)}^2.
\end{equation*}
Thanks to the continuity of operator $\xPi_\edges$: $\norm{\xPi_\edges \bfu}_{\xbfL^q(\Omega)}\leq C \norm{\bfu}_{\xbfL^q(\Omega)}$, the inverse inequality $\norm{\bfu}_{\xbfL^\infty(\Omega)} \leq  h_\mesh^{-\frac 12}\norm{\bfu}_{\xbfL^6(\Omega)}$, and the discrete Sobolev inequality $\norm{\bfu}_{\xbfL^6(\Omega)}\leq C\norm{\bfu}_{1,2,\mesh}$ (with $C$ only depending on $\Omega$ and $\theta_0$), we obtain:
\begin{equation*}
\mu \norm{\bfu}^2_{1,2,\mesh}
\leq C \Big ( \norm{\bff}_{\xbfL^2(\Omega)}\,\norm{\bfu}_{1,2,\mesh} + h_\mesh^{\xi_1- 1} |\Omega|\,\rho^\star \norm{ \bfu}_{1,2,\mesh}^2 \Big ).
\end{equation*}
Applying Young's inequality, we get that for all $\kappa>0$:
\begin{equation*}
(\mu-C\kappa-C  h_\mesh^{\xi_1-1} |\Omega|\,\rho^\star)\,\norm{\bfu}^2_{1,2,\mesh}
\leq \frac{C}{4\kappa}\,\norm{\bff}_{\xbfL^2(\Omega)}^2.
\end{equation*}
Since $\xi_1>1$, taking $h_\mesh$ and $\kappa$ small enough yields:
\[
 \norm{\bfu}_{1,2,\mesh}^2\leq C_1,
\]
where $C_1$ only depends on $\bff$, $\mu$, $\rho^\star$, $\Omega$, $\xi_1$ and $\theta_0$.

\end{proof}


\subsection{Estimates on the discrete density}
\label{sec:est:rho}

One remarkable property of the staggered discretization is the existence of a discrete analogue to the Bogovskii operator, which is also equivalent to an \emph{$L^q$ inf-sup} property satisfied by discrete functions 
(see for instance \cite{gallouet2012} for a proof which concerns the MAC scheme).

\medskip
\begin{lem}[Discrete $L^q$ inf-sup property] 
\label{lmm:inf-sup}
Let $\disc=(\mesh,\edges)$ be a staggered discretization of $\Omega$ such that $\theta_\mesh \leq \theta_0$ (where $\theta_\mesh$ is defined by \eqref{eq:reg}) for some positive constant $\theta_0$.
Then, there exists a linear operator
\[
\mcal{B}_\mesh: \xL_{\mesh,0}(\Omega) \longrightarrow \xbfH_{\mesh,0}(\Omega)
\]
depending only on $\Omega$ and on the discretization such that the following properties hold:
\begin{itemize}
 \item[(i)] For all $p\in\xL_{\mesh,0}(\Omega)$, 
\[
\int_\Omega r \,\dive_\mesh(\mcal{B}_\mesh p) \dx= \int_{\Omega} r\,p \dx, \qquad \forall r\in\xL_\mesh(\Omega).
\]
\item[(ii)] For all $q \in (1,+\infty)$, there exists $C=C(q,d,\Omega,\theta_0)$, such that
\[
 \norm{\mcal{B}_\mesh p}_{1,q,\mesh} \leq C\norm{p}_{\xL^q(\Omega)}.
\]
\end{itemize}
\end{lem}

\medskip
Before deriving the control of the discrete pressure, we first present a second form of the weak formulation of the momentum equation which will be more convenient to handle. 

\begin{lem}[Weak formulation of the momentum balance - second form]{\label{lem:d-weakmom2}}
The discrete weak formulation of the momentum balance \eqref{eq:sch_mom:weak} can be rewritten in the following form:
\begin{multline}
\label{eq:sch_mom:weak:2}
-\int_\Omega (\xP_\edges\rho)(\xPi_\edges\bfu)\otimes(\xPi_\edges\bfu): \gradi_\edges \bfv \dx \\
+ \mu \int_\Omega\gradi_\mesh \bfu:\gradi_\mesh \bfv \dx 
+ (\mu+\lambda)\int_\Omega \dive_\mesh\,\bfu \, \dive_\mesh\,\bfv \dx \\  
- a\int_\Omega \rho^\gamma\, \dive_\mesh\,\bfv \dx - h_\mesh^{\xi_3}\int_\Omega \rho^\Gamma\, \dive_\mesh\,\bfv \dx + R_{\rm conv}(\rho,\bfu,\bfv)= \int_\Omega \bff\cdot\xPi_\edges \bfv \dx,
\end{multline}
where
\[
 R_{\rm conv}(\rho,\bfu,\bfv)= \int_\Omega \divv_\edges(\rho \bfu\otimes \bfu)\cdot \xPi_\edges \bfv\dx+\int_\Omega (\xP_\edges\rho)(\xPi_\edges\bfu)\otimes(\xPi_\edges\bfu): \gradi_\edges \bfv \dx.
\]
The remainder term satisfies the following estimate for some constant  $C=C(\Omega,\gamma,\Gamma,\theta_0)$:
\begin{align}\label{eq:convconv:3D}
\big | R_{\rm conv}(\rho,\bfu,\bfv) \big | 
& \leq C \, h_\mesh^{\frac 12 - \frac{1}{ \Gamma}(\frac{3}{1+\eta} + \xi_3)}
	 \norm{h_\mesh^{\xi_3} \rho^\Gamma}_{\xL^{1+\eta}(\Omega)}^{\frac{1}{\Gamma}}  \norm{\bfu}_{1,2,\mesh}^2\,  \norm{\bfv}_{1,2,\mesh} \\[2ex]
&\quad + C \,  h_\mesh^{\xi_2 -\frac 1\eta- \frac{1}{\eta \Gamma}(\frac{3}{1+\eta} + \xi_3)} \norm{h_\mesh^{\xi_3}\rho^\Gamma}_{\xL^{1+\eta}(\Omega)}^{\frac{1}{\eta \Gamma}}  \norm{\bfu}_{1,2,\mesh}\,  \norm{\bfv}_{1,2,\mesh}. \nonumber
\end{align}
\end{lem}

\medskip
\begin{proof}
This result is proved in Appendix \ref{sec:proof:lmm:estimates_Bbar}.
\end{proof}

\medskip
\begin{rmrk}{\label{rmrk:conrol_Rconv}}
Note that in the previous inequality \eqref{eq:convconv:3D}, under the conditions \eqref{thm:cond_xi3}-\eqref{thm:cond_xi2} (since $\frac{\eta}{1+\eta} \leq \frac{1}{2}$), we guarantee that:
\begin{align*}
& \frac 12 - \frac{1}{ \Gamma}\left(\frac{3}{1+\eta} + \xi_3\right) > \frac{\eta}{1+\eta} - \frac{5}{ 4\Gamma}\left(\frac{3}{1+\eta} + \xi_3\right) , \\
& \xi_2 -\frac 1\eta- \frac{1}{\eta \Gamma}\left(\frac{3}{1+\eta} + \xi_3\right) > \xi_2 -\frac 1\eta- \frac{5}{4\eta \Gamma}\left(\frac{3}{1+\eta} + \xi_3\right),
\end{align*}
so that the exponents of $h_{\mesh}$ appearing in \eqref{eq:convconv:3D} are positive under the assumptions \eqref{thm:cond_xi3}-\eqref{thm:cond_xi2}.
\end{rmrk}

\medskip
We may now prove the following result which states mesh independent estimates satisfied by the discrete density when $(\rho,\bfu)\in\xL_\mesh(\Omega)\times\xbfH_{\mesh,0}(\Omega)$ is a solution of the numerical scheme \eqref{eq:sch}.

\medskip
\begin{prop}
\label{prop:estimate:p}

\medskip
Let $(\rho,\bfu)\in\xL_\mesh(\Omega)\times\xbfH_{\mesh,0}(\Omega)$ be a solution of the numerical scheme \eqref{eq:sch}. 
Then,
we have the following estimates:
\begin{itemize}
\item 
There exists $C_2=C_2(\bff,\mu,\lambda,\rho^\star,\Omega,\gamma,\Gamma,\xi_1,\xi_2,\xi_3,\theta_0)$ such that:
\begin{equation}
\label{estimates:p}
\norm{\rho}_{\xL^{3(\gamma-1)}(\Omega)}
+ \norm{h_\mesh^{\xi_3}\rho^\Gamma}_{\xL^{1 + \eta}(\Omega)} 
  \leq  C_2.
\end{equation}
\item
There exists $C_3=C_3(\bff,\mu,\lambda,\rho^\star,\Omega,\gamma,\Gamma,\xi_1,\xi_2,\xi_3,\theta_0)$ 
such that:
\begin{equation}
\label{estimates2}
\sum_{\substack{\edge\in\edgesint \\ \edge=K|L}} |\edge|\, (\rho_L-\rho_K)^2 \,|\bfu_\edge\cdot\bfn_{K,\edge}| + h_\mesh^{\xi_2}\,\snorm{\rho}_{\frac{1+\eta}{\eta},\mesh}^{\frac{1+\eta}{\eta}}
  \leq   C_3 \, h_\mesh^{- \frac{5}{4\Gamma}\left(\frac{3}{1+\eta}+\xi_3\right)}.
\end{equation}
\end{itemize}
\end{prop}


\begin{rmrk}
Note that from \eqref{estimates:p} we can easily deduce by interpolation between Lebesgue spaces that for all $p$ with $1\leq p < 1+\eta$, there exists  $r\in[0,1)$ depending on $p$ and $\gamma$ (with $r=0$ if $p=1$) such that:
\begin{align}
\label{estimates:artifp}
 \norm{h_\mesh^{\xi_3} \rho^\Gamma}_{\xL^{p}(\Omega)} 
 & \leq h_\mesh^{\xi_3} \norm{\rho^\Gamma}_{\xL^1(\Omega)}^{1-r}\norm{\rho^\Gamma}_{\xL^{1+\eta}(\Omega)}^{r} 
 \leq h_\mesh^{\xi_3(1-r)} \norm{\rho^\Gamma}_{\xL^1(\Omega)}^{1-r} \norm{h_\mesh^{\xi_3}\rho^\Gamma}_{\xL^{1+\eta}(\Omega)}^{r}
 \leq C_4 \, h_\mesh^{\xi_3(1-r)},
\end{align}
with $C_4=C_4(\bff,\mu,\lambda,\rho^\star,\Omega,\gamma,\Gamma,\xi_1,\xi_2,\xi_3,\theta_0,p)$.
\end{rmrk}	


\medskip

\begin{proof}
Let us set $P(\rho) = a \rho^\gamma + h_\mesh^{\xi_3} \rho^\Gamma$.
We apply Lemma \ref{lmm:inf-sup} to $P(\rho)^{\eta}-<P(\rho)^\eta>$ and we define 
$\bfv\in\xbfH_{\mesh,0}(\Omega)$ by $\bfv=\mcal{B}_\mesh (P(\rho)^{\eta}-<P(\rho)^\eta>)$. There exists $C = C(\Omega,\gamma,\theta_0)$ such that
\begin{align*}
\norm{\bfv}_{1,2,\mesh}
& \leq C \norm{P(\rho)^\eta - <P(\rho)^\eta>}_{\xL^2(\Omega)} \\
& \leq C \Big(\norm{P(\rho)}_{\xL^{2\eta}(\Omega)}^\eta + \norm{P(\rho)}_{\xL^\eta(\Omega)}^\eta \Big) \\
& \leq C \norm{P(\rho)}_{\xL^{1+\eta}(\Omega)}^\eta
\end{align*}
since $\eta = \frac{2\gamma-3}{\gamma} \leq 1$ for $\gamma \leq 3$.
With the same arguments, we have
\begin{align*}
\norm{\bfv}_{1,\frac{1+\eta}{\eta},\mesh}
& \leq C \norm{P(\rho)}_{\xL^{1+\eta}(\Omega)}^\eta.
\end{align*}
Taking $\bfv$ as a test function in \eqref{eq:sch_mom:weak:2}, we obtain:
\begin{align}\label{eq:est-discr-p-0}
\int_\Omega (P(\rho))^{1+\eta }\dx 
&=  <P(\rho)^\eta>  \left(\int_\Omega P(\rho)\dx\right) \nonumber \\
& -\int_\Omega (\xP_\edges\rho)(\xPi_\edges\bfu)\otimes(\xPi_\edges\bfu): \gradi_\edges \bfv \dx 
+ \mu \int_\Omega\gradi_\mesh \bfu:\gradi_\mesh \bfv \dx\nonumber \\
&+ (\mu+\lambda)\int_\Omega \dive_\mesh\,\bfu \, \dive_\mesh\,\bfv \dx 
- \int_\Omega \bff\cdot\xPi_\edges \bfv\dx + R_{\rm conv}(\rho,\bfu,\bfv) \nonumber\\
&= T_1 + \dots + T_6. 
\end{align}
We estimate the $T_i$ as follows. 
First for $T_1$ we have $C = C(a,\Omega,\gamma,\theta_0)$ such that:
\begin{align*}
|T_1| 
& \leq C \Big(\int_{\Omega} P(\rho) \Big) \norm{P(\rho)}_{\xL^{1+\eta}(\Omega)}^{\eta} \\
& \leq C \Big(
\norm{\rho}_{\xL^1(\Omega)}^{\gamma(1-r_1)} \norm{\rho}_{\xL^{\gamma(1+\eta)}(\Omega)}^{\gamma r_1} 
+ h\n^{\xi_3}\norm{\rho}_{\xL^1(\Omega)}^{\Gamma(1-r_2)} \norm{\rho}_{\xL^{\Gamma(1+\eta)}(\Omega)}^{\Gamma r_2} 
\Big)  \norm{P(\rho)}_{\xL^{1+\eta}(\Omega)}^{\eta} \\
& \leq C \Big(\norm{\rho^\gamma}_{\xL^{1+\eta}(\Omega)}^{r_1} 
+ h\n^{\xi_3(1-r_2)} \norm{h\n^{\xi_3}\rho^\Gamma}_{\xL^{1+\eta}(\Omega)}^{r_2} 
\Big)  \norm{P(\rho)}_{\xL^{1+\eta}(\Omega)}^{\eta}
\end{align*}
where we have used an interpolation inequality with
\[
r_1 = \frac{(\gamma-1)(1+\eta)}{\gamma(1+\eta) -1} < 1, \quad 
r_2 = \frac{(\Gamma-1)(1+\eta)}{\Gamma(1+\eta) -1} < 1.
\]
Hence, by a Young inequality, we have $C = C(\Omega,\gamma,\Gamma, \xi_3,\theta_0)$  such that:
\begin{align*}
|T_1|
& \leq C + \frac{1}{5} \norm{P(\rho)}_{\xL^{1+\eta}(\Omega)}^{1+\eta} .
\end{align*}
The second term is controlled as follows, with  $C = C(\Omega,\gamma,\theta_0)$:
\begin{align*}
|T_2| 
& = \Big| \int_\Omega (\xP_\edges\rho)(\xPi_\edges\bfu)\otimes(\xPi_\edges\bfu): \gradi_\edges \bfv \dx \Big | \\
&\leq C\, \norm{\xP_\edges\rho}_{\xL^{\gamma(1+\eta)}(\Omega)} \, \norm{\xPi_\edges\bfu}_{\xbfL^{6}(\Omega)}^2 \, \norm{\bfv}_{1,\frac{1+\eta}{\eta},\mesh}\\
&\leq  C\, \norm{P(\rho)}_{\xL^{1+\eta}(\Omega)}^{1/\gamma} \, \norm{\bfu}_{\xbfL^{6}(\Omega)}^2 \, \norm{P(\rho)}_{\xL^{1+\eta}(\Omega)}^\eta \\
& \leq C + \frac{1}{5} \norm{P(\rho)}_{\xL^{1+\eta}(\Omega)}^{1+\eta}.
\end{align*}
Next, observing that $\norm{\dive_\mesh\,\bfv}_{\xL^2(\Omega)}\leq \sqrt{3}\,\norm{\bfv}_{1,2,\mesh}$ for all $\bfv\in\xbfH_{\mesh}(\Omega)$, we have $C = C(\Omega,\gamma,\theta_0)$ such that:
\[
|T_3| + |T_4|
\leq (\mu+3(\mu+\lambda))\, \norm{\bfu}_{1,2,\mesh}\, \norm{\bfv}_{1,2,\mesh}
\leq C \, C_1\, (\mu+3(\mu+\lambda))\,\norm{P(\rho)}_{\xL^{1+\eta}(\Omega)}^\eta.
\]
By the discrete Poincar\'e inequality, the term $T_5$ satisfies with $C = C(\Omega,\gamma,\theta_0)$:
\[
|T_5| 
\leq \norm{\bff}_{\xbfL^2(\Omega)}\,\norm{\xPi_\edges \bfv}_{\xbfL^2(\Omega)} 
\leq C\,\norm{\bff}_{\xbfL^2(\Omega)}\,\norm{\bfv}_{\xbfL^2(\Omega)} 
\leq C \, \norm{\bff}_{\xbfL^2(\Omega)}\,\norm{\bfv}_{1,2,\mesh} 
\leq C \, \norm{\bff}_{\xbfL^2(\Omega)}\,\norm{P(\rho)}_{\xL^{1+\eta}(\Omega)}^\eta.
\]
Hence we get:
\[
 |T_3| + |T_4| + |T_5|\leq  C + \frac{1}{5} \norm{P(\rho)}_{\xL^{1+\eta}(\Omega)}^{1+\eta}.
\]
The last term $T_6$ is the remainder term $R_{\rm conv}(\rho,\bfu,\bfv)$ in the weak formulation of the momentum balance \eqref{eq:sch_mom:weak:2}. 
We have thanks to \eqref{eq:convconv:3D}, with  $C=(\Omega,\gamma,\Gamma,\theta_0)$:
\[
\begin{aligned}
|T_6| = |R_{\rm conv}(\rho,\bfu,\bfv)|
& \leq C \, h_\mesh^{\frac 12 -\frac{1}{ \Gamma}(\frac{3}{1+\eta} + \xi_3)}
\norm{h_\mesh^{\xi_3} \rho^\Gamma}_{\xL^{1+\eta}(\Omega)}^{\frac{1}{\Gamma}}  \norm{\bfu}_{1,2,\mesh}^2\,  \norm{\bfv}_{1,2,\mesh} \\[2ex]
& + C \,  h_\mesh^{\xi_2 -\frac 1\eta -\frac{1}{\eta \Gamma}(\frac{3}{1+\eta} + \xi_3)}
\norm{h_\mesh^{\xi_3}\rho^\Gamma}_{\xL^{1+\eta}(\Omega)}^{\frac{1}{\eta \Gamma}}  \norm{\bfu}_{1,2,\mesh}\,  \norm{\bfv}_{1,2,\mesh}.
\end{aligned}
\]
As a consequence of Remark \ref{rmrk:conrol_Rconv}, there exists $\nu > 0$ such that
\begin{align*}
|T_6|
& \leq C h_\mesh^\nu \, \norm{\bfu}_{1,2,\mesh}^2 \,  \norm{h_\mesh^{\xi_3}\rho^\Gamma}_{\xL^{1+\eta}(\Omega)}^{\frac{1}{\Gamma}} \norm{\bfv}_{1,2,\mesh} 
+ C h_\mesh^\nu \, \norm{\bfu}_{1,2,\mesh}  \norm{h_\mesh^{\xi_3}\rho^\Gamma}_{\xL^{1+\eta}(\Omega)}^{\frac{1}{\eta \Gamma}}\norm{\bfv}_{1,2,\mesh}\\
& \leq C h_\mesh^\nu \, \norm{\bfu}_{1,2,\mesh}^2 \,  \norm{h_\mesh^{\xi_3}\rho^\Gamma}_{\xL^{1+\eta}(\Omega)}^{\frac{1}{\Gamma}}  \norm{P(\rho)}_{\xL^{1+\eta}(\Omega)}^{\eta} 
+ C h_\mesh^\nu \, \norm{\bfu}_{1,2,\mesh}  \norm{h_\mesh^{\xi_3}\rho^\Gamma}_{\xL^{1+\eta}(\Omega)}^{\frac{1}{\eta \Gamma}} \norm{P(\rho)}_{\xL^{1+\eta}(\Omega)}^{\eta}\\
& \leq C h_\mesh^\nu \, \norm{\bfu}_{1,2,\mesh}^2 \,  \norm{P(\rho)}_{\xL^{1+\eta}(\Omega)}^{\eta+\frac{1}{\Gamma}} 
+ C h_\mesh^\nu \, \norm{\bfu}_{1,2,\mesh}\, \norm{P(\rho)}_{\xL^{1+\eta}(\Omega)}^{\eta+\frac{1}{\eta \Gamma}}\\
& \leq C + \frac{1}{5} \norm{P(\rho)}_{\xL^{1+\eta}(\Omega)}^{1+\eta},
\end{align*}
since $\frac 1\Gamma$ and $\frac{1}{\eta \Gamma}$ are both less than $1$ (consequence of \eqref{thm:cond_xi3}). Gathering the bounds on $T_1$,..,$T_6$ and coming back to \eqref{eq:est-discr-p-0} we get:
\[
\int_{\Omega}{(P(\rho))^{1+\eta} \dx}
\leq C + \dfrac{4}{5} \int_{\Omega}{(P(\rho))^{1+\eta}\dx}.
\] 
This achieves the proof of \eqref{estimates:p}. 
		
	
\medskip
It remains to prove \eqref{estimates2}. 
Taking $\beta=2$ in the discrete renormalization identity \eqref{eq:renorm}, we get:
\begin{align*}
\sum_{\substack{\edge\in\edgesint \\ \edge=K|L}} |\edge|\, (\rho_L-\rho_K)^2 \,|\bfu_\edge\cdot\bfn_{K,\edge}|
+ h_\mesh^{\xi_2}\snorm{\rho}_{\frac{1+\eta}{\eta},\mesh}^{\frac{1+\eta}{\eta}} 
& \leq -\int_\Omega \rho^2\,\dive_\mesh\,\bfu\dx
\end{align*}
with

	\begin{align*}
	\left|\int_\Omega \rho^2\,\dive_\mesh\,\bfu\dx\right|
	& \leq \left|\int_{\Omega} \rho^{\frac{5}{4}} \, \rho^{\frac{3}{4}} \,\dive_\mesh \, \bfu \dx  \right| \\
	& \leq \norm{\rho}_{\xL^\infty(\Omega)}^{\frac{5}{4}} \norm{\rho}_{\xL^{\frac{3}{2}}(\Omega)}^{\frac{3}{4}} \norm{\dive_\mesh\,\bfu}_{\xL^{2}(\Omega)} \\
	& \leq C(\Omega,\gamma,\Gamma,\theta_0) h_\mesh^{-\frac{5}{4\Gamma}\big(\frac{3}{1+\eta}+\xi_3\big)} \norm{h_\mesh^{\xi_3}\rho^\Gamma}_{\xL^{1+\eta}(\Omega)}^{\frac{5}{4\Gamma}}
	\norm{\rho}_{\xL^{3(\gamma-1)}(\Omega)}^{\frac{3}{4}} \norm{\dive_\mesh\,\bfu}_{\xL^{2}(\Omega)}.
	\end{align*}

This achieves the proof of \eqref{estimates2}.
\end{proof}

\subsection{Existence of a solution to the numerical scheme}
\label{sec:sch:exist}

The existence of a solution to the scheme \eqref{eq:sch}, which consists in an algebraic non-linear system, is obtained by a topological degree argument.
Its proof is based on an abstract theorem stated for instance in \cite{gal-08-unc} (Theorem 2.5) 
which relies on linking by a homotopy the problem at hand to a linear system.

\medskip
Let $N={\rm card}(\mesh)$ and $M=d\ {\rm card}(\edgesint)$; we identify $\xL_\mesh(\Omega)$ with $\xR^{N}$ and $\xbfH_{\mesh,0}(\Omega)$ with $\xR^{M}$.
Let $V=\xR^{N} \times \xR^{M}$.
We consider the function $\mcal{F}:\ V\times [0,1] \rightarrow V$ given by:
\begin{equation}
\label{functionF}
\mcal{F}(\rho,\bfu,\delta)=
\left| \begin{array}{ll} \displaystyle
\delta  \frac 1 {|K|} \sum_{\edge \in\edges(K)} \overline{F}_{K,\edge}(\rho,\bfu) + h_\mesh^{\xi_1}\, (\rho_K-\rho^\star),
&
K \in \mesh
\\[4ex] \displaystyle
\delta  \frac 1 {|D_\edge|} \sum_{\edged \in\edgesd(D_\edge)} \fluxd(\rho,\bfu) \bfu_\edged + \delta \, a  (\gradi_\edges (\rho^\gamma))_{|\Ds}+ \delta \, h_\mesh^{\xi_3} \,(\gradi_\edges (\rho^\Gamma))_{|\Ds}
\\[4ex] \hspace{1cm} - \mu\,(\lapi_\edges\bfu)_{|\Ds} - (\mu+\lambda)\, \big((\gradi\circ\dive)_\edges\bfu \big)_{|\Ds}
-(\widetilde{\Pi}_\edges \bff)_{|\Ds},
 &
\edge \in \edgesint.
\end{array} \right. 
\end{equation}
Solving the problem $\mcal{F}(\rho,\bfu,\delta)=0$ is equivalent to solving the following system analogous to \eqref{eq:sch}:

\medskip
\emph{Solve for $\rho\in\xL_\mesh(\Omega)$ and $\bfu\in\xbfH_{\mesh,0}(\Omega)$:}
\begin{subequations}\label{eq:sch:delta}
\begin{align}
\label{eq:sch_mass:delta}  & 
\delta\, \dive_\mesh(\rho \bfu)  + h_\mesh^{\xi_1} \, (\rho-\rho^\star) - \delta\, h_\mesh^{\xi_2} \, \Delta_{\frac{1+\eta}{\eta},\mesh}(\rho)= 0, \\[2ex]  
\label{eq:sch_mom:delta}  &
\delta\, \divv_\edges(\rho \bfu\otimes \bfu) - \mu \lapi_\edges \bfu - (\mu+\lambda) (\gradi \circ \dive)_\edges \bfu + \delta\,a \gradi_\edges (\rho^\gamma)   + \delta\, h_\mesh^{\xi_3} \,  \gradi_\edges (\rho^\Gamma)   = \widetilde{\Pi}_\edges\bff.
\end{align}
\end{subequations}

Note that system \eqref{eq:sch} corresponds to $\mcal{F}(\rho,\bfu,1)=0$.
An easy verification shows that any solution $(\rho,\bfu)$ of the problem $\mcal{F}(\rho,\bfu,\delta)=0$ for $\delta$ in $[0,1]$, satisfies the same estimates as stated in Propositions \ref{prop:rho-pos} (positivity of $\rho$) and \ref{prop:estimate:u} (estimate on $\norm{\bfu}_{1,2,\mesh}$) uniformly in $\delta$. However, the positivity of the density is not sufficient to apply the topological degree 
theorem .
We need to prove that there exists a positive lower bound on $\rho$, if $(\rho, \bfu)$ is a solution of \eqref{eq:sch:delta}, which is uniform with respect to $\delta\in[0,1]$. For the lower bound, we use Proposition \ref{prop:rho-pos} and the fact that $\norm{\bfu}_{1,2,\mesh}\leq C_1$ uniformly with respect to $\delta\in[0,1]$ which implies that the quantity $\max\limits_{K\in\mesh}\dive(\bfu)_K$ is also controlled uniformly with respect to $\delta\in[0,1]$ as follows:
\[
 \Big|\max\limits_{K\in\mesh}\dive(\bfu)_K\Big|\leq \sqrt{3\min\limits_{K\in\mesh} |K|^{-1}}\, C_1.
\]
Hence a positive lower bound on $\rho$, if $(\rho, \bfu)$ is a solution of \eqref{eq:sch:delta}, which is uniform with respect to $\delta\in[0,1]$, is given by:
\begin{equation}
 \bar \rho_{\rm min}= \frac{\rho^\star}{1+ h_\mesh^{-\xi_1} \, \sqrt{3\min\limits_{K\in\mesh} |K|^{-1}}\, C_1 }.
\end{equation}
We also obtain a uniform upper bound on $\rho$ by remarking that:
\[
 \norm{\rho}_{\xL^\infty(\Omega)} \leq \frac{1}{\min\limits_{K\in\mesh} |K|^{-1}} \,  \norm{\rho}_{\xL^1(\Omega)}=   \frac{1}{\min\limits_{K\in\mesh} |K|^{-1}} \,  |\Omega|\,\rho^\star =: \bar \rho_{\rm max}.
\]

We then have the next theorem (see details in \cite{gal-08-unc} previously cited).


\begin{thrm}[Existence of a solution] \label{thrm:existence}
Let $\disc=(\mesh,\edges)$ be a staggered discretization of $\Omega$ in the sense of Definition \ref{def:disc}.
The non-linear system \eqref{eq:sch} admits at least one solution $(\rho, \bfu)$ in $\xL_\mesh(\Omega) \times \xbfH_{\mesh,0}(\Omega)$, and any possible solution satisfies the estimates of Propositions \ref{prop:rho-pos},  \ref{prop:estimate:u} and \ref{prop:estimate:p}.
\end{thrm}

\section{Proof of the convergence result}
\label{sec:pass-limit}


\medskip
Let $(\disc\n)_{n\in\xN}$ be a regular sequence of staggered discretizations of $\Omega$ as defined in Definition \ref{def:reg_disc}.  We denote $h\n$ instead of $h_{\mesh\n}$ in order to ease the notations. Similar simplifications will be used thereafter.

\medskip
Theorem \ref{thrm:existence} applies and without loss of generality (assuming $h\n$ is small enough for all $n\in\xN$), we can assume that for all $n\in\xN$ there exists a solution $(\rho\n,\bfu\n)\in\xL_{\mesh\n}(\Omega)\times\xbfH_{\mesh\n,0}(\Omega)$ to the numerical scheme \eqref{eq:sch} with the discretization $\disc\n$. In addition, the obtained density $\rho\n$ is positive \emph{a.e.} in $\Omega$.
Since $\theta_{\mesh\n} \leq \theta_0$ for all $n\in\xN$, the sequence $(\rho\n,\bfu\n)_{n\in \xN}$ satisfies the following estimates. There exist $C_0>0$, $p\in(1,1+\eta)$ and $r\in(0,1)$ such that:
\begin{multline}
\label{unif:est}
 \norm{\bfu\n}_{1,2,\mesh\n} 
 + \norm{\rho\n}_{\xL^{3(\gamma-1)}(\Omega)} 
 +  \norm{h\n^{\xi_3}\rho\n^\Gamma}_{\xL^{1 + \eta}(\Omega)}
 + h\n^{\xi_2 + \frac{5}{4\Gamma}\left(\frac{3}{1+\eta}+\xi_3\right)} \,\snorm{\rho\n}_{\frac{1+\eta}{\eta},\mesh\n}^{\frac{1+\eta}{\eta}}\\[2ex]
 + h\n^{-\xi_3(1-r)} \norm{h\n^{\xi_3} \rho\n^\Gamma}_{\xL^{p}(\Omega)} 
 + h\n^{\frac{5}{4\Gamma}\left(\frac{3}{1+\eta}+\xi_3\right)}\sum_{\substack{\edge\in{\edgesintn} \\ \edge=K|L}} |\edge|\, (\rho_L-\rho_K)^2 \,|\bfu_\edge\cdot\bfn_{K,\edge}|
 \leq C_0, \quad \forall n\in\xN.
\end{multline}
In order to ease the notations, the subscript $n$ has been omitted in the above summation on the internal faces of $\edges\n$.

\medskip 
Thanks to these estimates, there is a subsequence of $(\disc\n)_{n\in\xN}$, still denoted $(\disc\n)_{n\in\xN}$ such that $(\rho\n)_{n\in\xN}$ weakly converges in $\xL^{3(\gamma-1)}(\Omega)$ to some $\rho\in\xL^{3(\gamma-1)}(\Omega)$, and $(\rho\n^\gamma)_{n\in\xN}$ weakly converges in $\xL^\frac{3(\gamma-1)}{\gamma}(\Omega)$ to some $\overline{\rho^\gamma}\in\xL^{\frac{3(\gamma-1)}{\gamma}}(\Omega)$.
The compactness of the sequence of velocities relies on the following theorem
which is a compactness result for the discrete $\xH^1_0$-norm similar to Rellich's theorem. We refer to \cite{gal-09-conv} (Theorem 3.3) for a proof. See also \cite{stum-80-bas}.

\begin{thrm}[Discrete Rellich theorem]
\label{thrm:stokes-lim-reg}
Let $(\disc\n)_{n\in\xN}$ be a sequence of staggered discretizations of $\Omega$ satisfying $\theta_{\mesh\n}\leq\theta_0$ for all $n\in\xN$. For all $n\in\xN$, let $\bfu\n\in \xbfH_{\mesh\n,0}(\Omega)$ and assume that there exists $C\in\xR$ such that $\norm{\bfu\n}_{1,2,\mesh\n}\leq C$, $\forall \,n\in\xN$. We suppose that $h\n\to 0$ as $n\to+\infty$. Then:
\begin{enumerate}
 \item There exists a subsequence of $(\bfu\n)_{n\in\xN}$, still denoted $(\bfu\n)_{n\in\xN}$, which converges in $\xbfL^2(\Omega)$ towards a function $\bfu\in\xbfL^2(\Omega)$.
 \item The limit function $\bfu$ belongs to $\xbfH^{1}_0(\Omega)$ with $\norm{\gradi \bfu}_{\xbfL^2(\Omega)^3} \leq C$.
 \item The sequence $(\gradi_{\mesh\n}\bfu\n)_{n\in\xN}$ weakly converges to $\gradi\bfu$ in $\xbfL^2(\Omega)^3$.
\end{enumerate}
\end{thrm}

Hence, upon extracting a new subsequence from $(\disc\n)_{n\in\xN}$, we may assume that there exists $\bfu\in\xbfH^{1}_0(\Omega)$ such that the sequence $(\bfu\n)_{n\in\xN}$ converges to $\bfu$ in $\xbfL^2(\Omega)$. By the discrete Sobolev inequality of Lemma \ref{lmm:injsobolev_br}, we can actually assume that $(\bfu\n)_{n\in\xN}$ converges to $\bfu$ in $\xbfL^q(\Omega)$ for all $q\in[1,6)$ and weakly in $\xbfL^6(\Omega)$.

\medskip
Following the same steps as 
in the continuous setting, we first pass to the limit $n\to+\infty$ in the mass and momentum equations in Sections \ref{sec:limit:mass} and \ref{sec:limit:mom} and then pass to the limit in the equation of state in Section \ref{sec:limit:eos}, by proving the strong convergence of the density.

\subsection{Passing to the limit in the mass conservation equation}
\label{sec:limit:mass}

\medskip
\begin{prop} \label{prop:convergence_rho}
Under the assumptions of Theorem \ref{main_thrm}, the limit pair $(\rho,\bfu)\in \xL^{3(\gamma-1)}(\Omega)\times\xbfH^1_0(\Omega)$ of the sequence $(\rho\n,\bfu\n)_{n\in\xN}$ satisfies the mass equation in the weak sense:
\begin{equation}
\label{prop:eq:convergence_rho}
-\int_\Omega \rho\, \bfu\cdot  \gradi\phi\dx = 0, \qquad \forall\phi\in \mcal{C}^{\infty}_c(\Omega). 
\end{equation}
\end{prop}

\medskip
Let us first state the following lemma which will be useful in the proof of Proposition \eqref{prop:convergence_rho}.

\medskip

\begin{lem} \label{lmm:test_mass}
Let $\phi\in \mcal{C}^{\infty}_c(\Omega)$. For $n \in \xN$ define $\phi\n\in\xL_{\mesh\n}(\Omega)$ by ${\phi\n}_{|K}=\phi_K$ the mean value of $\phi$ over $K$, for $K \in \mesh\n$. Denote $\phi_\edge=|\edge|^{-1}\int_\edge \phi(\bfx) \dedge(\bfx)$ for all $\edge \in \edges\n$ and define  
a discrete gradient of $\phi\n$ by:
\[
\overline{\gradi}_{\mesh\n} \phi\n(\bfx) = \sum_{K \in \mesh\n} (\gradi \phi)_K \, \mathcal{X}_K(\bfx), 
\quad
\mbox{with }
\quad
(\gradi \phi)_K = \frac 1 {|K|} \sum_{\edge\in\edges(K)} |\edge|\,  \phi_\edge\, \bfn_{K,\edge}.
\]
Then for all $q$ in $[1,\infty]$, there exists $C=C(\Omega,q,\phi,\theta_0)$ such that:
\begin{equation} 
\label{eq:interpolate_t} 
\norm{\overline{\gradi}_{\mesh\n} \phi\n-\gradi \phi}_{\xbfL^q( \Omega)} \leq C \, h\n.
\end{equation}
\end{lem}

\medskip
\begin{proof}
Let $q\in[1,+\infty)$. We have  $\norm{\overline{\gradi}_{\mesh\n} \phi\n-\gradi \phi}_{\xbfL^q( \Omega)}^q=\sum_{K\in\mesh\n}\norm{\overline{\gradi}_{\mesh\n} \phi\n-\gradi \phi}_{\xbfL^q(K)}^q$ with for $K\in\mesh\n$:
\[
 \begin{aligned}
  \norm{\overline{\gradi}_{\mesh\n} \phi\n-\gradi \phi}_{\xbfL^q(K)}^q
  &= \int_K \Big| \frac 1 {|K|} \sum_{\edge\in\edges(K)} |\edge|\,  \phi_\edge\, \bfn_{K,\edge} -\gradi\phi(\bfx)\Big  |^q\dx\\
  &= \int_K \Big| \frac 1 {|K|} \sum_{\edge\in\edges(K)} \int_\edge \phi(\bfy) \dedge(\bfy)\, \bfn_{K,\edge} -\gradi\phi(\bfx)\Big  |^q\dx\\
  &= \int_K \Big| \frac 1 {|K|} \int_{\dv K} \phi(\bfy) \dedge(\bfy)\, \bfn_{K} -\gradi\phi(\bfx)\Big  |^q\dx\\
  &\leq \int_K \Big( \frac 1 {|K|} \int_{K} |\gradi\phi(\bfy) -\gradi\phi(\bfx)|\dy \Big  )^q\dx.\\
 \end{aligned}
\]
By a Taylor expansion, we have for all $\bfy$, $\bfx\in K$, $|\gradi\phi(\bfy) -\gradi\phi(\bfx)|\leq h\n\,\snorm{\phi}_{\xW^{2,\infty}(\Omega)}$. Thus we have: $\norm{\overline{\gradi}_{\mesh\n} \phi\n-\gradi \phi}_{\xbfL^q(K)}^q\leq h\n^q\,\snorm{\phi}_{\xW^{2,\infty}(\Omega)}^q\,|K|$ which concludes the proof for $q\in[1,+\infty)$. The proof is similar for $q=+\infty$.
\end{proof}

\medskip
We can now give the proof of Proposition \eqref{prop:convergence_rho}.

\medskip
\begin{proof}[Proof of Proposition \eqref{prop:convergence_rho}.]
To prove this result we pass to the limit $n\to+\infty$ in the weak formulation of the discrete mass balance. Let $\phi\in \mcal{C}^{\infty}_c(\Omega)$ and for $n \in \xN$ define $\phi\n\in\xL_{\mesh\n}(\Omega)$ by ${\phi\n}_{|K}=\phi_K$ the mean value of $\phi$ over $K$, for $K \in \mesh\n$. Multiplying the discrete mass balance \eqref{eq:sch_mass} by $|K|\,\phi_K\,\Ind_K$, summing over $K \in \mesh\n$ and performing a discrete integration by parts (\emph{i.e.} reordering the sum) yields:
\begin{equation}
\label{mass00}
-\int_\Omega (\xP_{\edges\n}\rho\n)\,(\xPi_{\edges\n}\bfu\n)\cdot\gradi_{\edges\n}\phi\n \,\dx +  R_1^n+R_2^n=0,
\end{equation}
with
\[
\begin{aligned}
 &R_1^n = h\n^{\xi_1} \sum_{K\in\mesh\n} |K| (\rho_K-\rho^\star)\,\phi_K, \\
 &R_2^n = h\n^{\xi_2} \sum_{K\in\mesh\n} \Big (\sum_{\edge \in\edges(K)} |\edge| \,\Big(\dfrac{|\edge|}{|D_\edge|}\Big)^\frac{1}{\eta}\,|\rho_K-\rho_L|^{\frac{1}{\eta}-1}\,(\rho_K-\rho_L)\Big )\,\phi_K,
\end{aligned}
\]
where $\gradi_{\edges}$ is the discrete gradient defined in \eqref{eq:def_grad_p} and $\xP_\edges$ is defined in \eqref{def:pepsrho}.

\medskip
In order to prove Proposition \ref{prop:convergence_rho}, we want to pass to the limit in the first term of \eqref{mass00}.
It is possible to prove that $\xPi_{\edges\n}\bfu\n\to\bfu$ strongly in $\xbfL^2(\Omega)$. However,
the discrete gradient $\gradi_{\edges\n}\phi\n$ is known to converge only weakly towards $\gradi\phi$ because locally on a dual cell $D_\edge$ it is supported by only one direction, that of the normal vector $\bfn_{K,\edge}$
Thus, it is not possible to pass to the limit in
\eqref{mass00}. Instead, we use the discrete gradient $\overline{\gradi}_{\mesh\n} \phi\n$ introduced in Lemma \ref{lmm:test_mass}, which is known to converge strongly towards $\gradi \phi$.

\medskip
We have:
\begin{align}
\nonumber -\int_\Omega (\xP_{\edges\n}\rho\n)\,(\xPi_{\edges\n}\bfu\n)\cdot\gradi_{\edges\n}\phi\n \,\dx &= - \sum_{\substack{\edge\in{\edgesintn} \\ \edge=K|L}} |\edge|\,\rho_\edge \,(\phi_L-\phi_K) \, \bfu_\edge\cdot\bfn_{K,\edge} \\
\label{mass1}     &= - \sum_{\substack{\edge\in{\edgesintn} \\ \edge=K|L}} |\edge|\,\tfrac{1}{2}(\rho_K+\rho_L) \, (\phi_L-\phi_K) \, \bfu_\edge\cdot\bfn_{K,\edge} + R_3^n,
\end{align}
with 
\[
 R_3^n = \sum_{\substack{\edge\in{\edgesintn} \\ \edge=K|L}} |\edge|\,(\tfrac{1}{2}(\rho_K+\rho_L)-\rho_\edge)\, (\phi_L-\phi_K) \, \bfu_\edge\cdot\bfn_{K,\edge}.
\]
Reordering the sum in the first term of \eqref{mass1} we get:
\begin{align}
\nonumber -\int_\Omega (\xP_{\edges\n}\rho\n)\,(\xPi_{\edges\n}\bfu\n)\cdot\gradi_{\edges\n}\phi\n \,\dx    
&
= - \frac 12\, \sum_{K\in\mesh\n}  \rho_K  \Big ( \sum_{\substack{\edge \in\edges(K) \\ \edge=K|L}} |\edge|\,(\phi_L-\phi_K) \, \bfu_\edge\cdot\bfn_{K,\edge} \Big )+  R_3^n,\\
\label{mass2}
&= -  \frac 12\, \sum_{K\in\mesh\n}  \rho_K \, \bfu_K \cdot \Big (\sum_{\substack{\edge \in\edges(K) \\ \edge=K|L}} |\edge|\,(\phi_L-\phi_K) \,\bfn_{K,\edge}\Big ) +  R_3^n + R_4^n,
\end{align}
where here, $\bfu_K$ is the mean value of $\bfu\n$ over $K$ and
\[
 R_4^n = \frac 12\, \sum_{K\in\mesh\n}  \rho_K \Big (\sum_{\substack{\edge \in\edges(K) \\ \edge=K|L}} |\edge|\,(\phi_L-\phi_K) \, (\bfu_K-\bfu_\edge) \cdot \bfn_{K,\edge}\Big ).
\]
Back to \eqref{mass2}, we have:
\begin{align}
-\int_\Omega (\xP_{\edges\n}\rho\n)\,(\xPi_{\edges\n}\bfu\n)&\cdot\gradi_{\edges\n}\phi\n \,\dx \nonumber \\
&= -  \sum_{K\in\mesh\n}  \rho_K \, \bfu_K \cdot \Big (\sum_{\substack{\edge \in\edges(K) \\ \edge=K|L}} |\edge|\,\tfrac 12(\phi_L+\phi_K) \,\bfn_{K,\edge}\Big ) +  R_3^n + R_4^n+R_5^n \nonumber \\
&= -  \sum_{K\in\mesh\n}  \rho_K \, \bfu_K \cdot \Big (\sum_{\edge \in\edges(K)} |\edge|\,\phi_\edge \,\bfn_{K,\edge}\Big ) +  R_3^n + R_4^n+R_5^n+R_6^n \nonumber\\
\label{mass00bis}
&= -  \int_\Omega \rho\n\,\bfu\n\cdot\overline{\gradi}_{\mesh\n} \phi \dx+  R_3^n + R_4^n+R_5^n+R_6^n.
\end{align}
where :
\[
\begin{aligned}
 R_5^n&= \sum_{K\in\mesh\n}  \rho_K \, \bfu_K \, \phi_K \cdot \Big (\sum_{\edge \in\edges(K)} |\edge| \,\bfn_{K,\edge}\Big ),\\
 R_6^n&= \sum_{K\in\mesh\n}  \rho_K \, \bfu_K \cdot \Big (\sum_{\substack{\edge \in\edges(K) \\ \edge=K|L}} |\edge|\,(\phi_\edge-\tfrac 12(\phi_L+\phi_K)) \,\bfn_{K,\edge}\Big ).
\end{aligned} 
\]
Replacing \eqref{mass00bis} in \eqref{mass00} we get:
\begin{equation}
\label{mass00ter}
 -\int_\Omega \rho\n\,\bfu\n\cdot\overline{\gradi}_{\mesh\n} \phi \dx + R_1^n+R_2^n+R_3^n + R_4^n+R_5^n+R_6^n=0.
\end{equation}
Since $\bfu\n\to\bfu$ strongly in $\xbfL^q(\Omega)$ for all $q \in [1,6)$, and $\overline{\gradi}_{\mesh\n}\phi_n\to\gradi\phi$ (by \eqref{eq:interpolate_t}) in $\xbfL^6(\Omega)^3$ as $n\to+\infty$, we have $\bfu\n\cdot\overline{\gradi}_{\mesh\n}\phi_n\to\bfu\cdot\gradi\phi$ strongly in $\xL^{3-\delta}(\Omega)$ for all $\delta\in(0,2]$. 
Furthermore, we have $\rho\n\rightharpoonup\rho$ weakly in $\xL^{3(\gamma-1)}(\Omega)$ with $3(\gamma-1)>\frac 32$ (since $\gamma > \frac 32$), which yields:
\[
 \lim\limits_{n\to+\infty} \,  \int_\Omega \rho\n\,\bfu\n\cdot\overline{\gradi}_{\mesh\n}\phi\n \,\dx  =\int_\Omega \rho\, \bfu\cdot  \gradi\phi\dx.
\]
It remains to prove that $ \sum_{i=1}^6 R_i^n \to 0$ as $n\to+\infty$. In the following, in order to ease the notations, we denote $A_n\lesssim B_n$ when there is a constant $C$, independent of $n$, such that $A_n\leq C\,B_n$.
We easily prove that $R_1^n\to0$ and $R_2^n\to0$ as $n\to+\infty$. Indeed, one has:
\[
 |R_1^n| \leq 2\, h\n^{\xi_1} \, |\Omega|\rho^\star\,\norm{\phi}_{\xL^\infty(\Omega)},
\]
which proves that $R_1^n\to0$ since $\xi_1>0$. For $R_2^n$, reordering the sum, we get:
\[
 \begin{aligned}
 |R_2^n| 
& \leq h\n^{\xi_2} \sum_{\substack{\edge\in{\edgesintn} \\ \edge=K|L}} |\edge| \,\Big(\dfrac{|\edge|}{|D_\edge|}\Big)^\frac{1}{\eta}\,|\rho_K-\rho_L|^{\frac{1}{\eta}}\,|\phi_K-\phi_L| \\
&\lesssim \norm{\gradi \phi}_{\xbfL^\infty(\Omega)}\, h\n^{\xi_2} \sum_{\substack{\edge\in{\edgesintn} \\ \edge=K|L}} |D_\edge| \,\Big(\dfrac{|\edge|}{|D_\edge|}\Big)^\frac{1}{\eta}\,|\rho_K-\rho_L|^{\frac{1}{\eta}}.
 \end{aligned}
\]
Applying H\"older's inequality (with coefficients $1+\eta$ and $(1+\eta)/\eta$) to the sum, we get:
\begin{align*}
|R_2^n| 
& \lesssim \norm{\gradi \phi}_{\xbfL^\infty(\Omega)}\, |\Omega|^{\frac{\eta}{\eta+1}}\, h\n^{\xi_2} \snorm{\rho}_{\frac{1+\eta}{\eta},\mesh\n}^{\frac{1}{\eta}} \\
& \lesssim h\n^{\frac{\eta}{1+\eta}(\xi_2 -  \frac{5}{4\eta\Gamma}(\frac{3}{1+\eta} + \xi_3))} \Big(h\n^{\frac{\eta}{1+\eta}(\xi_2 +  \frac{5}{4\Gamma}(\frac{3}{1+\eta} + \xi_3))} \snorm{\rho}_{\frac{1+\eta}{\eta},\mesh\n}\Big )^{\frac{1}{\eta}}.
\end{align*}
Thanks to \eqref{unif:est} and to assumption \eqref{thm:cond_xi2}, we have $R_2^n\to0$ as $n\to+\infty$.
Let us now turn to $R_3^n$. Recalling the upwind definition of $\rho_\edge$, we get:
\[
 |R_3^n| \leq \frac 12 \, \sum_{\substack{\edge\in{\edgesintn} \\ \edge=K|L}} |\edge|\,|\rho_K-\rho_L|\, |\phi_L-\phi_K| \, |\bfu_\edge\cdot\bfn_{K,\edge}|.
\]
Applying the Cauchy-Schwarz inequality, we infer that:
\[
\begin{aligned}
|R_3^n| 
&\leq \frac 12 \, \Big ( \sum_{\substack{\edge\in{\edgesintn} \\ \edge=K|L}} |\edge|\,|\rho_K-\rho_L|^2\, |\bfu_\edge\cdot\bfn_{K,\edge}| \Big)^{\frac12} \, \Big ( \sum_{\substack{\edge\in{\edgesintn} \\ \edge=K|L}} |\edge|\, |\phi_L-\phi_K|^2 \, |\bfu_\edge\cdot\bfn_{K,\edge}| \Big)^{\frac12}\\
&\lesssim  h_n^{- \frac{5}{8\Gamma} \big( \frac{3}{1+\eta}+\xi_3\big )}  
\, \Big ( \sum_{\substack{\edge\in{\edgesintn} \\ \edge=K|L}} |\edge|\, |\phi_L-\phi_K|^2 \, |\bfu_\edge\cdot\bfn_{K,\edge}| \Big)^{\frac12}
\end{aligned}
\]
by estimate \eqref{unif:est}. By Taylor's inequality applied to the smooth function $\phi$ and the regularity of the discretization, we have $|\phi_L-\phi_K|^2\lesssim h\n\,|D_\edge|/|\edge|\,\norm{\gradi\phi}_{\xbfL^\infty(\Omega)}^2$. Hence:
\begin{align*}
|R_3^n| 
& \lesssim   h_n^{\frac{5}{8\Gamma} \big( \frac{4}{5}\Gamma - \frac{3}{1+\eta} -\xi_3\big )}
\,\Big ( \sum_{\substack{\edge\in{\edgesintn} \\ \edge=K|L}} |D_\edge|\,|\bfu_\edge|\Big)^{\frac12}
=   h_n^{ \frac{5}{8\Gamma} \big( \frac{4}{5} \Gamma - \frac{3}{1+\eta} -\xi_3\big )}
\, \norm{\xPi_{\edges\n}\bfu\n}_{\xbfL^1(\Omega)}^{\frac12} 
\\
& \lesssim   h_n^{ \frac{5}{8\Gamma} \big( \frac{4}{5}\Gamma - \frac{3}{1+\eta} -\xi_3\big )}
\,\norm{\bfu\n}_{\xbfL^1(\Omega)}^{\frac12} \\
& \lesssim   h_n^{ \frac{5}{8\Gamma} \big( \frac{4}{5}\Gamma - \frac{3}{1+\eta} -\xi_3\big )}
\end{align*}
since $\norm{\bfu\n}_{\xbfL^6(\Omega)}$, and thus $\norm{\bfu\n}_{\xbfL^1(\Omega)}$, is controlled by $\norm{\bfu\n}_{1,2,\mesh\n}$ which is bounded by $C_0$. 
Since $(\Gamma,\xi_3)$ satisfy \eqref{thm:cond_xi3} we get $R_3^n\to0$ as $n\to+\infty$. \\
We now turn to $R_4^n$. 
By a Taylor inequality on the smooth function $\phi$ and the regularity of the discretization, we have: $|\phi_L-\phi_K| \lesssim h\n \, \norm{\gradi \phi}_{\xbfL^\infty(\Omega)}$. Hence:
	\begin{equation}
	\label{mass4}
	|R_4^n| \lesssim  h\n \sum_{K\in\mesh\n} \rho_K  \sum_{\edge \in\edges(K)} |\edge|\,|\bfu_K-\bfu_\edge|.
	\end{equation}
	The vectors $\bfu_K$ and $\bfu_\edge$ are the mean values of $\bfu$ over $K$ and $\edge\in\edges(K)$ respectively. By the Cauchy-Schwarz inequality, we can prove that:
	\[
	|\bfu_K-\bfu_\edge|^2 \leq \frac{1}{|\edge|}\,\frac{1}{|K|}\,\int_\edge \int_K |\bfu\n(\bfy)-\bfu\n(\bfx)|^2\,\dx\dedge(\bfy).
	\]
	Since $\bfu\n$ is smooth over $K$ we have for $\bfx,\,\bfy\in K$:
	\[
	|\bfu\n(\bfy)-\bfu\n(\bfx)|^2 \leq |\bfy-\bfx|^2\int_0^1 |\gradi\bfu\n(t\bfy+(1-t)\bfx)|^2\dt.
	\]
	Bounding $|\bfy-\bfx|$ by $h_K$ we obtain, using Fubini's theorem that $|\bfu_K-\bfu_\edge|^2 \leq \frac{h_K^2}{|K|} \norm{\gradi \bfu\n}^2_{\xbfL^2(K)^3}$. Injecting in \eqref{mass4}, and invoking the regularity of the discretization we get:
	\[
	\begin{aligned}
	|R_4^n| 
	& \lesssim  h\n \sum_{K\in\mesh\n} |K|^{\frac 12} \rho_K  \norm{\gradi \bfu\n}_{\xbfL^2(K)^3} \\
	& \lesssim h\n \, \norm{\rho\n}_{\xL^{\infty}(\Omega)}^{1-\frac{3(\gamma-1)}{2}}\sum_{K\in\mesh\n} |K|^{\frac 12} \rho_K^{\frac{3(\gamma-1)}{2}}  \norm{\gradi \bfu\n}_{\xbfL^2(K)^3}\\
	& \lesssim h\n \, \norm{\rho\n}_{\xL^{\infty}(\Omega)}^{\frac{5-3\gamma}{2}} \norm{\rho\n}_{\xL^{3(\gamma-1)}(\Omega)}^{\frac{3(\gamma-1)}{2}}  \norm{\bfu\n}_{1,2,\mesh\n}.
	\end{aligned}
	\]
	Thus, by the inverse inequality $\norm{\rho\n}_{\xL^{\infty}(\Omega)}\lesssim h\n^{-\frac{1}{\gamma-1}}\norm{\rho\n}_{\xL^{3(\gamma-1)}(\Omega)}$ and since the sequence $(\rho\n)_{n\in\xN}$ is bounded in $\xL^{3(\gamma-1)}(\Omega)$ and the sequence $(\norm{\bfu\n}_{1,2,\mesh\n})_{n\in\xN}$ is bounded we get:
	\[
	|R_4^n| \lesssim h\n^{1-\frac{1}{\gamma-1}\frac{5-3\gamma}{2}} = h\n^{\frac{5\gamma-7}{2(\gamma-1)}}.
	\]
	Since, $\gamma>\frac32>\frac 75$, we deduce that $R_4^n\to0$ as $n\to+\infty$.\\
The fifth remainder term satisfies $R_5^n=0$ since $\sum_{\edge \in\edges(K)} |\edge| \,\bfn_{K,\edge}=0$ for all $K\in\mesh\n$. Let us conclude with the control of $R_6^n$. Denoting $\hat\phi_\edge=\frac12(\phi_L+\phi_K)$ for $\edge=K|L$, we may write $R_6^n=R_{6,1}^n+R_{6,2}^n$ with:
\[
\begin{aligned}
R_{6,1}^n& =\sum_{K\in\mesh\n}  \rho_K  \Big (\sum_{\edge \in\edges(K)} |\edge|\,(\phi_\edge-\hat\phi_\edge)  \, (\bfu_K-\bfu_\edge) \cdot\bfn_{K,\edge}\Big ),\\
R_{6,2}^n& =\sum_{K\in\mesh\n}  \rho_K  \Big (\sum_{\edge \in\edges(K)} |\edge|\,(\phi_\edge-\hat\phi_\edge)  \, \bfu_\edge \cdot\bfn_{K,\edge}\Big ).
\end{aligned}
\]
The term $R_{6,1}^n$ can be controlled the same way as $R_4^n$ and we obtain $R_{6,1}^n\to0$ as $n\to+\infty$. Reordering the sum in $R_{6,2}^n$ we get:
\[
R_{6,2}^n = \sum_{\substack{\edge\in{\edgesintn} \\ \edge=K|L}} |\edge|\,(\rho_K-\rho_L)\, (\phi_\edge-\hat\phi_\edge) \, \bfu_\edge\cdot\bfn_{K,\edge}.
\]
Hence $R_{6,2}^n$ can be controlled the same way as $R_3^n$ and we obtain $R_{6,2}^n\to0$ as $n\to+\infty$ and this concludes the proof of \eqref{prop:eq:convergence_rho}.

\end{proof}

\subsection{Passing to the limit in the momentum equation}
\label{sec:limit:mom}

\medskip
\begin{prop} \label{prop:convergence_u}
Under the assumptions of Theorem \ref{main_thrm}, the limit triple $(\rho,\bfu,\overline{\rho^\gamma})\in \xL^{3(\gamma-1)}(\Omega)\times\xbfH^1_0(\Omega)\times\xL^{\frac{3(\gamma-1)}{\gamma}}(\Omega)$ of the sequence $(\rho\n,\bfu\n,\rho\n^\gamma)_{n\in\xN}$ satisfies the momentum equation in the weak sense:
\begin{multline}
\label{prop:eq:convergence_u}
-\int_\Omega \rho\, \bfu \otimes \bfu : \gradi \bfv \,\dx
+ \mu\, \int_\Omega \gradi \bfu : \gradi \bfv  \dx 
+ (\mu+\lambda)\,\int_\Omega \dive\, \bfu \,  \dive\, \bfv\, \dx \\[2ex]
- a \int_\Omega \overline{\rho^\gamma}\,  \dive\, \bfv\, \dx
= \int_\Omega \bff \cdot \bfv \,\dx, \qquad \forall \bfv\in\mcal{C}^{\infty}_c(\Omega)^3.
\end{multline}
Moreover, we have the following energy inequality satisfied at the limit:
\begin{equation}\label{eq:c-energy:disc}
\mu \int_\Omega{|\gradi \bfu|^2 \dx} + (\lambda+\mu)\int_{\Omega}{(\dive\,\bfu)^2 \dx} 	\leq \int_\Omega \bff \cdot \bfu \dx.	
\end{equation}
\end{prop}


\medskip

For $\disc=(\mesh,\edges)$ a staggered discretization of $\Omega$, we define $I_\mesh$ the following Fortin operator associated with the Crouzeix-Raviart finite element:
\begin{equation}\label{eq:rh}
I_\mesh :
\left\lbrace
\begin{array}{ccl}
\xbfW^{1,q}(\Omega)  & \longrightarrow & \xbfH_{\mesh}(\Omega) \\ 
\bfv & \longmapsto & \displaystyle I_\mesh \bfv= \sum_{\edge \in \edges}  \bfv_\edge\,\zeta_\edge, \quad  \text{ with $\bfv_\edge=|\edge|^{-1} \int_\edge  \bfv \dedge(\bfx)$ for $\edge\in\edges$.}
\end{array}
\right.
\end{equation}
The following lemma states the main properties of operator $I_\mesh$.
We refer to the appendix of \cite{gallouet2015} for a proof.

\medskip
\begin{lem}[Properties of the operator $I_\mesh$]
\label{lmm:RT}
Let $\disc=(\mesh,\edges)$ be a staggered discretization of $\Omega$ such that $\theta_\mesh \leq \theta_0$ (where $\theta_\mesh$ is defined by \eqref{eq:reg}) for some positive constant $\theta_0$.
For any $q \in [1,+\infty)$, there exists $C=C(\theta_0,q)$ such that:
\begin{enumerate}
\item[(i)] Stability: 
\[
\norm{I_\mesh \bfu}_{1,q,\mesh} \leq C\ \snorm{\bfu}_{\xbfW^{1,q}(\Omega)}, \qquad \forall \bfu\in\xbfW^{1,q}_0(\Omega).
\]
\item[(ii)] Approximation:
For all $K\in\mesh$:
\[
\norm{\bfu-I_\mesh \bfu}_{\xbfL^q(K)} + h_K\, \norm{\gradi (\bfu-I_\mesh \bfu)}_{\xbfL^q(K)^3} 
\leq C\, h_K^2\, |\bfu|_{\xbfW^{2,q}(K)}, \qquad \forall \bfu \in \xbfW^{2,q}(\Omega)\cap \xbfW^{1,q}_0(\Omega). 
\]

\item[(iii)] Preservation of the divergence:
\[
 \int_{\Omega} p\, \dive_\mesh(I_\mesh \bfu) \dx= \int_\Omega p\, \dive\, \bfu \dx, \qquad \forall p\in\xL_\mesh(\Omega), \ \bfu\in\xbfW_{0}^{1,q}(\Omega).
\]
\end{enumerate}
\end{lem}

\medskip
\begin{lem} \label{lmm:test_mom}
Let $\bfv\in \mcal{C}^{\infty}_c(\Omega)^3$. Let $(\disc\n)_{n\in\xN}$ be a regular sequence of staggered discretizations as defined in Definition \ref{def:reg_disc}. For $n \in \xN$ define $\bfv\n\in\xbfH_{\mesh\n,0}(\Omega)$ by $\bfv\n=I_{\mesh\n}\bfv$. Then, for any $q \in [1,+\infty)$, there exists $C=C(\Omega,q,\bfv,\theta_0)$ such that:
\begin{align} 
\label{eq:interpolate1} 
& \norm{\bfv\n-\bfv}_{\xbfL^q( \Omega)} \leq C \, h\n^2, \\
\label{eq:interpolate2}
& \norm{\gradi_{\mesh\n}\bfv\n-\gradi\bfv}_{\xbfL^q( \Omega)^3} \leq C \, h\n,\\
\label{eq:interpolate1bis} 
&\norm{\xPi_{\edges\n} \bfv\n-\bfv}_{\xbfL^q(\Omega)} \leq C \, h\n.
\end{align}
In addition, denoting $\bfv_\edge=|\edge|^{-1}\int_\edge \bfv\dedge(\bfx)$ for all $\edge \in \edges\n$ we define a discrete gradient of $\bfv\n$ by:
\[
\overline{\gradi}_{\mesh\n} \bfv\n(\bfx) = \sum_{K \in \mesh\n} (\gradi \bfv)_K \, \mathcal{X}_K(\bfx), 
\quad
\mbox{with }
\quad
(\gradi \bfv)_K = \frac 1 {|K|} \sum_{\edge\in\edges(K)} |\edge|\,  \bfv_\edge\otimes \bfn_{K,\edge}.
\]
Then for all $q$ in $[1,\infty]$, there exists $C=C(\Omega,q,\bfv,\theta_0)$ such that:
\begin{equation} 
\label{eq:interpolate3}
\norm{\overline{\gradi}_{\mesh\n} \bfv\n-\gradi \bfv}_{\xbfL^q( \Omega)^3} \leq C \, h\n.
\end{equation}
\end{lem}

\medskip
\begin{proof}
The estimates \eqref{eq:interpolate1} and \eqref{eq:interpolate2} are direct consequences of the approximation properties of the interpolation operator $I_{\mesh\n}$. The proof of \eqref{eq:interpolate3} is similar to that of Lemma \ref{lmm:test_mass}. To prove \eqref{eq:interpolate1bis} we write:
\[
 \begin{aligned}
  \norm{\xPi_{\edges\n} \bfv\n-\bfv\n}_{\xbfL^q(\Omega)}^q 
  &= \sum_{K\in\mesh\n} \sum_{\edge\in\edges(K)}\int_{D_{K,\edge}} |\bfv_\edge-\bfv\n(\bfx)|^q\dx \\
  &= \sum_{K\in\mesh\n} \sum_{\edge\in\edges(K)}\int_{D_{K,\edge}} \Big |\bfv_\edge-\sum_{\edge'\in\edges(K)}\bfv_{\edge'}\,\zeta_{\edge'}(\bfx) \Big |^q\dx\\
  &= \sum_{K\in\mesh\n} \sum_{\edge\in\edges(K)}\int_{D_{K,\edge}} \Big |\sum_{\edge'\in\edges(K)}(\bfv_\edge-\bfv_{\edge'})\,\zeta_{\edge'}(\bfx) \Big |^q\dx \\
  &\lesssim  h\n^q\,\sum_{K\in\mesh\n} h_K^{3-q} \sum_{\edge,\edge'\in\edges(K)}|\bfv_\edge-\bfv_{\edge'}|^2. 
 \end{aligned}
\]
Hence we have $\norm{\xPi_{\edges\n} \bfv\n-\bfv\n}_{\xbfL^q(\Omega)}\lesssim h\n \,\norm{\bfv\n}_{1,q,\edges\n}\lesssim h\n\,\norm{\bfv\n}_{1,q,\mesh\n} \lesssim h\n \snorm{\bfv}_{\xbfW^{1,q}(\Omega)^3}$. 
Combining this with \eqref{eq:interpolate1} yields the result.
\end{proof}

\medskip
We can now give the proof of Proposition \eqref{prop:convergence_u}.

\medskip
\begin{proof}[Proof of Proposition \ref{prop:convergence_u}]
To prove this result, we pass to the limit $n\to+\infty$ in the weak formulation of the discrete momentum balance. Let $\bfv\in\mcal{C}^{\infty}_c(\Omega)^3$ and for $n\in\xN$, define $\bfv\n=I_{\mesh\n}\bfv\in\xbfH_{\mesh\n,0}(\Omega)$.
We have $\norm{\bfv\n}_{1,q,\mesh\n}\leq C\,\norm{\bfv}_{\xbfW^{1,q}_0(\Omega)}$ for all $q\in[1,+\infty)$ by Lemma \ref{lmm:RT}.
Taking the test function $\bfv\n$ in the weak formulation of the discrete momentum balance \eqref{eq:sch_mom:weak:2}, we get for all $n\in\xN$:
\begin{multline}
\label{mom2}
- \int_\Omega(\xP_{\edges\n}\rho\n) \,(\xPi_{\edges\n}\bfu\n)\otimes (\xPi_{\edges\n}\bfu\n): \gradi_{\edges\n} \bfv\n\dx \\[2ex] 
+\mu \int_\Omega\gradi_{\mesh\n} \bfu\n:\gradi_{\mesh\n} \bfv\n \dx 
+ (\mu+\lambda)\int_\Omega \dive_{\mesh\n}\,\bfu\n \, \dive_{\mesh\n}\,\bfv\n \dx
\\[2ex]
- a \int_\Omega \rho\n^\gamma\, \dive_{\mesh\n}\,\bfv\n \dx
- \int_{\Omega}{h\n^{\xi_3}\rho\n^\Gamma \, \dive_{\mesh\n} \bfv\n} + R_{\rm conv}(\rho\n,\bfu\n,\bfv\n) 
= \int_\Omega \bff\cdot\xPi_{\edges\n} \bfv\n\dx.
\end{multline}

The term involving the artificial pressure tends to zero as $n\to+\infty$ since $(h\n^{\xi_3} \rho\n^\Gamma)_{n\in\xN}$ converges strongly to $0$ in $\xL^{p}(\Omega)$ for some $1<p<1+\eta$ (see \eqref{unif:est}) and $(\dive_{\mesh\n} \bfv\n)_{n\in\xN}$ is bounded in $\xL^q(\Omega)$ for all $q\in(1,+\infty)$.
On the other hand, Lemma \ref{lem:d-weakmom2} gives
\[
\begin{aligned}
\big | R_{\rm conv}(\rho\n,\bfu\n,\bfv\n) \big | 
& \leq C \, h\n^{\frac 12 - \frac{1}{\Gamma} \big( \frac{3}{1+\eta}+\xi_3\big)}
	 \norm{h\n^{\xi_3} \rho^\Gamma}_{\xL^{1+\eta}(\Omega)}^{\frac{1}{\Gamma}}  \norm{\bfu\n}_{1,2,\mesh\n}^2\,  \norm{\bfv\n}_{1,2,\mesh\n} \\[2ex]
&+ C \,  h\n^{\xi_2 -\frac 1\eta- \frac{1}{\eta \Gamma}\big(\frac{3}{1+\eta} + \xi_3 \big)}  \norm{h\n^{\xi_3}\rho\n^\Gamma}_{\xL^{1+\eta}(\Omega)}^{\frac{1}{\eta \Gamma}}  \norm{\bfu\n}_{1,2,\mesh\n}\,  \norm{\bfv\n}_{1,2,\mesh\n},
\end{aligned}
\]
with $C$ independent of $n$,
so $R_{\rm conv}(\rho\n,\bfu\n,\bfv\n) \rightarrow 0$ as $n \rightarrow +\infty$ using Remark \ref{rmrk:conrol_Rconv}.
We also easily obtain the convergence of the diffusion and pressure terms.
Since $(\gradi_{\mesh\n}\bfu\n)_{n\in\xN}$ (\emph{resp}. $(\dive_{\mesh\n}\,\bfu\n)_{n\in\xN}$) weakly converges to $\gradi\bfu$ (\emph{resp}. $\dive\, \bfu$) in $\xbfL^2(\Omega)^3$, $(\rho\n^\gamma)_{n\in\xN}$ weakly converges to $\overline{\rho^\gamma}$ in $\xL^{1+\eta}(\Omega)$ and $(\gradi_{\mesh\n} \bfv\n)_{n\in\xN}$ (\emph{resp.} $(\dive_{\mesh\n} \bfv\n)_{n\in\xN})$) strongly converges to $\gradi \bfv$ (\emph{resp.} $\dive\,\bfv$) in $\xbfL^q(\Omega)^3$ for all $q\in(1,+\infty)$ we obtain:
\begin{multline*}
\lim\limits_{n\to+\infty} \Big(
 \mu \int_\Omega\gradi_{\mesh\n} \bfu\n:\gradi_{\mesh\n} \bfv\n \dx 
+ (\mu+\lambda)\int_\Omega \dive_{\mesh\n}\,\bfu\n \, \dive_{\mesh\n}\,\bfv\n \dx 
\\[2ex]
- a \int_\Omega \rho\n^\gamma\, \dive_{\mesh\n}\,\bfv\n  \dx
- \int_\Omega \bff \cdot \xPi_{\edges\n} \bfv\n \dx 
\Big )
\\[2ex]
=
\mu\, \int_\Omega \gradi \bfu : \gradi \bfv  \dx 
+ (\mu+\lambda)\,\int_\Omega \dive\, \bfu \,  \dive\, \bfv\, \dx 
- a \int_\Omega \overline{\rho^\gamma}\,  \dive\, \bfv\, \dx  - \int_\Omega \bff \cdot \bfv \dx
\end{multline*}
where the convergence of the source term is given by \eqref{eq:interpolate1bis}.

\medskip
Let us now prove the convergence of the convective term. We have:
	\begin{align}
	\nonumber - \int_\Omega(\xP_{\edges\n}\rho\n) \,(\xPi_{\edges\n}\bfu\n)&\otimes (\xPi_{\edges\n}\bfu\n): \gradi_{\edges\n} \bfv\n\,\dx \\
	\nonumber          &= - \sum_{\substack{\edge\in\edgesintn \\ \edge=K|L } } |\edge|\,\rho_\edge\, \bfu_\edge\otimes\bfu_\edge:(\bfv_L-\bfv_K)\otimes\bfn_{K,\edge} \\
	\label{mom3} &= - \sum_{\substack{\edge\in\edgesintn \\ \edge=K|L } } |\edge|\, \tfrac{1}{2} (\rho_K+\rho_L)\, \bfu_\edge\otimes\bfu_\edge:(\bfv_L-\bfv_K)\otimes\bfn_{K,\edge} + R_1^n,
	\end{align}
	with 
	\[
	R_1^n = \sum_{\substack{\edge\in\edgesintn \\ \edge=K|L } } |\edge|\, (\tfrac{1}{2} (\rho_K+\rho_L)-\rho_\edge)\, \bfu_\edge\otimes\bfu_\edge:(\bfv_L-\bfv_K)\otimes\bfn_{K,\edge}.
	\]
	Reordering the sum in the first term of \eqref{mom3} we get:
	\begin{align}
	\nonumber - \int_\Omega(\xP_{\edges\n}\rho\n)& \,(\xPi_{\edges\n}\bfu\n)\otimes (\xPi_{\edges\n}\bfu\n): \gradi_{\edges\n} \bfv\n\,\dx 
	\\
	\nonumber           &= - \frac 12\, \sum_{K\in\mesh\n} \rho_K \sum_{\substack{\edge\in\edges(K) \\ \edge=K|L}}|\edge|\, \bfu_\edge\otimes\bfu_\edge:(\bfv_L-\bfv_K)\otimes\bfn_{K,\edge} + R_1^n \\
	\label{mom4} &= - \frac 12\, \sum_{K\in\mesh\n} \rho_K \, \bfu_K\otimes\bfu_K:\sum_{\substack{\edge\in\edges(K) \\ \edge=K|L}}|\edge|\, (\bfv_L-\bfv_K)\otimes\bfn_{K,\edge} + R_1^n + R_2^n,
	\end{align}
	where $\bfu_K$ is the mean value of $\bfu\n$ over $K$ and
	\[
	R_2^n = \frac 12\, \sum_{K\in\mesh\n} \rho_K \sum_{\substack{\edge\in\edges(K) \\ \edge=K|L}}|\edge|\, (\bfu_K\otimes\bfu_K-\bfu_\edge\otimes\bfu_\edge):(\bfv_L-\bfv_K)\otimes\bfn_{K,\edge}.
	\]
	Back to \eqref{mom4} we get:
	\[
	\begin{aligned}
	- \int_\Omega(\xP_{\edges\n}\rho\n) & \,(\xPi_{\edges\n}\bfu\n)\otimes (\xPi_{\edges\n}\bfu\n): \gradi_{\edges\n} \bfv\n\,\dx \\
	&= - \sum_{K\in\mesh\n} \rho_K \, \bfu_K\otimes\bfu_K:\sum_{\substack{\edge\in\edges(K) \\ \edge=K|L}}|\edge|\, \tfrac{1}{2}(\bfv_L+\bfv_K)\otimes\bfn_{K,\edge} + R_1^n + R_2^n + R_3^n \\
	&=-\sum_{K\in\mesh\n} \rho_K \, \bfu_K\otimes\bfu_K:\sum_{\edge\in\edges(K)}|\edge|\, \bfv_\edge\otimes\bfn_{K,\edge}+ R_1^n + R_2^n + R_3^n+ R_4^n\\
	&= - \int_\Omega \rho\n \,\bfu\n\otimes \bfu\n: \overline{\gradi}_{\mesh\n} \bfv\n\,\dx  + R_1^n + R_2^n + R_3^n + R_4^n,
	\end{aligned}
	\]
	with
	\[
	\begin{aligned}
	R_3^n &= \sum_{K\in\mesh\n} \rho_K \, \bfu_K\otimes\bfu_K:\bfv_K\otimes \Big ( \sum_{\edge\in\edges(K)}|\edge|\,\bfn_{K,\edge} \Big ),\\
	R_4^n &=\sum_{K\in\mesh\n} \rho_K \, \bfu_K\otimes\bfu_K:\sum_{\substack{\edge\in\edges(K) \\ \edge=K|L}}|\edge|\, (\bfv_\edge-\tfrac12(\bfv_L+\bfv_K))\otimes\bfn_{K,\edge}.
	\end{aligned}
	\]
	Since $\bfu\n\to\bfu$ in $\xbfL^q(\Omega)$ for all $q\in[1,6)$, and $\overline{\gradi}_{\mesh\n} \bfv\n\to\gradi\bfv$ in $\xbfL^{r}(\Omega)^{3}$ for all $r\in(1,+\infty)$, we have $\bfu\n\otimes \bfu\n: \overline{\gradi}_{\mesh\n} \bfv\n\to\bfu \otimes \bfu : \gradi \bfv$ in $\xL^{3-\delta}(\Omega)$ for all $\delta\in(0,2]$. 
	Furthermore, we have $\rho\n\rightharpoonup\rho$ weakly in $\xL^{3(\gamma-1)}(\Omega)$ with $3(\gamma-1)>\frac 32$ (since $\gamma > \frac 32$), which yields:
	\[
	\lim\limits_{n\to+\infty} \ - \int_\Omega \rho\n \,\bfu\n\otimes \bfu\n: \overline{\gradi}_{\mesh\n} \bfv\n\,\dx = -\int_\Omega \rho\, \bfu \otimes \bfu : \gradi \bfv \,\dx.
	\]
	Let us now prove that $ \sum_{i=1}^4 R_i^n \to 0$ as $n\to+\infty$. 
	We begin with $R_1^n$. Recalling the upwind definition of $\rho_\edge$ and the fact that $\bfa\otimes\bfb:\bfc\otimes\bfd= (\bfa\cdot\bfc)\,(\bfb\cdot\bfd)$ for $\bfa,\bfb,\bfc,\bfd\in\xR^3$ we get:
	\[
	|R_1^n| \leq  \frac 12\, \sum_{\substack{\edge\in{\edgesintn} \\ \edge=K|L}} |\edge|\,|\rho_K-\rho_L|\, |\bfu_\edge\cdot\bfn_{K,\edge}|\,|\bfv_L-\bfv_K| \, |\bfu_\edge|.
	\]
	As a consequence we have
	\[
	\begin{aligned}
	|R_1^n| 
	&\leq \frac 12\, \Big ( \sum_{\substack{\edge\in{\edgesintn} \\ \edge=K|L}} |\edge|\,|\rho_K-\rho_L|^2\, |\bfu_\edge\cdot\bfn_{K,\edge}| \Big)^{\frac12} \, \Big ( \sum_{\substack{\edge\in{\edgesintn} \\ \edge=K|L}} |\edge|\, |\bfv_L-\bfv_K|^2 \, |\bfu_\edge|^3 \Big)^{\frac12}\\
	&\lesssim h\n^{- \frac{5}{8\Gamma}\big (\frac{3}{1+\eta} + \xi_3\big )}
	\, \Big ( \sum_{\substack{\edge\in{\edgesintn} \\ \edge=K|L}} |\edge|\, |\bfv_L-\bfv_K|^2 \, |\bfu_\edge|^3 \Big)^{\frac12}
	\end{aligned}
	\]
	by estimate \eqref{unif:est}. By Taylor's inequality applied to the smooth function $\bfv$ and the regularity of the discretization, we have $|\bfv_L-\bfv_K|^2\lesssim h\n\,|D_\edge|/|\edge|\,\norm{\gradi\bfv}_{\xbfL^\infty(\Omega)^3}^2$.
	Hence:
	\[
	|R_1^n| 
	\lesssim h\n^{\frac{5}{8\Gamma} \big (\frac{4}{5}\Gamma - \frac{3}{1+\eta} - \xi_3\big )}\, \norm{\xPi_{\edges\n}\bfu\n}_{\xbfL^3(\Omega)}^{\frac32} 
	\\
	\lesssim h\n^{ \frac{5}{8\Gamma}\big ( \frac{4}{5}\Gamma - \frac{3}{1+\eta} - \xi_3\big )} \,\norm{\bfu\n}_{\xbfL^6(\Omega)}^{\frac32} \\
	\lesssim h\n^{\frac{5}{8\Gamma} \big ( \frac{4}{5}\Gamma - \frac{3}{1+\eta} - \xi_3\big )}
	\]
	since $\norm{\bfu\n}_{\xbfL^6(\Omega)}$ is controlled by $\norm{\bfu\n}_{1,2,\mesh\n}$ which is bounded by $C_0$. 
	Since $(\Gamma,\xi_3)$ satisfy \eqref{thm:cond_xi3}, we get $R_1^n\to0$ as $n\to+\infty$.
	We now turn to $R_2^n$. We write
	\[
	\bfu_K\otimes\bfu_K-\bfu_\edge\otimes\bfu_\edge = (\bfu_K-\bfu_\edge)\otimes\bfu_K+ \bfu_\edge\otimes(\bfu_K-\bfu_\edge).
	\]
	Hence, $|R_2^n|\leq |R_{2,1}^n|+|R_{2,2}^n|$ with:
	\[
	\begin{aligned}
	|R_{2,1}^n| &= \frac12\,\sum_{K\in\mesh\n} \rho_K \sum_{\substack{\edge\in\edges(K) \\ \edge=K|L}}|\edge|\, |\bfu_K-\bfu_\edge|\,|\bfu_K|\,|\bfv_L-\bfv_K|,\\
	|R_{2,2}^n| &= \frac12\,\sum_{K\in\mesh\n} \rho_K \sum_{\substack{\edge\in\edges(K) \\ \edge=K|L}}|\edge|\, |\bfu_K-\bfu_\edge|\,|\bfu_\edge|\,|\bfv_L-\bfv_K|.
	\end{aligned}
	\]
	We only treat $|R_{2,1}^n|$, since the treatment of $|R_{2,2}^n|$ is similar.
By a Taylor inequality on the smooth function $\bfv$ and the regularity of the discretization, we have: $|\bfv_L-\bfv_K| \lesssim h\n \, \norm{\gradi \bfv}_{\xbfL^\infty(\Omega)^3}$. Hence:
	\begin{equation}
	|R_{2,1}^n| \lesssim  h\n \sum_{K\in\mesh\n} \rho_K \, |\bfu_K| \sum_{\edge \in\edges(K)} |\edge|\,|\bfu_K-\bfu_\edge|.
	\end{equation}
	Proceeding as in the proof of Proposition \ref{prop:convergence_rho} (see the computation after \eqref{mass4}) we get:
	\[
	\begin{aligned}
	|R_{2,1}^n| 
	& \lesssim  h\n \sum_{K\in\mesh\n} |K|^{\frac 12} \rho_K  \, |\bfu_K| \,\norm{\gradi \bfu\n}_{\xbfL^2(K)^3} \\
	& \lesssim h\n \, \norm{\rho\n}_{\xL^{\infty}(\Omega)}^{1-\frac{3(\gamma-1)}{2}}  \, \norm{\bfu\n}_{\xbfL^\infty(\Omega)} \sum_{K\in\mesh\n} |K|^{\frac 12} \rho_K^{\frac{3(\gamma-1)}{2}}  \norm{\gradi \bfu\n}_{\xbfL^2(K)^3}\\
	& \lesssim h\n \, \norm{\rho\n}_{\xL^{\infty}(\Omega)}^{\frac{5-3\gamma}{2}}\, \norm{\bfu\n}_{\xbfL^\infty(\Omega)}.
	\end{aligned}
	\]
	We have the inverse inequalities $\norm{\rho\n}_{\xL^{\infty}(\Omega)}\lesssim h\n^{-\frac{1}{\gamma-1}}\norm{\rho\n}_{\xL^{3(\gamma-1)}(\Omega)}$ and $\norm{\bfu\n}_{\xbfL^\infty(\Omega)}\lesssim h\n^{-\frac 12} \norm{\bfu\n}_{\xbfL^6(\Omega)}$. Thus, since $(\rho\n)_{n\in\xN}$ is bounded in $\xL^{3(\gamma-1)}(\Omega)$ and since the sequence $(\norm{\bfu\n}_{\xbfL^6(\Omega)})_{n\in\xN}$ is bounded we get:
	\[
	|R_{2,1}^n| \lesssim h\n^{1-\frac{1}{\gamma-1}\frac{5-3\gamma}{2}-\frac 12} = h\n^{\frac{2\gamma-3}{\gamma-1}}.
	\]
	Since, $\gamma>\frac32$, we get $R_{2,1}^n\to0$ as $n\to+\infty$.
	As said previously, the same holds for $R_{2,2}^n$.
	The third remainder term satisfies $R_3^n=0$ since $\sum_{\edge \in\edges(K)} |\edge| \,\bfn_{K,\edge}=0$ for all $K\in\mesh\n$.  
	Let us conclude with the control of $R_4^n$.
	Denoting $\hat\bfv_\edge=\frac12(\bfv_L+\bfv_K)$ for $\edge=K|L$, we may write $R_4^n=R_{4,1}^n+R_{4,2}^n$ with:
	\[
	\begin{aligned}
	R_{4,1}^n& =\sum_{K\in\mesh\n}  \rho_K \, \Big (\sum_{\edge \in\edges(K)} |\edge|\,(\bfu_K\otimes\bfu_K-\bfu_\edge\otimes\bfu_\edge):(\bfv_\edge-\hat\bfv_\edge)\otimes\bfn_{K,\edge}\Big ),\\
	R_{4,2}^n& =\sum_{K\in\mesh\n}  \rho_K \,\Big (\sum_{\edge \in\edges(K)} |\edge|\,\bfu_\edge\otimes\bfu_\edge:(\bfv_\edge-\hat\bfv_\edge)\otimes\bfn_{K,\edge}\Big ).
	\end{aligned}
	\]
	The term $R_{4,1}^n$ can be controlled in the same way as $R_2^n$ and we obtain $R_{4,1}^n\to0$ as $n\to+\infty$. Reordering the sum in $R_{4,2}^n$ we get:
	\[
	R_{4,2}^n = \sum_{\substack{\edge\in{\edgesintn} \\ \edge=K|L}} |\edge|\,(\rho_K-\rho_L)\, \,\bfu_\edge\otimes\bfu_\edge:(\bfv_\edge-\hat\bfv_\edge)\otimes\bfn_{K,\edge}.
	\]
	Hence $R_{4,2}^n$ can be controlled in the same way as $R_1^n$ and we obtain $R_{4,2}^n\to0$ as $n\to+\infty$. 
	This concludes the proof of \eqref{prop:eq:convergence_u}.
	
	\medskip 
	It remains to prove \eqref{eq:c-energy:disc}. We proceed as in the proof of Proposition \ref{prop:estimate:u}. Taking $\bfu\n$ as a test function in the first form of the discrete weak formulation of the momentum equation and using \eqref{eq:renorm} with $\beta=\gamma$ and $\beta=\Gamma$ we get:
	\begin{multline*}
	 \frac12 \, h\n^{\xi_1} \,   \int_{\Omega}(\tilde \rho\n-\rho^\star)\, |\xPi_{\edges\n} \bfu\n|^2 \dx  +\mu \int_\Omega |\gradi_{\mesh\n} \bfu\n|^2\dx \\
	+ (\mu+\lambda)\int_\Omega (\dive_{\mesh\n}\, \bfu\n)^2\dx 
	\leq \int_\Omega \bff\cdot\xPi_{\edges\n} \bfu\n \dx, \qquad \forall n\in\xN, 
	\end{multline*}
	 where $\tilde \rho\n$ is the piecewise constant scalar function which is equal to $\rho_\Ds$ on every dual cell $D_\edge$, and which satisfies $\tilde \rho\n>0$ (because $\rho\n>0$) and $\int_\Omega \tilde \rho\n\dx=\int_\Omega \rho\n\dx = |\Omega|\rho^\star$. Since $(\rho\n)_{n\in\xN}$ is bounded in $\xL^{\frac 32}(\Omega)$ and $(\xPi_{\edges\n} \bfu\n)_{n\in\xN}$ in $\xbfL^6(\Omega)$, the first term tends to zero as $n\to+\infty$. Thus, passing to the limit $n\to+\infty$ in the above inequality and recalling that $\gradi_{\mesh\n} \bfu\n \rightharpoonup \gradi \bfu$ weakly in $\xbfL^2(\Omega)^3$ and $\xPi_{\edges\n} \bfu\n\to\bfu$ strongly in (say) $\xbfL^2(\Omega)$ yields \eqref{eq:c-energy:disc}.

\end{proof}

\subsection{Passing to the limit in the equation of state}
\label{sec:limit:eos}

\subsubsection{Weak compactness of the effective viscous flux}\label{sec:cvf-eff-flux-d}

As in the continuous case, the equation of state is satisfied at the limit as a consequence of the compactness of the so-called effective viscous flux. Indeed, we have the following result.

\begin{prop}
\label{prop:eff_flux1}
Under the assumptions of Theorem \ref{main_thrm},
let $(\rho,\bfu,\overline{\rho^\gamma})\in \xL^{3(\gamma-1)}(\Omega)\times\xbfH^1_0(\Omega)\times\xL^{\frac{3(\gamma-1)}{\gamma}}(\Omega)$ be the limit triple of the sequence $(\rho\n,\bfu\n,\rho\n^\gamma)_{n\in\xN}$.
For $k\in\xN^*$, define
\begin{equation*}
T_k(t)= \begin{cases}
		\ t \quad & \text{if} ~ t \in [0,k), \\
		\ k \quad & \text{if} ~ t \in [k,+\infty). 
\end{cases}
\end{equation*}
The sequence $(T_k(\rho\n))_{n\in\xN}$ is bounded in $\xL^{\infty}(\Omega)$ and, up to extracting a subsequence, it converges for the weak-* topology in $\xL^\infty(\Omega)$ towards some function denoted $\overline{T_k(\rho)}$.
Then (up to extracting a subsequence) the following identity holds:
\[
\lim_{n\rightarrow +\infty} \int_{\Omega}{\big((2\mu +\lambda)\ \dive_{\mesh\n} \bfu\n - a\rho\n^\gamma \big) T_k(\rho\n) \phi \dx} = \int_{\Omega}{\big((2\mu +\lambda)\, \dive \, \bfu - a\overline{\rho^\gamma}\big) \overline{T_k(\rho)} \phi\dx}, \quad \forall \phi \in \mcal{C}_c^{\infty}(\Omega). 
\]
\end{prop}

\medskip
\begin{rmrk}{\label{rmrk:MAC}}
As in the continuous case, this result is obtained by taking the test function $\bfv=\phi \widetilde{\bfw}\n$ in the discrete momentum equation \eqref{eq:sch_mom:weak:2}, where $\widetilde{\bfw}\n$ is computed from $T_k(\rho\n)$ by applying Lemma~\ref{lem:inv-div}, \emph{i.e.} $\widetilde{\bfw}\n=\mcal{A}T_k(\rho\n)$, and satisfies $\dive\,\widetilde{\bfw}\n=T_k(\rho\n)$, $\rot\,\widetilde{\bfw}\n=0$.
Unfortunately, the discrete gradient, divergence and rotational operators associated with the Crouzeix-Raviart approximation do not satisfy a discrete equivalent of the global identity \eqref{curlcurl2}, namely
\[
\int_\Omega \gradi \bfu : \gradi \bfv \dx=  \int_\Omega \dive \, \bfu \, \dive\, \bfv \,\dx + \int_\Omega \rot \, \bfu \, \rot \, \bfv \,\dx.
\]
Instead, one needs to apply \eqref{curlcurl1} locally on each control volume $K\in\mesh\n$. 
The accumulating boundary terms must then be controlled through an estimate of $\widetilde{\bfw}\n$ in $\xbfW^{2,2}(\Omega)$. 
Moreover, it also appears in the analysis that the control of some remainder terms involving the pressure (which is controlled in $\xL^{1+\eta}(\Omega)$) requires an estimate of $\widetilde{\bfw}\n$ in $\xbfW^{2,\frac{1+\eta}{\eta}}(\Omega)$. Since $\frac{1+\eta}{\eta}\geq 2$, this latter control is more restrictive.
Such control is what motivates the introduction of the stabilization term 
\[
T_{\rm stab}^2 = - h_\mesh^{\xi_2} \Delta_{\frac{1+\eta}{\eta},\mesh}(\rho)
\] 
in the numerical scheme.  
For the MAC scheme, studied for instance in \cite{gal-18-conv}, we directly have an equivalent of \eqref{curlcurl2} and $T_{\rm stab}^2$ is useless.
\end{rmrk}

\medskip
The function $\widetilde{\bfw}\n$ defined in Remark \ref{rmrk:MAC} is not in $\xbfW^{2,\frac{1+\eta}{\eta}}(\Omega)$ because $T_k(\rho\n)$ is not in $\xW^{1,\frac{1+\eta}{\eta}}(\Omega)$. 
We rather define $\bfw\n=\mcal{A} (i_{\mesh\n}T_k(\rho\n))$ so that $\dive\,\bfw\n=i_{\mesh\n}T_k(\rho\n)$ and $\rot\,\bfw\n=0$ where $i_{\mesh}\,T_k(\rho)$ is a regularization of $T_k(\rho)$, the $\xW^{1,\frac{1+\eta}{\eta}}$ semi-norm of which is controlled by $\snorm{\rho}_{\frac{1+\eta}{\eta},\mesh}$. 
The operator $i_\mesh$ is specified in the following definition and its properties in Lemma \ref{lmm:iq}.

\medskip

\begin{defi}\label{defi:iq}
Let $\disc=(\mesh,\edges)$ be a staggered discretization of $\Omega$
and $\mcal{S}$ be the set of vertices of the primal mesh $\mesh$. For $s\in\mcal{S}$, we denote by $\mcal{N}_s\subset\mesh$ the set of the elements $K\in\mesh$ of which $s$ is a vertex. Let $p\in\xL_\mesh(\Omega)$. We denote $i_\mesh\,p$ the function defined as follows:
\begin{itemize}
 \item $i_\mesh\,p\in\mcal{C}^0(\Omega)$,
 \item for all $K\in\mesh$, the restriction of $i_\mesh\,p$ to $K$ is affine,
 \item for all $s\in\mcal{S}$, $\dsp (i_\mesh\,p)(s)=\frac{1}{{\rm card}(\mcal{N}_s)}\sum_{K\in\mcal{N}_s}p_K$.
\end{itemize}
\end{defi}

\medskip

\begin{lem}\label{lmm:iq}
Let $\disc=(\mesh,\edges)$ be a staggered discretization of $\Omega$ such that $\theta_\mesh \leq \theta_0$, with $\theta_\mesh$ defined by \eqref{eq:reg}.
For all $r\in[1,+\infty]$ there exists $C=C(r,\theta_0)$ such that:
\begin{equation}
\label{eq:iq0}
\norm{i_\mesh\,p}_{\xL^r(\Omega)} \leq C \norm{p}_{\xL^r(\Omega)}, \qquad \forall p\in\xL_\mesh(\Omega).
\end{equation}
Moreover, for all $p\in\xL_\mesh(\Omega)$ we have
$i_\mesh\,p \in \xW^{1,q}(\Omega)$ for all $q\in [1,+\infty)$ and 
there exists a constant $C=C(q,\theta_0)$ such that :
\begin{equation}
\label{eq:iq}
\norm{i_\mesh\,p-p}_{\xL^q(\Omega)} + h_\mesh\,\snorm{i_\mesh\,p}_{\xW^{1,q}(\Omega)} 
\leq C\,h_\mesh\,\snorm{p}_{q,\mesh} \qquad \forall p\in\xL_\mesh(\Omega).
\end{equation}
\end{lem}

\medskip
\begin{proof}
 The proof is similar to that of \cite[Lemma 5.8]{eym-10-conv-isen}. We skip the details.
\end{proof}

\medskip
We also have the following technical result which will be useful hereinafter. The proof can be found in \cite[Lemma 2.4]{gal-09-conv}.

\medskip
\begin{lem}\label{lmm:aedge}
Let $\disc=(\mesh,\edges)$ be a staggered discretization of $\Omega$ such that $\theta_\mesh \leq \theta_0$ (where $\theta_\mesh$ is defined by \eqref{eq:reg}) for some positive constant $\theta_0$.
Let $(n_\edge)_{\edge\in\edgesint}$ be a family of real numbers such that for all $\edge\in\edgesint$, $|n_\edge|\leq 1$, and let $\bfu\in\xbfH_\mesh(\Omega)$. 
Then, for any $q \in (1,\infty)$ there exists $C=C(q,\theta_0)$ such that:
\[
 \sum_{\edge\in\edgesint} \left| \int_\edge n_\edge\,[\bfu]_\edge\,\bff \dedge(\bfx) \right| \leq C\,h_\mesh\,\norm{\bfu}_{1,q',\mesh}\,\snorm{\bff}_{\xW^{1, q}(\Omega)}, \qquad \forall \bff\in\xW^{1,q}_0(\Omega),
\]
where $q' = \frac{q-1}{q}$,  $\norm{\bfu}_{1,q',\mesh}^2 = \sum_{K\in\mesh}\int_K|\gradi\,\bfu|^{q'}\,\dx$.
\end{lem}

\medskip

We may now give the proof Proposition \ref{prop:eff_flux1} which is similar to that of \cite[Prop. 5.9 and 5.10]{eym-10-conv-isen}. The main difference is that we here have to handle the additional convective term in the momentum balance. 

\medskip
\begin{proof}[Proof of Proposition \ref{prop:eff_flux1}]
Let $k\in\xN^*$. Since $(T_k(\rho\n))_{n\in\xN}$ is bounded in $\xL^\infty(\Omega)$ (by $k$) we have by \eqref{eq:iq0}:
\begin{equation}
\label{eq:eff_flux0}
 \norm{i_{\mesh\n}\,T_k(\rho\n)}_{\xL^{\infty}(\Omega)} \lesssim 1.
\end{equation}
Furthermore, by \eqref{eq:iq} and \eqref{unif:est} (observing that $|T_k(r_1)-T_k(r_2)|\leq |r_1-r_2|$ for all $r_1,r_2 \geq 0$) we have for all $n\in\xN$:
\begin{equation}
 \label{eq:eff_flux1}
 \norm{i_{\mesh\n}\,T_k(\rho\n)}_{\xW^{1,\frac{1+\eta}{\eta}}(\Omega)} \lesssim  h\n^{-\xi}, \quad \norm{i_{\mesh\n}\,T_k(\rho\n)-T_k(\rho\n)}_{\xL^{\frac{1+\eta}{\eta}}(\Omega)} \lesssim  h\n^{1-\xi},
\end{equation}
where, by assumption \eqref{thm:cond_xi2} on $\xi_2$,
\begin{equation}\label{eq:csq-hyp-xi2}
\xi = \frac{\eta }{1+\eta} \left[ \xi_2 +  \frac{5}{4\Gamma}\left(\frac{3}{1+\eta} + \xi_3\right) \right] < 1.
\end{equation} 
Let $(\bfw\n)_{n\in\xN}$ be the sequence of functions defined from $(i_{\mesh\n}\,T_k(\rho\n))_{n\in\xN}$ by Lemma \ref{lem:inv-div}.
We have
\[
\dive \ \bfw_n = i_{\mesh\n}\,T_k(\rho\n), \quad \rot \ \bfw_n = 0, \quad \norm{\bfw\n}_{\xbfW^{1,q}(\Omega)} \lesssim 1,
\quad \forall q \in (1,+\infty).
\]
Moreover, by the Sobolev injection $\xW^{1,q}(\Omega) \subset \xL^\infty(\Omega)$ for $q>3$, the sequence $(\bfw_n)_{n\in\xN}$ is bounded in $\xbfL^{\infty}(\Omega)$ and 
up to extracting a subsequence, as $n\to+\infty$, it strongly converges in $\xbfL^q(\Omega)$ and weakly in $\xbfW^{1,q}(\Omega)$ for all $q \in(1,+\infty)$ towards some function $\bfw$ satisfying:
\begin{equation*}
\dive \ \bfw = \overline{T_k(\rho)} \quad \text{and} \quad \rot \ \bfw = 0.
\end{equation*}
Inequality \eqref{eq:eff_flux1} and the properties of operator $\mathcal{A}$ yield
\begin{equation}
\norm{\bfw\n}_{\xbfW^{2,\frac{1+\eta}{\eta}}(\Omega)} 
\lesssim \norm{i_{\mesh\n}\,T_k(\rho\n)}_{\xW^{1,\frac{1+\eta}{\eta}}(\Omega)}
\lesssim h\n^{-\xi}.
\end{equation}

\medskip
Let $\phi\in\mcal{C}^\infty_c(\Omega)$ and take $\bfv\n=I_{\mesh\n}(\phi\bfw\n)\in \xbfH_{\mesh\n,0}(\Omega)$ as a test function in the discrete weak formulation of the momentum balance \eqref{eq:sch_mom:weak:2}. 
We get for all $n\in\xN$:
\begin{multline}
\label{eq:eff_flux2}
- \int_\Omega(\xP_{\edges\n}\rho\n) \,(\xPi_{\edges\n}\bfu\n)\otimes (\xPi_{\edges\n}\bfu\n): \gradi_{\edges\n} \bfv\n \,\dx \\[2ex]
+\mu \int_\Omega\gradi_{\mesh\n} \bfu\n:\gradi_{\mesh\n} \bfv\n \,\dx 
+ (\mu+\lambda)\int_\Omega \dive_{\mesh\n}\,\bfu\n \, \dive_{\mesh\n}\,\bfv\n \, \dx 
\\[2ex]
- a \int_\Omega \rho\n^\gamma\, \dive_{\mesh\n}\,\bfv\n \, \dx + R_1^n
 = \int_\Omega \bff\cdot\xPi_{\edges\n} \bfv\n\,\dx,
\end{multline}
where
\[
 R_1^n = - \int_\Omega h\n^{\xi_3}\rho\n^\Gamma \, \dive_{\mesh\n}\,\bfv\n \, \dx
	+ R_{\rm conv}(\rho\n,\bfu\n,\bfv\n).
\]
By Lemma \ref{lem:d-weakmom2}, we have:
\[
\begin{aligned}
\big | R_{\rm conv}(\rho\n,\bfu\n,\bfv\n) \big | 
& \leq C \, h\n^{\frac 12 - \frac{1}{\Gamma} \big( \frac{3}{1+\eta}+\xi_3\big)}
	 \norm{h\n^{\xi_3} \rho^\Gamma}_{\xL^{1+\eta}(\Omega)}^{\frac{1}{\Gamma}}  \norm{\bfu\n}_{1,2,\mesh\n}^2\,  \norm{\bfv\n}_{1,2,\mesh\n} \\[2ex]
& + C \,  h\n^{\xi_2 -\frac 1\eta- \frac{1}{\eta \Gamma}\big(\frac{3}{1+\eta} + \xi_3\big)}
	\norm{h\n^{\xi_3}\rho\n^\Gamma}_{\xL^{1+\eta}(\Omega)}^{\frac{1}{\eta \Gamma}}  \norm{\bfu\n}_{1,2,\mesh\n}\,  \norm{\bfv\n}_{1,2,\mesh\n}.
\end{aligned}
\]
Since $\norm{\bfv\n}_{1,2,\mesh\n}\lesssim \norm{\phi\bfw\n}_{\xbfH^1(\Omega)}$, we can apply Remark \ref{rmrk:conrol_Rconv} and we get that $| R_{\rm conv}(\rho\n,\bfu\n,\bfv\n) \big | \to 0$ as $n\to+\infty$.
Moreover, by \eqref{unif:est}, we have $h\n^{\xi_3}\rho\n^\Gamma\to0$ in $\xL^{p}(\Omega)$ with $1<p<1+\eta$ as $n\to+\infty$. Since $(\dive_{\mesh\n}\,\bfv\n)_{n\in\xN}$ is bounded in $\xL^{p'}(\Omega)$, we obtain that $R_1^n\to0$ as $n\to+\infty$. 
Hence, denoting $\bfdelta\n=\bfv\n-\phi\bfw\n=I_{\mesh\n}(\phi\bfw\n)-\phi\bfw\n$, we have

	\begin{multline}
	\label{eq:eff_flux3}
	- \int_\Omega(\xP_{\edges\n}\rho\n) \,(\xPi_{\edges\n}\bfu\n)\otimes (\xPi_{\edges\n}\bfu\n): \gradi_{\edges\n} \bfv\n\,\dx \\[2ex]
	+\mu \int_\Omega\gradi_{\mesh\n} \bfu\n:\gradi (\phi\bfw\n) \, \dx 
	+ (\mu+\lambda)\int_\Omega \dive_{\mesh\n}\,\bfu\n \, \dive(\phi\bfw\n) \, \dx 
	\\[2ex]
	- a\int_\Omega \rho\n^\gamma \, \dive(\phi\bfw\n) \, \dx
	+ R_2^n + \underset{n\to+\infty}{o}(1) = \int_\Omega \bff\cdot(\phi\bfw\n)\,\dx,
	\end{multline}
	where: 
	\begin{multline*}
	R_2^n = \mu \int_\Omega\gradi_{\mesh\n} \bfu\n:\gradi_{\mesh\n} \bfdelta\n  \dx 
	+ (\mu+\lambda)\int_\Omega \dive_{\mesh\n}\,\bfu\n \, \dive_{\mesh\n}\bfdelta\n \dx 
	\\[2ex]
	- a \int_\Omega \rho\n^\gamma \, \dive_{\mesh\n}\bfdelta\n \dx
	- \int_\Omega \bff\cdot\bfdelta\n \dx- \int_\Omega \bff\cdot(\xPi_{\edges\n} \bfv\n-\bfv\n)\dx.
	\end{multline*}
	By the properties of the Fortin operator $I_{\mesh\n}$, we have $\norm{\bfdelta\n}_{\xL^2(\Omega)}\lesssim h_n^2\snorm{\phi\bfw\n}_{\xbfW^{2,\frac{1+\eta}{\eta}}(\Omega)}\lesssim h\n^{2-\xi}$ and $\norm{\bfdelta\n}_{1,\frac{1+\eta}{\eta},\mesh\n}\lesssim h_n\snorm{\phi\bfw\n}_{\xbfW^{2,\frac{1+\eta}{\eta}}(\Omega)}\lesssim h\n^{1-\xi}$ with $\frac{1+\eta}{\eta} \geq 2 $. Since $( \norm{\bfu\n}_{1,2,\mesh\n})_{n\in\xN}$ is bounded, $(\rho\n^\gamma)_{n\in\xN}$ is bounded in $\xL^{1+\eta}(\Omega)$ (recall that $1+\eta=\frac{3(\gamma-1)}{\gamma}$), $\xPi_{\edges\n} \bfv\n-\bfv\n\to0$ in $\xbfL^2(\Omega)$ as $n\to+\infty$, and $\xi<1$, we get that $R_2^n\to0$ as $n\to+\infty$.
	
	\medskip
	Applying the identity \eqref{curlcurl1} over each control volume, we get:
	\begin{equation}
	\label{eq:eff_flux4}
	\mu \int_\Omega\gradi_{\mesh\n} \bfu\n:\gradi (\phi\bfw\n) \, \dx  
	 =  \mu \int_\Omega\dive_{\mesh\n} \bfu\n \, \dive (\phi\bfw\n) \, \dx 
	+ \mu \int_\Omega\rot_{\mesh\n} \bfu\n \cdot \rot (\phi\bfw\n) \, \dx
	+  R_3^n
	\end{equation}
	with $R_3^n$ which has the following structure:
	\begin{equation}\label{eq:remainder-curlcurl}
	R_3^n = \mu\,\sum_{\edge\in\edgesintn} \int_\edge \sum_{1\leq i,j,k \leq 3} n_{\edge,i,j,k}\, [(\bfu\n)_i]_\edge\, (\gradi(\phi\bfw\n))_{j,k}\, \dedge(\bfx),
	\end{equation}
	where the family of real numbers $(n_{\edge,i,j,k})_{\edge\in\edges,1\leq i,j,k\leq3}$ is uniformly bounded.
	Injecting \eqref{eq:eff_flux4} in \eqref{eq:eff_flux3} we get:
	\begin{multline}
	\label{eq:eff_flux5}
	- \int_\Omega(\xP_{\edges\n}\rho\n) \,(\xPi_{\edges\n}\bfu\n)\otimes (\xPi_{\edges\n}\bfu\n): \gradi_{\edges\n} \bfv\n\,\dx \\[2ex]
	+ (2\mu+\lambda)\int_\Omega \dive_{\mesh\n}\,\bfu\n \, \dive(\phi\bfw\n) \, \dx 
	+ \mu\, \int_\Omega\rot_{\mesh\n} \bfu\n \cdot \rot (\phi\bfw\n) \, \dx
	\\[2ex]
	- a \int_\Omega \rho\n^\gamma\, \dive(\phi\bfw\n) \, \dx 
	+ R_3^n + \underset{n\to+\infty}{o}(1)= \int_\Omega \bff\cdot(\phi\bfw\n) \, \dx.
	\end{multline}
	By Lemma \ref{lmm:aedge} with $q=2$, we have:
	\[
	 |R_3^n| \lesssim h\n \, \norm{\bfu\n}_{1,2,\mesh\n}\, \snorm{\gradi(\phi\bfw\n)}_{\xbfH^{1}(\Omega)}  \lesssim h\n \,  \norm{\bfw\n}_{\xbfW^{2,\frac{1+\eta}{\eta}}(\Omega)}\lesssim h\n^{1-\xi}.
	\]

	The choice of $\bfw\n$ gives $\dive(\phi\bfw\n)=i_{\mesh\n}T_k(\rho\n)\, \phi+\bfw\n\cdot\gradi\phi$ and $\rot(\phi\bfw\n)=L(\phi)\,\bfw\n$, where $L(\phi)$ is a matrix with entries involving first order derivatives of $\phi$. Hence, reordering \eqref{eq:eff_flux5} we have:
	\begin{multline}
	\label{eq:eff_flux6}
	 \int_\Omega \big( (2\mu+\lambda) \dive_{\mesh\n}\,\bfu\n - a \rho\n^\gamma \big)\,T_k(\rho\n)\,\phi\,\dx + R_4^n + \underset{n\to+\infty}{o}(1)
	 \\[2ex] 
	 =\int_\Omega(\xP_{\edges\n}\rho\n) \,(\xPi_{\edges\n}\bfu\n)\otimes (\xPi_{\edges\n}\bfu\n): \gradi_{\edges\n} \bfv\n\,\dx 
	- (2\mu+\lambda)\int_\Omega\dive_{\mesh\n}\,\bfu\n\,(\bfw\n\cdot\gradi\phi)\, \dx 
	\\[2ex]
	- \mu\, \int_\Omega\rot_{\mesh\n} \bfu\n \cdot (L(\phi)\,\bfw\n) \,  \dx
	+ a \int_\Omega \rho\n^\gamma \,\bfw\n\cdot\gradi\phi \, \dx  
	+\int_\Omega \bff\cdot(\phi\bfw\n)\,\dx, 
	\end{multline}
	with
	\[
	R_4^n= \int_\Omega \big((2\mu+\lambda)\,\dive_{\mesh\n}\,\bfu\n- a\rho\n^\gamma \big)\big(i_{\mesh\n}T_k(\rho\n)-T_k(\rho\n)\big)\phi\,\dx.
	\]
	Since $(\dive_{\mesh\n}\,\bfu\n)_{n\in\xN}$ is bounded in $\xL^2(\Omega)$ and $(\rho\n^\gamma)_{n\in\xN}$ in $\xL^{1+\eta}(\Omega)$, estimate \eqref{eq:eff_flux1} (with $\frac{1+\eta}{\eta}\geq 2$) yields $R_4^n\to0$ as $n\to+\infty$. Moreover, we know that $\dive_{\mesh\n}\,\bfu\n$, $\rot_{\mesh\n} \bfu\n$ (\emph{resp.} $\rho\n^\gamma$) weakly converge in $\xL^{2}(\Omega)$ (\emph{resp.} in $\xL^{1+\eta}(\Omega)$) respectively towards $\dive\,\bfu$, $\rot\,\bfu$ and $\overline{\rho^\gamma}$. 
	Since $\bfw\n$ strongly converges in $\xbfL^q(\Omega)$ towards $\bfw$ for all $q\in(1,+\infty)$, we get, passing to the limit $n\to+\infty$ in \eqref{eq:eff_flux6}:
	\begin{align}
	\lim\limits_{n\to+\infty} \ \int_\Omega \big( (2\mu+\lambda) & \dive_{\mesh\n}\,\bfu\n  - a\rho\n^\gamma \big)\,T_k(\rho\n)\,\phi\,\dx   \nonumber \\ 
	& = \lim\limits_{n\to+\infty}\int_\Omega(\xP_{\edges\n}\rho\n) \,(\xPi_{\edges\n}\bfu\n)\otimes (\xPi_{\edges\n}\bfu\n): \gradi_{\edges\n} \bfv\n\,\dx \nonumber \\[2ex]
	& - (2\mu+\lambda)\int_\Omega\dive\,\bfu\,(\bfw\cdot\gradi\phi)\, \dx 
	- \mu\, \int_\Omega\rot\, \bfu \cdot (L(\phi)\,\bfw) \,  \dx \nonumber\\[2ex]
	\label{eq:eff_flux7}
	& + a \int_\Omega \overline{\rho^\gamma}\,\bfw\cdot\gradi\phi \, \dx +\int_\Omega \bff\cdot(\phi\bfw) \,\dx.  
	\end{align}
	
	Let us now determine the limit of the convective term in the right hand side of \eqref{eq:eff_flux7}.
	As in the continuous case, we introduce a mollifying sequence $(\bfomega_\delta)_{\delta >0}$ and the regularized velocities $\bfu_{n,\delta} = \bfu\n * \bfomega_\delta$ and $\bfu_\delta = \bfu * \bfomega_\delta$ where $\bfu\n$ and $\bfu$ have been extended by $0$ outside $\Omega$. We have $\bfu_{n,\delta}\in\xbfL^6(\Omega)$ with $\norm{\bfu_{n,\delta}}_{\xbfL^6(\Omega)}\leq C\norm{\bfu_{n}}_{\xbfL^6(\Omega)}$ and for $q\in(6,+\infty]$, $\bfu_{n,\delta}\in\xbfL^q(\Omega)$ with $\norm{\bfu_{n,\delta}}_{\xbfL^q(\Omega)}\leq C_\delta \norm{\bfu_{n}}_{\xbfL^6(\Omega)}$. Moreover, for all $m\in\xN$ and $q\in[1,+\infty]$, $\bfu_{n,\delta}\in\xbfW^{m,q}(\Omega)$ with $\snorm{\bfu_{n,\delta}}_{\xbfW^{m,q}}\leq C_\delta\norm{\bfu_{n}}_{\xbfL^6(\Omega)}$.
	Furthermore, we recall that
	\begin{align}
	\bfu_{n,\delta} &\underset{n \to +\infty}{\longrightarrow} \bfu_\delta  \quad \text{strongly in} ~ \xbfL^q_{\rm loc}(\xR^3) ~\forall q\in[1,6) ~\text{uniformly in} ~\delta, \label{eq:d-cvg-reg-u-1}\\
	\bfu_{n,\delta} &\underset{\delta \to 0}{\longrightarrow} \bfu\n \quad \text{strongly in} ~ \xbfL^q_{\rm loc}(\xR^3) ~\forall q\in[1,6) ~(\text{uniformly in} ~n)  \label{eq:d-cvg-reg-u-2}, \\
	\bfu_\delta &\underset{\delta \to 0}{\longrightarrow} \bfu \quad \text{strongly in} ~\xbfL^6_{\rm loc}(\xR^3)  \label{eq:d-cvg-reg-u-3}.
	\end{align}
	Denoting $\tilde{\bfu}_{n,\delta} = I_{\mesh\n} \bfu_{n,\delta}$, we have:
	\begin{multline}
	 \label{eq:d-conv-reg-0}
	 - \int_\Omega(\xP_{\edges\n}\rho\n) \,(\xPi_{\edges\n}\bfu\n)\otimes (\xPi_{\edges\n}\bfu\n): \gradi_{\edges\n} \bfv\n\,\dx  \\
	 = - \int_\Omega(\xP_{\edges\n}\rho\n) \, (\xPi_{\edges\n}\tilde{\bfu}_{n,\delta})\otimes (\xPi_{\edges\n}\bfu\n) : \gradi_{\edges\n} \bfv\n\,\dx + R_5^{n,\delta}
	\end{multline}
	with
	\[
	R_5^{n,\delta} = - \int_\Omega(\xP_{\edges\n}\rho\n) \,\big(\xPi_{\edges\n}\bfu\n -(\xPi_{\edges\n}\tilde{\bfu}_{n,\delta})\big)\otimes (\xPi_{\edges\n}\bfu\n) : \gradi_{\edges\n} \bfv\n\,\dx.
	\]
	Since $(\rho_n \bfu_n)_{n\in\xN}$ is bounded in $\xbfL^p(\Omega)$ for some $p > \frac 65$, $(\gradi \bfw_n)_{n\in\xN}$ is bounded in $\xbfL^s(\Omega)^3$ for any $s \in (1,+\infty)$, then  the following inequality holds, for some triple $(p,q,s)$, such that $p>\frac 65$, $s>1$, $q<6$ and $\frac{1}{p}+ \frac{1}{q} + \frac{1}{s}=1$:
	\begin{align*}
	|R_5^{n,\delta}| 
	& \lesssim \norm{\rho\n \bfu\n}_{\xL^{p}(\Omega)}  \norm{\gradi_{\edges\n} \bfv\n}_{\xbfL^s(\Omega)^{3}} \norm{\xPi_{\edges\n}\bfu\n -\xPi_{\edges\n}\tilde{\bfu}_{n,\delta}}_{\xbfL^q(\Omega)} \\
	& \lesssim \norm{\rho\n \bfu\n}_{\xL^{p}(\Omega)}  \norm{\gradi \bfw\n}_{\xbfL^s(\Omega)^{3}} \norm{\xPi_{\edges\n}\bfu\n -\xPi_{\edges\n}\tilde{\bfu}_{n,\delta}}_{\xbfL^q(\Omega)} \\
	& \lesssim  \norm{\xPi_{\edges\n}\bfu\n -\xPi_{\edges\n}\tilde{\bfu}_{n,\delta}}_{\xbfL^q(\Omega)} \\
	& \lesssim \norm{\bfu\n -\tilde{\bfu}_{n,\delta}}_{\xbfL^q(\Omega)} \\
	& \lesssim \big(\norm{\bfu\n - \bfu_{n,\delta}}_{\xbfL^q(\Omega)} + \norm{\tilde{\bfu}_{n,\delta} - \bfu_{n,\delta}}_{\xbfL^q(\Omega)} \big)
	\end{align*}
	where the constants involved in these inequalities are independent of $n$ and $\delta$. From Lemma \ref{lmm:RT} we have:
	\[
	\norm{\tilde{\bfu}_{n,\delta} - \bfu_{n,\delta}}_{\xbfL^q(\Omega)} 
	\lesssim h\n^2|\bfu_{n,\delta}|_{\xbfW^{2,q}(\Omega)}
	\lesssim C_\delta h\n^2\norm{\bfu_{n}}_{\xbfL^6(\Omega)}
	\lesssim C_\delta h\n^2 .
	\]
	Therefore:
	\begin{equation}\label{eq:control-R5}
	\limsup_{n\rightarrow +\infty} |R_5^{n,\delta}|
	\lesssim \limsup_{n\rightarrow +\infty}\norm{\bfu\n - \bfu_{n,\delta}}_{\xbfL^q(\Omega)},
	\end{equation}
	where the involved constant is independent of $n$ and $\delta$.
	Let us now deal with the integral in the right hand side of \eqref{eq:d-conv-reg-0}. Performing a discrete integration by parts we get: 
	\[
	\begin{aligned}
	- \int_\Omega(\xP_{\edges\n}\rho\n) \,(\xPi_{\edges\n}\tilde{\bfu}_{n,\delta})&\otimes (\xPi_{\edges\n}\bfu\n) : \gradi_{\edges\n} \bfv\n\,\dx \\
	&= - \sum_{\substack{\edge\in\edgesintn \\ \edge=K|L } } |\edge|\,\rho_\edge\, (\tilde{\bfu}_{n,\delta})_\edge\otimes \bfu_\edge:(\bfv_L-\bfv_K)\otimes\bfn_{K,\edge} \\
	&= - \sum_{\substack{\edge\in\edgesintn \\ \edge=K|L } } |\edge|\,\rho_\edge\,(\bfu_\edge\cdot\bfn_{K,\edge})\, ((\tilde{\bfu}_{n,\delta})_\edge \cdot (\bfv_L-\bfv_K)) \\
	&=  \sum_{K\in\mesh\n} \Big( \sum_{\edge\in\edges(K)  } |\edge|\,\rho_\edge\, (\bfu_\edge\cdot\bfn_{K,\edge})\, (\tilde{\bfu}_{n,\delta})_\edge \Big ) \cdot \bfv_K.
	\end{aligned} 
	\]
	Injecting $(\tilde{\bfu}_{n,\delta})_\edge= (\tilde{\bfu}_{n,\delta})_\edge-(\tilde{\bfu}_{n,\delta})_K+(\tilde{\bfu}_{n,\delta})_K$, where $ (\tilde{\bfu}_{n,\delta})_K$ is the mean value of the function $\tilde{\bfu}_{n,\delta}$ over $K$, we get:
	\begin{multline}
	 \label{eq:d-conv-reg-1}
	- \int_\Omega(\xP_{\edges\n}\rho\n) \,(\xPi_{\edges\n}\tilde{\bfu}_{n,\delta})\otimes (\xPi_{\edges\n}\bfu\n) : \gradi_{\edges\n} \bfv\n\,\dx \\
	=  \sum_{K\in\mesh\n} \Big( \sum_{\edge\in\edges(K)  } |\edge|\,\rho_\edge\, (\bfu_\edge\cdot\bfn_{K,\edge})\, \big((\tilde{\bfu}_{n,\delta})_\edge -(\tilde{\bfu}_{n,\delta})_K\big)\Big ) \cdot \bfv_K
	+ R_6^{n,\delta}+R_7^{n,\delta}
	\end{multline}
	where, using the discrete mass conservation equation \eqref{eq:sch_mass}, we have:
	\[
	\begin{aligned}
	 &R_{6}^{n,\delta} =- h\n^{\xi_1}\sum_{K\in\mesh\n} |K|(\rho_K-\rho^\star)\,(\tilde{\bfu}_{n,\delta})_K\cdot \bfv_K,\\
	 &R_{7}^{n,\delta} = h\n^{\xi_2}\sum_{K\in\mesh\n} \Big ( \sum_{\substack{\edge\in\edges(K)\cap\edgesintn \\ \substack \edge=K|L}} |\edge|\, \Big(\dfrac{|\edge|}{|D_\edge|}\Big)^{\frac 1\eta}\,|\rho_K-\rho_L|^{\frac 1\eta -1}\,(\rho_K-\rho_L) \Big )\,(\tilde{\bfu}_{n,\delta})_K\cdot \bfv_K.
	\end{aligned}
	\]
Since $(\rho\n)_{n\in\xN}$ is bounded in $\xL^{\frac 32}(\Omega)$, $(\bfv\n)_{n\in\xN}$ is bounded in $\xbfL^{\infty}(\Omega)$ and 
\[
\norm{\tilde{\bfu}_{n,\delta}}_{\xbfL^6(\Omega)}\lesssim \norm{\bfu_{n,\delta}}_{\xbfL^6(\Omega)}\lesssim 1
\]
where the involved constants are independent of $n$ and $\delta$, we obtain that:
\begin{equation}
\label{eq:control-R6}
|R_{6}^{n,\delta}|\to 0 \quad \text{ as $n\to+\infty$ uniformly with respect to $\delta>0$}.
\end{equation}
Reordering the sum in $R_{7}^{n,\delta}$ we get:
\[
 \begin{aligned}
  R_{7}^{n,\delta} 
  &= - h\n^{\xi_2} \sum_{\substack{\edge\in\edgesintn \\ \substack \edge=K|L}} |\edge|\, \Big(\dfrac{|\edge|}{|D_\edge|}\Big)^{\frac 1\eta}\,|\rho_K-\rho_L|^{\frac 1\eta -1}\,(\rho_K-\rho_L) \,\big( (\tilde{\bfu}_{n,\delta})_L\cdot \bfv_L- (\tilde{\bfu}_{n,\delta})_K\cdot \bfv_K\big ) \\
  &= R_{7,1}^{n,\delta}  + R_{7,2}^{n,\delta},
 \end{aligned}
\]
where 
\[
 \begin{aligned}
  &R_{7,1}^{n,\delta} 
  = - h\n^{\xi_2} \sum_{\substack{\edge\in\edgesintn \\ \substack \edge=K|L}} |\edge|\, \Big(\dfrac{|\edge|}{|D_\edge|}\Big)^{\frac 1\eta}\,|\rho_K-\rho_L|^{\frac 1\eta -1}\,(\rho_K-\rho_L) \,\bfv_L \cdot \big( (\tilde{\bfu}_{n,\delta})_L- (\tilde{\bfu}_{n,\delta})_K\big ), \\
  &R_{7,2}^{n,\delta} 
  = - h\n^{\xi_2} \sum_{\substack{\edge\in\edgesintn \\ \substack \edge=K|L}} |\edge|\, \Big(\dfrac{|\edge|}{|D_\edge|}\Big)^{\frac 1\eta}\,|\rho_K-\rho_L|^{\frac 1\eta -1}\,(\rho_K-\rho_L) \,(\tilde{\bfu}_{n,\delta})_K\cdot ( \bfv_L- \bfv_K ).
 \end{aligned}
\]
The first term is controlled as follows:
\[
\begin{aligned}
  |R_{7,1}^{n,\delta}| 
  & \leq h\n^{\xi_2} \norm{\bfv\n}_{\xbfL^\infty(\Omega)} \, \sum_{\substack{\edge\in\edgesintn \\ \substack \edge=K|L}} |D_\edge|\, \Big(\dfrac{|\edge|}{|D_\edge|}\,|\rho_K-\rho_L|\Big)^{\frac 1\eta}\, \frac{|\edge|}{|D_\edge|}\big| (\tilde{\bfu}_{n,\delta})_L- (\tilde{\bfu}_{n,\delta})_K\big | \\
  & \lesssim  h\n^{\xi_2}  \, \norm{|\gradi_{\edges\n}(\rho\n)|^{\frac 1\eta}}_{\xbfL^{1+\eta}(\Omega)} \,
   \Big( \sum_{\substack{\edge\in\edgesintn \\ \substack \edge=K|L}}|D_\edge|\, \Big(\frac{|\edge|}{|D_\edge|} \Big )^{\frac{1+\eta}{\eta}} \big| (\tilde{\bfu}_{n,\delta})_L- (\tilde{\bfu}_{n,\delta})_K\big |^\frac{1+\eta}{\eta} \Big )^{\frac{\eta}{1+\eta}}
\end{aligned}
\]
where, following similar steps as in the proof of Proposition \ref{prop:convergence_rho} (see the calculation after eq. \eqref{mass4}), we have:
\[
\begin{aligned}
  \big| (\tilde{\bfu}_{n,\delta})_L- (\tilde{\bfu}_{n,\delta})_K\big |^\frac{1+\eta}{\eta} 
  & \lesssim  \big| (\tilde{\bfu}_{n,\delta})_L- (\tilde{\bfu}_{n,\delta})_\edge\big |^\frac{1+\eta}{\eta} + \big| (\tilde{\bfu}_{n,\delta})_K- (\tilde{\bfu}_{n,\delta})_\edge \big |^\frac{1+\eta}{\eta} \\[2ex]
  & \lesssim  \frac{h_L^\frac{1+\eta}{\eta}}{|L|} \,  \norm{\gradi \tilde{\bfu}_{n,\delta}}_{\xbfL^\frac{1+\eta}{\eta}(L)^3}^\frac{1+\eta}{\eta}+\frac{h_K^\frac{1+\eta}{\eta}}{|K|} \,  \norm{\gradi \tilde{\bfu}_{n,\delta}}_{\xbfL^\frac{1+\eta}{\eta}(K)^3}^\frac{1+\eta}{\eta}.
\end{aligned}
\]
By the regularity of the sequence of discretizations, we get:
\[
 \begin{aligned}
  |R_{7,1}^{n,\delta}| 
  & \lesssim  h\n^{\xi_2}  \, \norm{|\gradi_{\edges\n}(\rho\n)|^{\frac 1\eta}}_{\xbfL^{1+\eta}(\Omega)} \, \norm{\tilde{\bfu}_{n,\delta}}_{1,\frac{1+\eta}{\eta},\mesh\n}\\
  & \lesssim  h\n^{\xi_2}  \, \norm{|\gradi_{\edges\n}(\rho\n)|^{\frac 1\eta}}_{\xbfL^{1+\eta}(\Omega)} \, \snorm{\bfu_{n,\delta}}_{\xbfW^{1,\frac{1+\eta}{\eta}}(\Omega)}\\
  &\lesssim C_\delta h\n^{\xi_2}  \, \norm{|\gradi_{\edges\n}(\rho\n)|^{\frac 1\eta}}_{\xbfL^{1+\eta}(\Omega)} \, \norm{\bfu_{n}}_{\xbfL^{6}(\Omega)} \\
  &\lesssim C_\delta h\n^{\xi_2}  \, \norm{|\gradi_{\edges\n}(\rho\n)|^{\frac 1\eta}}_{\xbfL^{1+\eta}(\Omega)},
 \end{aligned}
\]
where the constants involved in $\lesssim$ are independent of $n$ (and $\delta$). By the uniform estimate \eqref{unif:est} we have
\[
\norm{|\gradi_{\edges\n}(\rho\n)|^{\frac 1\eta}}_{\xbfL^{1+\eta}(\Omega)}
= \norm{\gradi_{\edges\n}(\rho\n)}_{\xbfL^{\frac{1+\eta}{\eta}}(\Omega)}^{\frac 1\eta}
\lesssim h\n^{-\frac{1}{1+\eta}(\xi_2 + \frac{5}{4\Gamma}(\frac{3}{1+\eta} + \xi_3))}
\]
Therefore
\[
|R_{7,1}^{n,\delta}| 
\lesssim C_\delta h\n^{\frac{\eta}{1+\eta}(\xi_2 - \frac{5}{4\eta\Gamma}(\frac{3}{1+\eta} + \xi_3))}
\]
which yields, with condition \eqref{thm:cond_xi2}:
\begin{equation}
\label{eq:control-R71}
|R_{7,1}^{n,\delta}|\to 0 \quad \text{ as $n\to+\infty$ for any fixed $\delta>0$}.
\end{equation}
The second term is controlled in a similar way:
\begin{align*}
|R_{7,2}^{n,\delta}| 
& \leq h\n^{\xi_2} \norm{\tilde{\bfu}_{n,\delta}}_{\xbfL^\infty(\Omega)} \, \sum_{\substack{\edge\in\edgesintn \\ \substack \edge=K|L}} |D_\edge|\, \Big(\dfrac{|\edge|}{D_\edge}\,|\rho_K-\rho_L|\Big)^{\frac 1\eta}\, \frac{|\edge|}{|D_\edge|}( \bfv_L- \bfv_K ) \\
& \lesssim  h\n^{\xi_2} \norm{\tilde{\bfu}_{n,\delta}}_{\xbfL^\infty(\Omega)} \, \norm{|\gradi_{\edges\n}(\rho\n)|^{\frac 1\eta}}_{\xbfL^{1+\eta}(\Omega)} \, \norm{\gradi_{\edges\n} \bfv\n}_{\xbfL^{\frac{\eta}{1+\eta}}(\Omega)} \\
& \lesssim  h\n^{\xi_2} \norm{\tilde{\bfu}_{n,\delta}}_{\xbfL^\infty(\Omega)} \, \norm{|\gradi_{\edges\n}(\rho\n)|^{\frac 1\eta}}_{\xbfL^{1+\eta}(\Omega)} \, \norm{\bfv\n}_{1,\frac{\eta}{1+\eta},\mesh\n} \\
& \lesssim  h\n^{\xi_2} \norm{\tilde{\bfu}_{n,\delta}}_{\xbfL^\infty(\Omega)} \, \norm{\gradi_{\edges\n}(\rho\n)}_{\xbfL^{\frac{1+\eta}{\eta}}(\Omega)}^{\frac 1\eta} \,\snorm{\bfw\n}_{\xbfW^{1,\frac{\eta}{1+\eta}}(\Omega)}\\
& \lesssim h\n^{\frac{\eta}{1+\eta}(\xi_2 -  \frac{5}{4\eta\Gamma}(\frac{3}{1+\eta} + \xi_3))} \norm{\tilde{\bfu}_{n,\delta}}_{\xbfL^\infty(\Omega)},\\
& \lesssim C_\delta h\n^{\frac{\eta}{1+\eta}(\xi_2 - \frac{5}{4\eta\Gamma}(\frac{3}{1+\eta} + \xi_3))},
\end{align*}
where the involved constants are independent of $n$ (and $\delta$).
Using again \eqref{thm:cond_xi2}, this implies that 
\begin{equation}
\label{eq:control-R72}
|R_{7,2}^{n,\delta}|\to 0 \quad \text{ as $n\to+\infty$ for any fixed $\delta>0$}.
\end{equation}

\medskip
Let $\bfQ_{n,\delta}$ and $\Pi_{\mesh\n} \bfv\n$ be the functions defined by:
\[
\begin{aligned}
\bfQ_{n,\delta}(\bfx) &= \sum_{K\in\mesh\n} \frac{1}{|K|}\Big(\sum_{\edge\in\edges(K)} |\edge|\,\rho_\edge \, (\bfu_\edge\cdot\bfn_{K,\edge})\, \big((\tilde{\bfu}_{n,\delta})_\edge-(\tilde{\bfu}_{n,\delta})_K\big) \Big ) \Ind_K(\bfx), \\[2ex]
\Pi_{\mesh\n} \bfv\n(\bfx) &= \sum_{K\in\mesh\n} \bfv_K \, \Ind_K(\bfx),
\end{aligned}
\]
so that, back to \eqref{eq:d-conv-reg-1}, we have :
\begin{equation}
\label{eq:d-conv-reg-2}
- \int_\Omega(\xP_{\edges\n}\rho\n) \,(\xPi_{\edges\n}\tilde{\bfu}_{n,\delta})\otimes (\xPi_{\edges\n}\bfu\n) : \gradi_{\edges\n} \bfv\n\,\dx 
 = \int_\Omega \bfQ_{n,\delta}\cdot \Pi_{\mesh\n} \bfv\n \, \dx + R_6^{n,\delta}+ R_{7,1}^{n,\delta} + R_{7,2}^{n,\delta}. 
\end{equation}
Let us prove that, for a fixed $\delta >0$, $\bfQ_{n,\delta}$ weakly converges (up to a subsequence) in $\xbfL^{r}(\Omega)$ for some $r>1$ towards $\rho(\bfu\cdot\gradi)\bfu_\delta$ as $n\to+\infty$.
The sum in $\bfQ_{n,\delta}(\bfx)$ can be rearranged as follows:
\begin{multline*}
\bfQ_{n,\delta}(\bfx)
= \sum_{\substack{\edge\in\edgesintn \\ \edge=K|L } } \rho_\edge \, \Bigg ( \frac{|\edge|}{|K|}\, \big((\tilde{\bfu}_{n,\delta})_\edge-(\tilde{\bfu}_{n,\delta})_K\big)\,(\bfu_\edge \cdot\bfn_{K,\edge})\Ind_K(\bfx) \\
+  \frac{|\edge|}{|L|}\,\big((\tilde{\bfu}_{n,\delta})_\edge-(\tilde{\bfu}_{n,\delta})_L\big)\,(\bfu_\edge \cdot\bfn_{L,\edge})\Ind_L(\bfx)\Bigg ).
\end{multline*}
Proceeding as above for the control of $\big| (\tilde{\bfu}_{n,\delta})_K- (\tilde{\bfu}_{n,\delta})_\edge \big |^6$,
and invoking once again the following estimates 
\[
\norm{\tilde{\bfu}_{n,\delta}}_{1,6,\mesh\n}
\lesssim
\snorm{\bfu_{n,\delta}}_{\xbfW^{1,6}(\Omega)}
\lesssim
C_\delta \norm{\bfu\n}_{\xbfL^6(\Omega)}
\leq C_\delta \norm{\bfu\n}_{1,2,\mesh\n},
\]
combined with the estimates on $\xP_{\edges\n}\rho\n$ in $\xL^{3(\gamma-1)}(\Omega)$, $\Pi_{\edges\n}\bfu\n$ in $\xbfL^6(\Omega)$, we can prove that $(\bfQ_{n,\delta})_{n\in\xN}$ is bounded in $\xbfL^r(\Omega)$ with $r > 1$ (because $3(\gamma-1)>\frac 32$). 
Then, up to a subsequence, $\bfQ_{n,\delta}$ weakly converges towards some $\bfQ_\delta$ in $\xbfL^r(\Omega)$ as $n \rightarrow +\infty$.

\medskip
Let us now identify $\bfQ_\delta$. Let $\bfpsi\in\mcal{C}_c^\infty(\Omega)^3$ and denote $\bfpsi\n=I_{\mesh\n} \bfpsi$. Since $\bfpsi$ is smooth, we have $\Pi_{\mesh\n}\bfpsi\n\to\bfpsi$ in $\xbfL^{r'}(\Omega)$ (with $\frac{1}{r}+\frac{1}{r'}=1$). Hence we have (observing that $R_6^{n,\delta}\to0$ and $R_7^{n,\delta}\to0$ as $n\to+\infty$ with $\bfpsi\n$ instead of $\bfv\n$ ):
\[
  \int_\Omega \bfQ_\delta\cdot \bfpsi \, \dx  
   = \lim\limits_{n\to+\infty}\, \int_\Omega \bfQ_{n,\delta} \cdot \Pi_{\mesh\n} \bfpsi\n \, \dx 
   = \lim\limits_{n\to+\infty}\,- \int_\Omega(\xP_{\edges\n}\rho\n) \,(\xPi_{\edges\n}\tilde{\bfu}_{n,\delta})\otimes (\xPi_{\edges\n}\bfu\n) : \gradi_{\edges\n} \bfpsi\n\,\dx.
\]
Since $\tilde \bfu_{n,\delta}$ converges strongly as $n \rightarrow +\infty$ to $\bfu_\delta$ in $\xbfL^q(\Omega)$  for all $q<6$ (uniformly with respect to $\delta$) and $\norm{\tilde \bfu_{n,\delta}}_{1,2,\mesh\n}\leq C_\delta \norm{\bfu\n}_{1,2,\mesh\n}$, we can reproduce the same arguments as those used in the previous Subsection \ref{sec:limit:mom} (passing to the limit in the momentum equation) and obtain:
\[
 \int_\Omega \bfQ_\delta \cdot \bfpsi \, \dx = -\int_\Omega \rho\bfu_\delta \otimes \bfu : \gradi \bfpsi \,\dx=-\int_\Omega \bfu_\delta \otimes (\rho \bfu) : \gradi \bfpsi \,\dx.
\]
Since the limit functions satisfy $(\rho,\bfu)\in\xL^{3(\gamma-1)}(\Omega)\times\xbfH^1_0(\Omega)$ and since we have already proved that $\dive(\rho\bfu)=0$ in Section \ref{sec:limit:mass}, we infer that: 
\[
\bfQ_\delta = \rho(\bfu\cdot\gradi)\bfu_\delta .
\] 
Back to \eqref{eq:d-conv-reg-0} and \eqref{eq:d-conv-reg-2} we get :
\begin{multline}
\label{eq:d-conv-reg-3}
- \int_\Omega(\xP_{\edges\n}\rho\n) \,(\xPi_{\edges\n}\bfu\n)\otimes (\xPi_{\edges\n}\bfu\n): \gradi_{\edges\n} \bfv\n\,\dx +\int_\Omega \rho \bfu \otimes \bfu :\gradi(\phi \bfw)\, \dx \\[2ex]
=  R_5^{n,\delta} + R_6^{n,\delta}+ R_{7,1}^{n,\delta} + R_{7,2}^{n,\delta} +  R_8^{n,\delta}+ R_9^{\delta}. 
\end{multline}
where
\[
 \begin{aligned}
  R_8^{n,\delta} &= \int_\Omega \bfQ_{n,\delta} \cdot \Pi_{\mesh\n} \bfv\n \dx - \int_\Omega \bfQ_{k} \cdot (\phi\bfw) \, \dx, \\[2ex]
  R_9^{\delta} &= \int_\Omega \rho (\bfu-\bfu_\delta) \otimes \bfu :\gradi(\phi \bfw)\, \dx .
 \end{aligned}
\]
The function $\Pi_{\mesh\n} \bfv\n$ converges to $\phi \bfw$ strongly in $\xbfL^{r'}(\Omega)$ as $n\to+\infty$. 
Indeed, we know that in $\xbfL^{r'}(\Omega)$, $\phi\bfw_n \to \phi \bfw$ and $\bfdelta\n = \bfv\n-\phi\bfw_n \to 0$ as $n\to+\infty$ and we also have $\Pi_{\mesh\n} \bfv\n - \bfv\n \to 0$ in $\xbfL^{r'}(\Omega)$ since:
\[
\begin{aligned}
\norm{\xPi_{\mesh\n} \bfv\n-\bfv\n}_{\xbfL^{r'}(\Omega)}^{r'} 
&= \sum_{K\in\mesh\n} \int_{K} \Big |\sum_{\edge,\edge'\in\edges(K)} (\bfv_\edge-\bfv_{\edge'}) \xi_K^\edge\,\zeta_{\edge'}(\bfx)\Big |^{r'}\dx \\
&\lesssim  h\n^{r'}\,\sum_{K\in\mesh\n} h_K^{3-r'}\,\sum_{\edge,\edge'\in\edges(K)}|\bfv_\edge-\bfv_{\edge'}|^{r'}. 
\end{aligned}
\]
Hence we have $\norm{\xPi_{\mesh\n} \bfv\n-\bfv\n}_{\xbfL^{r'}(\Omega)}\lesssim h\n \,\norm{\bfv\n}_{1,r',\edges\n}\lesssim h\n\,\norm{\bfv\n}_{1,r',\mesh\n} \lesssim h\n \snorm{\phi\bfw\n}_{\xbfW^{1,r'}(\Omega)}\lesssim h\n$. Therefore, by the weak convergence of $\bfQ_{n,\delta}$ towards $\bfQ_\delta$ in $\xbfL^{r}(\Omega)$ we have:
\begin{equation}
\label{eq:control-R8}
|R_{8}^{n,\delta}|\to 0 \quad \text{ as $n\to+\infty$ for any fixed $\delta>0$}.
\end{equation}

\medskip
Combining the estimates \eqref{eq:control-R5}-\eqref{eq:control-R6}-\eqref{eq:control-R71}-\eqref{eq:control-R72}-\eqref{eq:control-R8} and passing to limit $n\to+\infty$ in \eqref{eq:d-conv-reg-3}, we obtain that:
\begin{multline}
\label{eq:d-conv-reg-4}
\limsup_{n\rightarrow +\infty} \Big | \int_\Omega(\xP_{\edges\n}\rho\n) \,(\xPi_{\edges\n}\bfu\n)\otimes (\xPi_{\edges\n}\bfu\n): \gradi_{\edges\n} \bfv\n\,\dx  - \int_\Omega \rho \bfu \otimes \bfu :\gradi(\phi \bfw)\, \dx \Big | \\[2ex]
\lesssim \limsup_{n\rightarrow +\infty}\norm{\bfu\n - \bfu_{n,\delta}}_{\xbfL^q(\Omega)} + |R_9^{\delta}|, 
\end{multline}
for some $q\in[1,6)$ and for all $\delta>0$. By \eqref{eq:d-cvg-reg-u-3} we have $R_9^{\delta}\to0$ as $\delta\to0$ and by the uniform in $n$ convergence \eqref{eq:d-cvg-reg-u-2} we finally obtain, letting $\delta\to0$ in \eqref{eq:d-conv-reg-4} that:
\[
 \lim\limits_{n\to+\infty} \int_\Omega(\xP_{\edges\n}\rho\n) \,(\xPi_{\edges\n}\bfu\n)\otimes (\xPi_{\edges\n}\bfu\n): \gradi_{\edges\n} \bfv\n\,\dx = \int_\Omega \rho \bfu \otimes \bfu :\gradi(\phi \bfw)\, \dx.
\]

\medskip
Going back to \eqref{eq:eff_flux7} we obtain:
\begin{multline}
\label{eq:eff_flux8}
\lim\limits_{n\to+\infty} \, \int_\Omega \big( (2\mu+\lambda) \dive_{\mesh\n}\,\bfu\n - a \rho\n^\gamma \big)\,T_k(\rho\n)\,\phi\,\dx   \\ 
=  \int_\Omega \rho \bfu \otimes \bfu :\gradi(\phi \bfw)\, \dx - (2\mu+\lambda)\int_\Omega\dive\,\bfu\,(\bfw\cdot\gradi\phi)\, \dx \\
- \mu\, \int_\Omega\rot\, \bfu \cdot (L(\phi)\,\bfw) \,  \dx
+ a \int_\Omega \overline{\rho^\gamma}\,\bfw\cdot\gradi\phi \, \dx +\int_\Omega \bff\cdot(\phi\bfw) \,\dx.
\end{multline}
Applying the identity \eqref{curlcurl2} to the functions $\bfu$ and $\phi\bfw\in\xbfH^1_0(\Omega)$, we get: 
\begin{align*}
&\lim\limits_{n\to+\infty} \ \int_\Omega \big( (2\mu+\lambda) \dive_{\mesh\n}\,\bfu\n - a\rho\n^\gamma \big)\,T_k(\rho\n)\,\phi\,\dx  \\
& \quad =
 \int_\Omega \big ((2\mu+\lambda)\,\dive\,\bfu- a\overline{\rho^\gamma} \big )\,\overline{T_k(\rho)}\,\phi\,\dx 
 + \int_\Omega \rho\, \bfu \otimes \bfu : \gradi (\phi\bfw) \,\dx
 \\
& \qquad 
 -\mu \int_\Omega \gradi\bfu:\gradi(\phi\bfw)\,\dx
 -(\mu+\lambda) \int_\Omega \dive\,\bfu \, \dive(\phi\bfw)\,\dx
 \\
& \qquad + a \int_\Omega \overline{\rho^\gamma} \,\dive(\phi\bfw)\,\dx
 +\int_\Omega \bff\cdot(\phi\bfw)\,\dx.
\end{align*}
We have already proved that the limit triple $(\rho,\bfu,\overline{\rho^\gamma})\in \xL^{3(\gamma-1)}(\Omega)\times\xbfH^1_0(\Omega)\times\xL^{\frac{3(\gamma-1)}{\gamma}}$ satisfies the momentum equation in the weak sense. 
Thus, applying Proposition \ref{prop:convergence_u} to $\bfv=\phi\bfw$ (using the density of $\mcal{C}_c^\infty(\Omega)^3$ in $\xbfW^{1,q}_0(\Omega)$ for all $q\in[1,+\infty)$) yields
\begin{equation*}
\lim\limits_{n\to+\infty} \ \int_\Omega \big( (2\mu+\lambda) \dive_{\mesh\n}\,\bfu\n - a\rho\n^\gamma \big)\,T_k(\rho\n)\,\phi\,\dx  =
 \int_\Omega \big ((2\mu+\lambda)\,\dive\,\bfu- a\overline{\rho^\gamma} \big )\,\overline{T_k(\rho)}\,\phi\,\dx,
\end{equation*}
thus concluding the proof of Lemma \ref{prop:eff_flux1}.
\end{proof}

\subsubsection{Strong convergence of the density and renormalization property} 

\paragraph{Properties of the truncation operators $T_k$.} 

We first state two results that are the discrete counterparts of Lemmas \ref{lem:cvg-Tk}-\ref{lem:osc}. 
\begin{lem}\label{lem:cvg-Tk:disc}
	There exists a constant $C$ such that the following inequality holds for all $1 \leq q < 3(\gamma-1)$, $n\in\xN$ and $k \in\xN^*$:
	\[
	\norm{\overline{T_k(\rho)} - \rho}_{\xL^q(\Omega)} + \norm{T_k(\rho) - \rho}_{\xL^q(\Omega)} + \norm{T_k(\rho\n) - \rho\n}_{\xL^q(\Omega)}
	\leq C k^{\frac{1}{3(\gamma-1)} - \frac{1}{q}}.
	\]	
	Consequently, as $k \rightarrow +\infty$, the sequences $(\overline{T_k(\rho)})_{k \in \mathbb{N}^*}$ and $(T_k(\rho))_{k \in \mathbb{N}^*}$ both converge strongly to $\rho$ in $\xL^q(\Omega)$ for all $q \in [1,3(\gamma-1))$.
\end{lem}

\begin{lem}\label{lem:osc:disc}
	There exists a constant $C$ such that the following estimate holds:
	\begin{equation}
	\sup_{k > 1} ~\limsup_{n \rightarrow +\infty} \norm{T_k(\rho\n) - T_k(\rho)}_{\xL^{\gamma+1}(\Omega)} \leq C.
	\end{equation}	
\end{lem}

The proofs of these two lemmas follow the same lines as in the continuous case.

\paragraph{Renormalization equation associated with $T_k$.} 

We first state a discrete renormalization property for truncated functions which is an analogous of the renormalization property stated in Remark \ref{rmk:renorm}. The proof is similar to that of Prop. \ref{prop:renorm} which is given in Appendix \ref{sec:prop:renorm}.

\begin{prop}
\label{prop:renorm:disc}
For any $b \in \mathcal{C}^1([0,+\infty))$,
denote $b_M$ the truncated function such that
\[
b_M(t) = \begin{cases}
\ b(t) \quad & \text{if} ~ t < M, \\
\ b(M) \quad & \text{if} ~ t \geq M,
\end{cases}
\]
and $[b_M]'_+$ its discontinuous derivative:
\[
[b_M]'_+(t) = \begin{cases}
\ b'(t) \quad & \text{if} ~ t < M, \\
\ 0 \quad & \text{if} ~ t \geq M.
\end{cases}
\]
Let $\disc=(\mesh,\edges)$ be a staggered discretization of $\Omega$.
If $(\rho,\bfu)\in\xL_\mesh(\Omega)\times\xbfH_{\mesh,0}(\Omega)$ satisfy the discrete mass balance \eqref{eq:sch_mass} with $\rho>0$ \emph{a.e.} in $\Omega$ (\emph{i.e.} $\rho_K>0$, $\forall K\in\mesh$) then we have:
\begin{align}\label{app:eq:discr-renorm-loc-disc-disc}
\dive \big(b_M(\rho) \bfu \big)_K + \big([b_M]'_+(\rho_K)\rho_K - b_M(\rho_K) \big)\dive (\bfu)_K + R^1_K + R^2_K + R^3_K = 0 \quad \forall K \in \mesh,
\end{align}
where
\[
\dive \big(b_M(\rho) \bfu \big)_K 
= \dfrac{1}{|K|} \sum_{\edge \in \edges(K)}{|\edge| \ b_M(\rho_\edge) \bfu_\edge \cdot \bfn_{K,\edge}} ,
\]
and
\begin{align*}
R^1_K & = \frac{1}{|K|}\sum_{\edge \in \edges(K)}|\edge| r_{K,\edge} \,( \bfu_\edge \cdot n_{K,\edge}) \quad \text{and} \quad 
		r_{K,\edge} = [b_M]'_+(\rho_K)(\rho_\edge-\rho_K) + b_M(\rho_K) - b_M(\rho_\edge), \\
R^2_K & =  h_\mesh^{\xi_2} \  [b_M]'_+(\rho_K) \frac{1}{|K|}\sum_{\edge\in\edges(K)}|\edge| \,\Big(\dfrac{|\edge|}{|D_\edge|}\Big)^\frac{1}{\eta}\,|\rho_K-\rho_L|^{\frac{1}{\eta}-1}\,(\rho_K-\rho_L), \\
R^3_K & = h_\mesh^{\xi_1} [b_M]'_+(\rho_K)(\rho_K - \rho^\star). 
\end{align*}
\end{prop}

\bigskip
\noindent
Now, for any $k\in\xN^*$ we consider the function $L_k$ introduced in Section \ref{sec:continuous} and defined as
\begin{equation*}
L_k(t) = \begin{cases} 
\ t (\ln t - \ln k - 1), & \quad \text{if} ~ t \in [0,k), \\
\ -k,  & \quad \text{if} ~ t \in [k,+\infty).
\end{cases}
\end{equation*}
We recall that $L_k \in \mathcal{C}^0([0,+\infty)) \cap \mathcal{C}^1((0,+\infty))$ and
\[
t L_k'(t) - L_k(t) = T_k(t) \quad \forall t \in [0,+\infty).
\]

\medskip
\begin{prop}\label{prop:c-renorm-Lk:disc}
Under the assumptions of Theorem \ref{main_thrm},
let $(\rho,\bfu)\in \xL^{3(\gamma-1)}(\Omega)\times\xbfH^1_0(\Omega)$ be the limit couple of the sequence $(\rho\n,\bfu\n)_{n\in\xN}$.
Then, for all $k\in\xN^*$, the following inequalities hold:
\begin{align}
&\dive_{\mesh\n}(L_k(\rho\n) \bfu\n) + T_k(\rho\n)\dive_{\mesh\n} \bfu\n + R\n= 0,  \quad &\text{in} ~\mathcal{D}'(\mathbb{R}^3), &\quad \forall n\in\xN. \label{eq:c-Lk-n:disc} \\
&\dive(L_k(\rho) \bfu) + T_k(\rho)\dive\, \bfu \geq 0 \quad &\text{in} ~\mathcal{D}'(\mathbb{R}^3), \label{eq:c-Lk:disc}
\end{align}
where the discrete function $R\n$ satisfies:
$ \displaystyle  \int_\Omega R\n \geq 0.$
\end{prop}

\medskip

\begin{proof}
To prove \eqref{eq:c-Lk-n:disc}, we apply Proposition \ref{prop:renorm} (completed by Remark \ref{rmrk:cond-b}) with the function $b=L_k$ which is a convex function satisfying $|L'_k(t)| \leq C |\ln t|$ for $t \leq 1$.
We straightforwardly obtain \eqref{eq:c-Lk-n:disc}.

\medskip
Let $M\in\xN^*$. Applying 
Proposition \ref{prop:renorm:disc} to the function $T_M(t)$ (\emph{i.e.} $T_M=b_M$ with $b=Id$) we obtain:
\begin{align*}
\dive \big(T_M(\rho) \bfu \big)_K +\big([T_M]'_+(\rho_K)\rho_K - T_M(\rho_K) \big)\, \dive (\bfu)_K + R^1_K + R^2_K + R^3_K =0,  \quad \forall K \in \mesh.
\end{align*}
Let $\phi \in \mathcal{C}^\infty_c(\Omega)$ with $\phi \geq 0$. For $n \in \xN$ define $\phi\n\in\xL_{\mesh\n}(\Omega)$ by ${\phi\n}_{|K}=\phi_K$ the mean value of $\phi$ over $K$, for $K \in \mesh\n$. Multiplying the above identify by $|K|\phi_K$ and summing over $K\in\mesh\n$ yields:
\begin{multline*}
\sum_{K\in \mesh\n} \sum_{\edge \in \edges(K)} |\edge|T_M(\rho_\edge)( \bfu_\edge \cdot \bfn_{K,\edge}) \phi_K   \\
 + \sum_{K\in \mesh\n} \big([T_M]'_+(\rho_K)\rho_K - T_M(\rho_K) \big) \, \phi_K \Big ( \sum_{\edge \in \edges(K)} |\edge|\, \bfu_\edge \cdot \bfn_{K,\edge} \Big ) + R_1^n + R_2^n + R_3^n =0,
\end{multline*}
with
\begin{align*}
R_1^n & = \sum_{K \in \mesh\n} \sum_{\edge \in \edges(K)}|\edge| \Big( [T_M]'_+(\rho_K)(\rho_\edge-\rho_K) + T_M(\rho_K) - T_M(\rho_\edge) \Big) (\bfu_\edge \cdot \bfn_{K,\edge})\phi_K,\\
R_2^n & =  h\n^{\xi_2} \sum_{K \in \mesh\n} \  [T_M]'_+(\rho_K)\phi_K \sum_{\edge\in\edges(K)}|\edge| \,\Big(\dfrac{|\edge|}{|D_\edge|}\Big)^\frac{1}{\eta}\,|\rho_K-\rho_L|^{\frac{1}{\eta}-1}\,(\rho_K-\rho_L), \\
R_3^n & =   h\n^{\xi_1} \sum_{K \in \mesh\n}|K|[T_M]'_+(\rho_K) \phi_K(\rho_K - \rho^\star).
\end{align*}
Since the function $T_M$ is concave and $\rho_\edge$ is the upwind value of the density at the face $\edge$ with respect to $\bfu_\edge\cdot\bfn_{K,\edge}$, we have $R_1^n\leq 0$. The second remainder term can be rearranged as follows:
\[
 \begin{aligned}
  R_2^n & =  h\n^{\xi_2} \sum_{\substack{\edge\in\edgesintn \\ \substack \edge=K|L}} |\edge| \,\Big(\dfrac{|\edge|}{|D_\edge|}\Big)^\frac{1}{\eta}\,|\rho_K-\rho_L|^{\frac{1}{\eta}-1}\,(\rho_K-\rho_L)\, ([T_M]'_+(\rho_K) \phi_K-[T_M]'_+(\rho_L) \phi_L), \\
  &= R_{2,1}^n+ R_{2,2}^n, 
 \end{aligned}
\]
where
\[
\begin{aligned}
  R_{2,1}^n & =  h\n^{\xi_2} \sum_{\substack{\edge\in\edgesintn \\ \substack \edge=K|L}} |\edge| \,\Big(\dfrac{|\edge|}{|D_\edge|}\Big)^\frac{1}{\eta}\,|\rho_K-\rho_L|^{\frac{1}{\eta}-1}\,(\rho_K-\rho_L) \,([T_M]'_+(\rho_K) -[T_M]'_+(\rho_L) )\,  \phi_L, \\[2ex]
  R_{2,2}^n & =  h\n^{\xi_2} \sum_{\substack{\edge\in\edgesintn \\ \substack \edge=K|L}} |\edge| \,\Big(\dfrac{|\edge|}{|D_\edge|}\Big)^\frac{1}{\eta}\,|\rho_K-\rho_L|^{\frac{1}{\eta}-1}\,(\rho_K-\rho_L)\,  [T_M]'_+(\rho_K) \,( \phi_K-\phi_L).
\end{aligned}
\]
Since $T_M$ is concave, we have $R_{2,1}^n \leq 0$.
Hence we get:
\begin{multline}\label{eq:d-renorm-weak}
\sum_{K\in \mesh\n} \sum_{\edge \in \edges(K)} |\edge|T_M(\rho_\edge)( \bfu_\edge \cdot \bfn_{K,\edge}) \phi_K  \\
 + \sum_{K\in \mesh\n} \ \big([T_M]'_+(\rho_K)\rho_K - T_M(\rho_K) \big) \, \phi_K \Big( \sum_{\edge \in \edges(K)} |\edge| \, \bfu_\edge \cdot \bfn_{K,\edge} \Big ) + R_{2,2}^n+ R_3^n \geq 0 .
\end{multline}
We want to pass to the limit $n \to +\infty$ in \eqref{eq:d-renorm-weak}. 
To that end, we show that the remainder terms $R_{2,2}^n$ and $R_3^n$ converge to $0$ as $n \to +\infty$. Observing that for all $K\in\mesh\n$, $|[T_M]'_+(\rho_K)|\leq 1$, and since $\phi$ is a smooth function, we get:
\[
 \begin{aligned}
  |R_{2,2}^n| 
  & \leq  h\n^{\xi_2} \,   \sum_{\substack{\edge\in\edgesintn \\ \substack \edge=K|L}} |\edge| \,\Big(\dfrac{|\edge|}{|D_\edge|}\Big)^\frac{1}{\eta}\,|\rho_K-\rho_L|^{\frac{1}{\eta}}\,| \phi_K-\phi_L|\\
  &\lesssim  h\n^{\xi_2} \,  \norm{\gradi \phi}_{\xbfL^\infty(\Omega)} \, \sum_{\substack{\edge\in\edgesintn \\ \substack \edge=K|L}} |D_\edge| \,\Big(\dfrac{|\edge|}{|D_\edge|}\Big)^\frac{1}{\eta}\,|\rho_K-\rho_L|^{\frac{1}{\eta}} \\
  &\lesssim  h\n^{\xi_2} \,  \norm{\gradi \phi}_{\xbfL^\infty(\Omega)} \,\norm{|\gradi_{\edges\n}(\rho\n)|^{\frac 1\eta}}_{\xbfL^{1}(\Omega)}.
 \end{aligned}
\]
Since $1+\eta>1$, H\"older's inequality yields 
\[
\norm{|\gradi_{\edges\n}(\rho\n)|^{\frac 1\eta}}_{\xbfL^{1}(\Omega)}\leq C(\Omega,\eta)\norm{\gradi_{\edges\n}(\rho\n)}_{\xbfL^{\frac{1+\eta}{\eta}}(\Omega)}^{\frac 1\eta}
\lesssim h\n^{-\frac{1}{1+\eta}(\xi_2 + \frac{5}{4\Gamma}(\frac{3}{1+\eta} + \xi_3))}.
\]
Therefore
\[
|R_{2,2}^n| 
\lesssim h\n^{\frac{\eta}{1+\eta}(\xi_2 -  \frac{5}{4\eta\Gamma}(\frac{3}{1+\eta} + \xi_3))} 
\longrightarrow 0 \quad \text{as} ~ n \rightarrow +\infty.
\]
For $R_3^n$ we may write:
\[
|R_3^n|
\lesssim h_n^{\xi_1} \,  \, \norm{\phi}_{\xL^\infty(\Omega)} \, \sum_{K \in \mesh\n}|K||\rho_K - \rho^\star| 
\lesssim 2\, |\Omega|\, \rho^\star  \, \norm{\phi}_{\xL^\infty(\Omega)} \, h_n^{\xi_1} 
\]
so that $R_{3}^n\to0$ as $n \to +\infty$.
Coming back to \eqref{eq:d-renorm-weak}, it remains to pass to the limit $n\to +\infty$ in the two terms
\[
\begin{aligned}
 \sum_{K\in \mesh\n} \sum_{\edge \in \edges(K)} |\edge|T_M(\rho_\edge)( \bfu_\edge \cdot \bfn_{K,\edge}) \phi_K 
\  \text{and} \  
\sum_{K\in \mesh\n} \big([T_M]'_+(\rho_K)\rho_K - T_M(\rho_K) \big) \, \phi_K \Big ( \sum_{\edge \in \edges(K)} |\edge| \, \bfu_\edge \cdot \bfn_{K,\edge} \Big ).
\end{aligned}
\]
On the one hand, we have by a discrete integration by parts
\begin{align*}
\sum_{K\in \mesh\n} \sum_{\edge \in \edges(K)} |\edge|T_M(\rho_\edge)( \bfu_\edge \cdot \bfn_{K,\edge}) \phi_K
& =  -\int_\Omega (\xP_{\edges\n}T_M(\rho\n))\,(\xPi_{\edges\n}\bfu\n)\cdot\gradi_{\edges\n}\phi\n \,\dx. 
\end{align*}
Then, using the same arguments as those to pass to the limit in the discrete weak formulation of the mass equation (see the proof of Proposition \ref{prop:convergence_rho} and replace $\rho\n$ by $T_M(\rho\n)$ which converges to $\overline{T_M(\rho)}$ in $\xL^\infty(\Omega)$ weak-* topology), we deduce that
\begin{align*}
\lim_{n\rightarrow +\infty} \, \sum_{K\in \mesh\n} \sum_{\edge \in \edges(K)} |\edge|T_M(\rho_\edge)( \bfu_\edge \cdot \bfn_{K,\edge}) \phi_K
= - \int_\Omega \overline{T_M(\rho)} \ \bfu \cdot \gradi \phi \,\dx.
\end{align*}
This is possible because $(T_M(\rho\n))_{n\in\xN}$ is bounded in $\xL^\infty(\Omega)$ (while $(\rho\n)_{n\in\xN}$ is bounded in $\xL^{3(\gamma-1)}$ with $3(\gamma-1) \in (\frac 32,6]$ since $\gamma\in(\frac 32,3]$) and a ``weak BV estimate'' is available for $T_M(\rho\n)$ thanks to the following inequality (recall that $|T_M(r_1)-T_M(r_2)| \leq |r_1-r_2|$ for all $r_1,r_2\geq 0$):
\[
 \sum_{\substack{\edge\in\edgesint \\ \edge=K|L}} |\edge|\, (T_M(\rho_L)-T_M(\rho_K))^2 \,|\bfu_\edge\cdot\bfn_{K,\edge}| \leq \sum_{\substack{\edge\in\edgesint \\ \edge=K|L}} |\edge|\, (\rho_L-\rho_K)^2 \,|\bfu_\edge\cdot\bfn_{K,\edge}|.
\]

\medskip
On the other hand, we have:
\begin{multline*}
  \sum_{K\in \mesh\n} \big([T_M]'_+(\rho_K)\rho_K - T_M(\rho_K) \big) \, \phi_K \Big ( \sum_{\edge \in \edges(K)} |\edge| \, \bfu_\edge \cdot \bfn_{K,\edge} \Big ) \\
 = \int_\Omega \big([T_M]'_+(\rho\n)\rho\n - T_M(\rho\n) \big)\, \dive_{\mesh\n} \bfu\n \, \phi \,\dx.
\end{multline*}
Hence, passing to the limit $n \to +\infty$  in \eqref{eq:d-renorm-weak} we obtain:
\begin{equation}
\label{eq:d-limit-renorm}
\dive\big(\overline{T_M(\rho)}\bfu\big) 
+ \overline{\big[\rho [T_M]'_+(\rho) - T_M(\rho)\big]\dive_\mesh \, \bfu} \geq 0
\quad \text{in}~ \mathcal{D}'(\xR^3)
\end{equation}
which corresponds to a relaxed version of Equation \eqref{eq:c-limit-renorm} from Section \ref{sec:continuous}.
For $k\in\xN^*$ and $\delta > 0$, we introduce the regularized function $L_{k,\delta}$ defined as
$L_{k,\delta}(t) = L_k(t+\delta)$,
the derivative of which is bounded close to $0$ unlike $L_k$.
Applying Lemma \ref{lem:diperna} (and the second part of Remark \ref{rmk:renorm}) to the pair $(\overline{T_M(\rho)},\bfu)$ (justified since $\overline{T_M(\rho)} \in L^\infty(\Omega)$ for $M$ fixed) with the function
$L_{k,\delta}$ and the source term $g = -\overline{\big[\rho [T_M]'_+(\rho) - T_M(\rho)\big]\dive_\mesh \, \bfu} \in \xL^1_{\rm loc}(\mathbb{R}^3)$, we get:
\begin{equation}
\label{eq:c-limit-renorm-Mk:disc}
\dive\big(L_{k,\delta}\big(\overline{T_M(\rho)}\big)\bfu\big) + T_{k,\delta}\big(\overline{T_M(\rho)}\big) \dive\, \bfu  
\geq  - L_{k,\delta}'\big(\overline{T_M(\rho)}\big) \overline{\big[\rho [T_M]'_+(\rho) - T_M(\rho)\big]\dive_\mesh \, \bfu}
\quad \text{in}~ \mathcal{D}'(\xR^3)
\end{equation}
where $T_{k,\delta}(t) = t L_{k,\delta}'(t)  - L_{k,\delta}(t)$.
Now, exactly as in the continuous case, we pass to the limits $M \rightarrow +\infty$ and then $\delta \rightarrow 0^+$ (see the proof of Prop. \ref{prop:c-renorm-Lk}) to get inequality \eqref{eq:c-Lk:disc}.

\end{proof}

\medskip
\paragraph{Strong convergence of the density}
\begin{prop}{\label{prop:c-cvg-rho:disc}}
Under the assumptions of Theorem \ref{main_thrm},
let $(\rho,\bfu)\in \xL^{3(\gamma-1)}(\Omega)\times\xbfH^1_0(\Omega)$ be the limit couple of the sequence $(\rho\n,\bfu\n)_{n\in\xN}$.
Up to extraction, the sequence $(\rho\n)_{n\in\xN}$ strongly converges towards $\rho$ in $\xL^q(\Omega)$ for all $q\in[1,3(\gamma-1))$. 
\end{prop}

\begin{proof}
Integrating inequalities \eqref{eq:c-Lk-n:disc} and \eqref{eq:c-Lk:disc} and summing, one obtains:
\begin{equation}
\label{strong-rho-1}
\int_{\Omega}T_k(\rho\n) \dive_{\mesh\n} \bfu\n \dx - \int_{\Omega}T_k(\rho) \dive\, \bfu \dx \leq 0, \qquad \forall n\in\xN. 
\end{equation}
Since,
$
|T_k(r_1) - T_k(r_2)|^{\gamma+1} \leq (r_1^\gamma - r_2^\gamma)(T_k(r_1) - T_k(r_2)),
$ for all $r_1$, $r_2\geq0$, 
we have
\begin{align*}
\limsup_{n\to+\infty} \int_{\Omega} |T_k(\rho\n) - T_k(\rho)|^{\gamma+1}\dx 
& \leq \limsup_{n\to+\infty} \int_{\Omega} (\rho\n^\gamma - \rho^\gamma)(T_k(\rho\n) - T_k(\rho))\dx \\
& \leq \int_{\Omega}\big(\overline{\rho^\gamma T_k(\rho)} - \overline{\rho^\gamma}\ \overline{T_k(\rho)} \big)\dx
+ \int_{\Omega}\big(\overline{\rho^\gamma} - \rho^\gamma\big) \big(\overline{T_k(\rho)} - T_k(\rho) \big)\dx.
\end{align*}
Invoking the convexity of the functions $t \mapsto t^\gamma$ and $t \mapsto - T_k(t)$, we have $\overline{\rho^\gamma} \geq \rho^\gamma$ and $\overline{T_k(\rho)} \leq T_k(\rho)$
so that
\begin{align*}
\limsup_{n\to+\infty} \int_{\Omega} |T_k(\rho\n) - T_k(\rho)|^{\gamma+1} \dx
& \leq \int_{\Omega}\big(\overline{\rho^\gamma T_k(\rho)} - \overline{\rho^\gamma}\ \overline{T_k(\rho)} \big)\dx.
\end{align*}
We can now use the weak compactness property satisfied by the effective viscous flux (Prop. \ref{prop:eff_flux1}):
\begin{align*}
&\limsup_{n\to+\infty} \int_{\Omega} |T_k(\rho\n) - T_k(\rho)|^{\gamma+1}\dx \\ 
& \leq \dfrac{2\mu+\lambda}{a} \limsup_{n\to+\infty} \int_{\Omega}\big(T_k(\rho\n) - \overline{T_k(\rho)} \big)\dive_{\mesh\n}\bfu\n \dx \\
&  = \dfrac{2\mu+\lambda}{a}  \int_{\Omega}(T_k(\rho)-\overline{T_k(\rho)}) \, \dive \, \bfu \dx
+ \limsup_{n\to+\infty} \Big ( \int_{\Omega}T_k(\rho\n) \dive_{\mesh\n} \bfu\n \dx - \int_{\Omega}T_k(\rho) \dive\, \bfu \dx \Big ) \\
&  \leq \dfrac{2\mu+\lambda}{a}  \int_{\Omega}(T_k(\rho)-\overline{T_k(\rho)}) \, \dive \, \bfu \dx,
\end{align*}
thanks to \eqref{strong-rho-1}. 
The end of the proof is the same as that of Proposition \ref{prop:c-cvg-rho}: thanks to the previous inequality we show that 
 \begin{equation*}
 \lim_{k \rightarrow +\infty} \limsup_{n\to+\infty} \int_{\Omega} |T_k(\rho\n) - T_k(\rho)|^{\gamma+1}\dx = 0,
 \end{equation*}
 and thus
 \begin{equation*}
 \lim_{k \rightarrow +\infty} \limsup_{n\to+\infty}  \norm{T_k(\rho\n) - T_k(\rho)}_{\xL^1(\Omega)} = 0.
 \end{equation*}
 We conclude to the strong convergence of the density by passing to the limits $n \rightarrow +\infty$, $k \rightarrow +\infty$ in the following inequality
 \[
 \norm{\rho - \rho\n}_{\xL^1(\Omega)} 
 \leq \norm{\rho\n - T_k(\rho\n)}_{\xL^1(\Omega)} + \norm{T_k(\rho\n) - T_k(\rho)}_{\xL^1(\Omega)} + \norm{T_k(\rho) - \rho}_{\xL^1(\Omega)}.
 \]

\end{proof}

%
%
\appendix
%

\section{Discrete renormalized equation, proof of Proposition \ref{prop:renorm}}
\label{sec:prop:renorm}

Multiplying by $b'(\rho_K)\mcal{X}_K$ the discrete mass conservation equation \eqref{eq:sch_mass} (together with the definition~\eqref{eq:def_FKedge}), one gets
\begin{multline*}
b'(\rho_K) \dfrac{1}{|K|} \sum_{\edge \in \edges(K)}{|\edge|\rho_\edge \bfu_\edge \cdot \bfn_{K,\edge}} +  h_\mesh^{\xi_1} b'(\rho_K)(\rho_K - \rho^\star) \\ 
+ h_\mesh^{\xi_2} b'(\rho_K) \dfrac{1}{|K|}  \sum_{\substack{\edge\in\edgesint \\ \edge=K|L}}|\edge| \,\Big(\dfrac{|\edge|}{|D_\edge|}\Big)^\frac{1}{\eta}\,|\rho_K-\rho_L|^{\frac{1}{\eta}-1}\,(\rho_K-\rho_L)
= 0. 
\end{multline*}
and then
\begin{align*}
&\dfrac{1}{|K|}\sum_{\edge \in \edges(K)}{|\edge|b(\rho_\sigma) \bfu_\edge \cdot \bfn_{K,\edge}}  
+ \dfrac{1}{|K|}\sum_{\edge \in \edges(K)}{|\edge|\big(b'(\rho_K)\rho_K - b(\rho_K)\big)  \bfu_\edge \cdot \bfn_{K,\edge}} \\
& + \dfrac{1}{|K|}\sum_{\edge \in \edges(K)}{|\edge|r_{K,\edge} \bfu_\edge \cdot \bfn_{K,\edge}} 
+ h_\mesh^{\xi_2} b'(\rho_K)\dfrac{1}{|K|} \sum_{\substack{\edge\in\edgesint \\ \edge=K|L}}|\edge| \,\Big(\dfrac{|\edge|}{|D_\edge|}\Big)^\frac{1}{\eta}\,|\rho_K-\rho_L|^{\frac{1}{\eta}-1}\,(\rho_K-\rho_L)
\\
& + h_\mesh^{\xi_1} b'(\rho_K)(\rho_K - \rho^\star)  = 0
\end{align*}
with
\[
r_{K,\edge} = b'(\rho_K)(\rho_\edge-\rho_K) + b(\rho_K) - b(\rho_\edge),
\]	
which corresponds to Equation \eqref{eq:discr-renorm-loc}.
Multiplying by $|K|$, summing over $K$, rearranging the sums, and using the discrete homogeneous Dirichlet boundary condition on the velocity, we get \eqref{eq:discr-renorm-glob}.

\medskip
Let us assume from now on that $b$ is convex, 
First, we have $r_{K,\edge}=0$ when $\bfu_\edge\cdot\bfn_{K,\edge} \geq 0$ (since then $\rho_\edge = \rho_K$) and, when $\bfu_\edge\cdot\bfn_{K,\edge} \leq0$, we have $r_{K,\edge} \leq 0 $ since $b$ is convex.
Hence $R^1_\edges \geq 0$.\\
Since 
$b$ is convex, we also deduce that
\[
R^2_\edges
=  h_\mesh^{\xi_2} \sum_{\substack{\edge\in\edgesint \\ \edge=K|L}} |\edge| \,\Big(\dfrac{|\edge|}{|D_\edge|}\Big)^\frac{1}{\eta}\,|\rho_K-\rho_L|^{\frac{1}{\eta}-1}\,(\rho_K-\rho_L)(b'(\rho_K)-b'(\rho_L))
\geq 0.
\]
Finally for the last remainder term $R^3_\mesh$, we combine the convexity of $b$ with a Taylor expansion and then use Jensen's inequality (recalling that $\sum_{K\in\mesh}|K|\,\rho_K = |\Omega|\rho^\star$) to get:
\[
R^3_\mesh 
 \geq h_\mesh^{\xi_1} \sum_{K\in\mesh}|K|\,(b(\rho_K)-b(\rho^\star)) 
 = h_\mesh^{\xi_1}\, |\Omega| \left( \frac{1}{|\Omega|} \int_\Omega b(\rho)\dx - b\Big( \frac{1}{|\Omega|} \int_\Omega \rho \dx \Big ) \right) 
 \geq 0.
\]


\section{Control of the convective term, proof of Lemma \ref{lem:d-weakmom2}}
\label{sec:proof:lmm:estimates_Bbar}

%
%
%

By definition, recalling that $\bfa\otimes\bfb:\bfc\otimes\bfd= (\bfa\cdot\bfc)\,(\bfb\cdot\bfd)$ for $\bfa,\bfb,\bfc,\bfd\in\xR^3$, we have:
\[
 \int_\Omega (\xP_\edges\rho)(\xPi_\edges\bfu)\otimes(\xPi_\edges\bfu): \gradi_\edges \bfv \dx = \sum_{\substack{\edge\in\edgesint \\ \edge=K|L}} |D_\edge| \, \rho_\edge\, (\bfu_\edge \cdot \bfn_{K,\edge}) \, \bfu_\edge \cdot \Big ( \frac{|\edge|}{|D_\edge|}(\bfv_L-\bfv_K) \Big ).
\]
Reordering the sum and using the definition of the primal fluxes \eqref{eq:def_FKedge} we get:
\[
\begin{aligned}
 - \int_\Omega (\xP_\edges\rho)(\xPi_\edges\bfu)\otimes(\xPi_\edges\bfu): \gradi_\edges \bfv \dx
&=\sum_{K\in \mesh} \bfv_K \cdot \sum_{\edge \in \edges(K)}|\edge| \, \rho_\edge\, (\bfu_\edge \cdot \bfn_{K,\edge}) \, \bfu_\edge \\[2ex]
& =\sum_{K\in \mesh} \bfv_K \cdot \sum_{\edge \in \edges(K)} \overline{F}_{K,\edge}(\rho,\bfu)\, \bfu_\edge + R_1
\end{aligned}
\]
where 
\[
 R_1 = -  h_\mesh^{\xi_2}\,\sum_{K\in \mesh} \bfv_K \cdot \sum_{\substack{\edge \in \edges(K)\\ \edge=K|L}} |\edge| \,\Big(\dfrac{|\edge|}{|D_\edge|}\Big)^\frac{1}{\eta}\,|\rho_K-\rho_L|^{\frac{1}{\eta}-1}\,(\rho_K-\rho_L) \,\bfu_\edge .
\]
By assumption $(H2)$ (conservativity of the dual fluxes) we may write :
\[
\begin{aligned}
 - \int_\Omega (\xP_\edges\rho)(\xPi_\edges\bfu)&\otimes(\xPi_\edges\bfu): \gradi_\edges \bfv \dx     \\
& =\sum_{K\in \mesh} \bfv_K \cdot \sum_{\edge \in \edges(K)} \overline{F}_{K,\edge}(\rho,\bfu)\, \bfu_\edge + R_1    \\
& =\sum_{K\in \mesh} \bfv_K \cdot \sum_{\edge \in \edges(K)}  \Big (  \overline{F}_{K,\edge}(\rho,\bfu)\ \bfu_\edge  
+ \sum_{\edged\in\edgesd(D_\edge), \, \edged\subset K}
\fluxd(\rho,\bfu)\,\bfu_{\edged} \Big ) + R_1.
\end{aligned}
\]
Writing $\bfv_K=\bfv_\edge+\bfv_K-\bfv_\edge$ we get:
\begin{multline}
\label{errQmom1}
- \int_\Omega (\xP_\edges\rho)(\xPi_\edges\bfu)\otimes(\xPi_\edges\bfu): \gradi_\edges \bfv \dx 
\\
=\sum_{K\in \mesh}  \sum_{\edge \in \edges(K)} \bfv_\edge \cdot \Big (  \overline{F}_{K,\edge}(\rho,\bfu)\ \bfu_\edge  
+ \sum_{\edged\in\edgesd(D_\edge), \, \edged\subset K}
\fluxd(\rho,\bfu)\,\bfu_{\edged} \Big )+ R_1 + R_2,
\end{multline}
with 
\[
 R_2= \sum_{K\in \mesh}  \sum_{\edge \in \edges(K)} (\bfv_K-\bfv_\edge) \cdot \Big (  \overline{F}_{K,\edge}(\rho,\bfu)\ \bfu_\edge  
+ \sum_{\edged\in\edgesd(D_\edge), \, \edged\subset K}
\fluxd(\rho,\bfu)\,\bfu_{\edged} \Big ).
\]
By conservativity of the primal fluxes (\ie\ using $\overline{F}_{K,\edge}(\rho,\bfu)=-\overline{F}_{L,\edge}(\rho,\bfu)$ for $\edge=K|L$) we see that the first term in the right hand side of \eqref{errQmom1} is equal to $\int_\Omega \divv_\edges(\rho\bfu\otimes\bfu) \cdot \Pi_\edges \bfv \dx$. Hence:
\[
 \Big |\int_\Omega \divv_\edges(\rho\bfu\otimes\bfu) \cdot \xPi_\edges\bfv \dx  + \int_\Omega (\xP_\edges\rho)(\xPi_\edges\bfu)\otimes(\xPi_\edges\bfu): \gradi_\edges \bfv \dx  \Big | \leq |R_1|+ |R_2|.
\]
Proving  Lemma \ref{lem:d-weakmom2} amounts to bounding $|R_1|$ and $|R_2|$. \\
We begin with $|R_1|$. 
Reordering the sum in $R_{1}$ we, get for $C=(\Omega,\gamma,\Gamma,\theta_0)$:
\begin{align*}
|R_1| 
& = h_\mesh^{\xi_2}\, \Big|\sum_{\substack{\edge \in \edgesint \\ \edge=K|L}} |\edge| \,\Big(\dfrac{|\edge|}{|D_\edge|}\Big)^\frac{1}{\eta}\,|\rho_K-\rho_L|^{\frac{1}{\eta}-1}\,(\rho_K-\rho_L) \,\bfu_\edge\cdot(\bfv_K-\bfv_L) \Big|\\
& \leq C\, h_\mesh^{\xi_2-\frac 1\eta}\,\norm{\rho}_{\xL^\infty(\Omega)}^{\frac{1}{\eta}} \,
	\left(\sum_{\edge \in \edgesint}|D_\edge||\bfu_\edge|^6\right)^{\frac 16}
    \left(\sum_{\edge \in \edgesint}|D_\edge|\left(\frac{|\edge|}{|D_\edge|}|\bfv_K -\bfv_L|\right)^{\frac 65}\right)^{\frac 56} \\
& \leq C\, h_\mesh^{\xi_2-\frac 1\eta}\,\norm{\rho}_{\xL^\infty(\Omega)}^{\frac{1}{\eta}} \, \norm{\Pi_\edges \bfu}_{\xbfL^6(\Omega)}  \, \norm{\gradi_\edges\bfv}_{\xbfL^{\frac 65}(\Omega)^{3}} \\
\end{align*}
Therefore
\begin{align*}
|R_1| 
& \leq 
	C \,h_\mesh^{\xi_2-\frac 1\eta-\frac{3}{(1+\eta)\eta\Gamma}} \,\norm{\rho^\Gamma}_{\xL^{1+\eta}(\Omega)}^{\frac{1}{\eta\Gamma}}\,\norm{\bfu}_{1,2,\mesh} \, \norm{\gradi_\edges\bfv}_{\xbfL^{2}(\Omega)^{3}}
    \\
& \leq  C \,h_\mesh^{\xi_2-\frac 1\eta-\frac{1}{\eta\Gamma}\big(\frac{3}{1+\eta}+\xi_3\big)}
\,\norm{h_\mesh^{\xi_3}\rho^\Gamma}_{\xL^{1+\eta}(\Omega)}^{\frac{1}{\eta \Gamma}}\,\norm{\bfu}_{1,2,\mesh} \, \norm{\bfv}_{1,2,\mesh}.
\end{align*}
Let us now turn to $R_2$. Recalling that $\bfu_\edged=\bfu_\edge+\frac 12(\bfu_{\edge'}-\bfu_\edge)$ and using $(H1)$, we write $R_2=R_{2,1}+ R_{2,2}$ with:
\[ \begin{aligned}
R_{2,1}= & \frac 12 \,
\sum_{K \in \mesh} \sum_{\edge \in \edges(K)} (\bfv_K-\bfv_\edge) \cdot \Big (
\sum_{\begin{array}{c} \scriptstyle \edged\in\edgesd(D_\edge) ,\\[-0.5ex] \scriptstyle \edged\subset K,\ \edged=D_\edge|D_\edge' \end{array}}
\fluxd(\rho,\bfu)\,(\bfu_{\edge'}-\bfu_\edge)\Big ),
\\[2ex]
R_{2,2}= &
\sum_{K \in \mesh} \sum_{\edge \in \edges(K)} (\bfv_K-\bfv_\edge)\cdot \ \bfu_\edge\ \xi_K^\edge \Bigl(\sum_{\edge'\in\edges(K)} \overline{F}_{K,\edge'}(\rho,\bfu) \Bigr).
\end{aligned}\]
The assumption $(H3)$ yields, for $C=C(\Omega,\theta_0)$:
\[
\begin{aligned}
|\fluxd(\rho,\bfu)| 
&\leq C \Big ( \norm{\rho}_{\xL^{\infty}(\Omega)} \norm{\xPi_\edges \bfu}_{\xbfL^\infty(\Omega)}\ h_K^{2}  +    \norm{\rho}_{\xL^{\infty}(\Omega)}^{\frac{1}{\eta}}\, h_\mesh^{\xi_2+1-\frac 1\eta}\, h_K \Big )\\
&\leq C \Big ( \norm{\rho}_{\xL^{\infty}(\Omega)} \norm{\bfu}_{\xbfL^\infty(\Omega)}\, h_\mesh \, h_K   +    \norm{\rho}_{\xL^{\infty}(\Omega)}^{\frac{1}{\eta}}\, h_\mesh^{\xi_2+1-\frac 1\eta}\, h_K \Big ).
\end{aligned}
\]
Since $\bfv_K$ is a convex combination of $(\bfv_\edge)_{\edge \in \edges(K)}$:
\begin{multline*} 
\Bigl|\sum_{\edge \in \edges(K)} (\bfv_K-\bfv_\edge) \cdot \Big (
\sum_{\begin{array}{c} \scriptstyle \edged\in\edgesd(D_\edge),\\
	\scriptstyle \edged\subset K,\ \edged=D_\edge|D_\edge' 
	\end{array}}
\fluxd(\rho,\bfu)\,(\bfu_{\edge'}-\bfu_\edge)  \Big ) \Bigr| 
\\ 
\leq C \Big ( \norm{\rho}_{\xL^{\infty}(\Omega)} \norm{\bfu}_{\xbfL^\infty(\Omega)}\ h_\mesh
+    \norm{\rho}_{\xL^{\infty}(\Omega)}^{\frac{1}{\eta}}\, h_\mesh^{\xi_2+1-\frac 1\eta} \Big ) \sum_{ \substack{\edge,\,\edge' \\ \edge'',\,\edge''' \in \edges(K)}}
h_K \ |\bfv_\edge-\bfv_{\edge'}|\ |\bfu_{\edge''}-\bfu_{\edge'''}|,
\end{multline*}
and, for $\edge,\ \edge' \in \edges(K)$, the quantity $|\bfu_\edge-\bfu_{\edge'}|$ (or $|\bfv_\edge-\bfv_{\edge'}|$) appears in the sum a finite number of times which depends on the number of faces of $K$. Hence, applying the Cauchy-Schwarz inequality and Lemma \ref{lmm:H1ns}, we may write
\begin{equation}\label{eq:R1}
\begin{aligned}
|R_{2,1}| & 
\leq C \Big ( \norm{\rho}_{\xL^{\infty}(\Omega)} \norm{\bfu}_{\xbfL^\infty(\Omega)}\ h_\mesh
+    \norm{\rho}_{\xL^{\infty}(\Omega)}^{\frac{1}{\eta}}\, h_\mesh^{\xi_2+1-\frac 1\eta} \Big )\, \norm{\bfu}_{1,2,\edges}\,   \norm{\bfv}_{1,2,\edges}
\\ & 
\leq C \Big ( \norm{\rho}_{\xL^{\infty}(\Omega)} \norm{\bfu}_{\xbfL^\infty(\Omega)}\ h_\mesh
+    \norm{\rho}_{\xL^{\infty}(\Omega)}^{\frac{1}{\eta}}\, h_\mesh^{\xi_2+1-\frac 1\eta} \Big )\, \norm{\bfu}_{1,2,\mesh}\,   \norm{\bfv}_{1,2,\mesh}
\end{aligned}
\end{equation}
We then get, for $C=C(\Omega,\gamma,\Gamma,\theta_0)$:
\[
\begin{aligned}
|R_{2,1}| 
&\leq C \, h_\mesh^{1-\frac 12 - \frac{1}{\Gamma}\big(\frac{3}{1+\eta}+\xi_3\big )}
	 \norm{h_\mesh^{\xi_3} \rho^\Gamma}_{\xL^{1+\eta}(\Omega)}^{\frac{1}{\Gamma}}  \norm{\bfu}_{1,2,\mesh}^2\,  \norm{\bfv}_{1,2,\mesh} \\[2ex]
& + C \,  h_\mesh^{\xi_2 +1-\frac1\eta-\frac{1}{\eta\Gamma}\big(\frac{3}{1+\eta}+\xi_3\big)} \norm{h_\mesh^{\xi_3}\rho^\Gamma}_{\xL^{1+\eta}(\Omega)}^{\frac{1}{\eta \Gamma}}  \norm{\bfu}_{1,2,\mesh}\,  \norm{\bfv}_{1,2,\mesh}.
\end{aligned}
\]
The estimation of $R_{2,2}$ follows similar steps. Indeed by definition of $\bfv_K$, we have  
\[
\sum_{\edge \in \edges(K)} \xi_K^\edge\ (\bfv_K-\bfv_\edge)=0,
\]
and we obtain that:
\[
R_{2,2}=
\sum_{K \in \mesh}\ \sum_{\edge \in \edges(K)} (\bfv_K-\bfv_\edge) \cdot \xi_K^\edge\ (\bfu_\edge-\bfu_K)\ \Bigl[\sum_{\edge'\in\edges(K)} \overline{F}_{K,\edge'}(\rho,\bfu) \Bigr],
\]
so, once again, denoting $\bfu_K=\sum_{\edge \in \edges(K)} \xi_K^\edge\ \bfu_\edge$:
\[
\begin{aligned}
|R_{2,2}| & \leq C \Big ( \norm{\rho}_{\xL^{\infty}(\Omega)} \norm{\bfu}_{\xbfL^\infty(\Omega)}\ h_\mesh
+    \norm{\rho}_{\xL^{\infty}(\Omega)}^{\frac{1}{\eta}}\, h_\mesh^{\xi_2+1-\frac 1\eta} \Big ) 
\sum_{K \in \mesh}\ h_K \sum_{\edge \in \edges(K)}  |\bfv_\edge-\bfv_K|\ |\bfu_\edge-\bfu_K| \\[2ex]
&\leq C \, h_\mesh^{1-\frac 12 - \frac{1}{\Gamma}\big(\frac{3}{1+\eta}+\xi_3\big )}
	 \norm{h_\mesh^{\xi_3} \rho^\Gamma}_{\xL^{1+\eta}(\Omega)}^{\frac{1}{\Gamma}}  \norm{\bfu}_{1,2,\mesh}^2\,  \norm{\bfv}_{1,2,\mesh} \\[2ex]
& + C \,  h_\mesh^{\xi_2 +1-\frac1\eta-\frac{1}{\eta\Gamma}\big(\frac{3}{1+\eta}+\xi_3\big)} \norm{h_\mesh^{\xi_3}\rho^\Gamma}_{\xL^{1+\eta}(\Omega)}^{\frac{1}{\eta \Gamma}}  \norm{\bfu}_{1,2,\mesh}\,  \norm{\bfv}_{1,2,\mesh}.
\end{aligned}
\]

\subsection*{Acknowledgements}
The authors warmly thank Thierry Gallou\"et and Rapha\`ele Herbin for the fruitful discussions with them and their encouragements.
This work was partially supported by a CNRS PEPS JCJC grant and by the SingFlows project, grant ANR-18-CE40-0027 of the French National Research Agency (ANR).

\bibliographystyle{plain}
\bibliography{bibfile.bib}

\begin{thebibliography}{10}

\bibitem{califs}
CALIF$^3$S.
\newblock A software components library for the computation of reactive
  turbulent flows.
\newblock \\ \texttt{https://gforge.irsn.fr/gf/project/isis}.

\bibitem{cro-73-con}
M.~Crouzeix and P.A. Raviart.
\newblock Conforming and nonconforming finite element methods for solving the
  stationary {S}tokes equations.
\newblock {\em RAIRO S\'erie Rouge}, 7:33--75, 1973.

\bibitem{diperna1989}
R.J. DiPerna and P.-L. Lions.
\newblock Ordinary differential equations, transport theory and {S}obolev
  spaces.
\newblock {\em Inventiones mathematicae}, 98(3):511--547, 1989.

\bibitem{dro-18-gra}
J.~Droniou, R.~Eymard, T.~Gallou{\"e}t, C.~Guichard, and R.~Herbin.
\newblock {\em The gradient discretisation method}, volume~82.
\newblock Springer, 2018.

\bibitem{ern2013}
A.~Ern and J.-L. Guermond.
\newblock {\em Theory and practice of finite elements}, volume 159.
\newblock Springer Science \& Business Media, 2013.

\bibitem{eym-00-book}
R.~Eymard, T.~Gallou{\"e}t, and R.~Herbin.
\newblock {\em The finite volume method}.
\newblock Handbook for Numerical Analysis, P.G. Ciarlet and J.-L. Lions
  Editors, North Holland, 2000.

\bibitem{eym-10-conv-isen}
R.~Eymard, T.~Gallou{\"e}t, R.~Herbin, and J.-C. Latch{\'e}.
\newblock A convergent finite element-finite volume scheme for the compressible
  {S}tokes problem. {P}art {II}: the isentropic case.
\newblock {\em Mathematics of Computation}, 79:649--675, 2010.

\bibitem{feireisl2001}
E.~Feireisl.
\newblock On compactness of solutions to the compressible isentropic
  {N}avier-{S}tokes equations when the density is not square integrable.
\newblock {\em Commentationes Mathematicae Universitatis Carolinae},
  42(1):83--98, 2001.

\bibitem{feireisl2016}
E.~Feireisl, T.G. Karper, and M.~Pokorn\'{y}.
\newblock {\em Mathematical theory of compressible viscous fluids}.
\newblock Advances in Mathematical Fluid Mechanics. Birkh\"{a}user/Springer,
  Cham, 2016.
\newblock Analysis and numerics, Lecture Notes in Mathematical Fluid Mechanics.

\bibitem{feireisl2018}
E.~Feireisl and M.~Luk{\'a}{\v{c}}ov{\'a}-Medvid’ov{\'a}.
\newblock Convergence of a mixed finite element--finite volume scheme for the
  isentropic {N}avier-{S}tokes system via dissipative measure-valued solutions.
\newblock {\em Foundations of Computational Mathematics}, 18(3):703--730, 2018.

\bibitem{fettah2013}
A.~Fettah and T.~Gallou{\"e}t.
\newblock Numerical approximation of the general compressible {S}tokes problem.
\newblock {\em IMA Journal of Numerical Analysis}, 33(3):922--951, 2013.

\bibitem{gal-08-unc}
T.~Gallou{\"e}t, L.~Gastaldo, R.~Herbin, and J.-C. Latch{\'e}.
\newblock An unconditionally stable pressure correction scheme for compressible
  barotropic {N}avier-{S}tokes equations.
\newblock {\em Mathematical Modelling and Numerical Analysis}, 42:303--331,
  2008.

\bibitem{gal-09-conv}
T.~Gallou{\"e}t, R.~Herbin, and J.-C. Latch{\'e}.
\newblock A convergent finite element-finite volume scheme for the compressible
  {S}tokes problem. p{}art {I}: The isothermal case.
\newblock {\em Mathematics of Computation}, 78(267):1333--1352, 2009.

\bibitem{gallouet2012}
T.~Gallou{\"e}t, R.~Herbin, and J.-C. Latch{\'e}.
\newblock ${W}^{1, q}$ stability of the {F}ortin operator for the {MAC} scheme.
\newblock {\em Calcolo}, 49(1):63--71, 2012.

\bibitem{gal-18-conv}
T.~Gallou{\"e}t, R.~Herbin, J.-C. Latch{\'e}, and D.~Maltese.
\newblock {Convergence of the MAC scheme for the compressible stationary
  {N}avier-{S}tokes equations}.
\newblock {\em {Mathematics of Computation}}, 87:1127--1163, 2018.

\bibitem{gallouet2015}
T.~Gallou{\"e}t, R.~Herbin, D.~Maltese, and A.~Novotn{\' y}.
\newblock Error estimates for a numerical approximation to the compressible
  barotropic {N}avier--{S}tokes equations.
\newblock {\em IMA Journal of Numerical Analysis}, 36(2):543--592, 2015.

\bibitem{girault2012}
V.~Girault and P.-A. Raviart.
\newblock {\em Finite element methods for {N}avier-{S}tokes equations: theory
  and algorithms}, volume~5.
\newblock Springer Science \& Business Media, 2012.

\bibitem{har-68-num}
F.H. Harlow and A.A. Amsden.
\newblock Numerical calculation of almost incompressible flow.
\newblock {\em Journal of Computational Physics}, 3:80--93, 1968.

\bibitem{har-71-num}
F.H. Harlow and A.A. Amsden.
\newblock A numerical fluid dynamics calculation method for all flow speeds.
\newblock {\em Journal of Computational Physics}, 8:197--213, 1971.

\bibitem{har-65-num}
F.H. Harlow and J.E. Welsh.
\newblock Numerical calculation of time-dependent viscous incompressible flow
  of fluid with free surface.
\newblock {\em Physics of Fluids}, 8:2182--2189, 1965.

\bibitem{herbin2018}
R.~Herbin, J.-C. Latch{\'e}, and K.~Saleh.
\newblock Low mach number limit of some staggered schemes for compressible
  barotropic flows.
\newblock {\em arXiv preprint arXiv:1803.09568}, 2018.

\bibitem{karper2013}
T.K. Karper.
\newblock A convergent {FEM-DG} method for the compressible {N}avier-{S}tokes
  equations.
\newblock {\em Numerische Mathematik}, 125(3):441--510, 2013.

\bibitem{lar-91-how}
B.~Larrouturou.
\newblock How to preserve the mass fractions positivity when computing
  compressible multi-component flows.
\newblock {\em Journal of Computational Physics}, 95:59--84, 1991.

\bibitem{lat-18-disc}
J.-C. Latch\'e, B.~Piar, and K.~Saleh.
\newblock A discrete kinetic energy preserving convection operator for variable
  density flows on locally refined staggered meshes.
\newblock {\em In preparation}, 2018.

\bibitem{lat-18-conv}
J.-C. Latch\'{e} and K.~Saleh.
\newblock A convergent staggered scheme for the variable density incompressible
  {N}avier-{S}tokes equations.
\newblock {\em Math. Comp.}, 87(310):581--632, 2018.

\bibitem{lio-98-comp}
P.-L. Lions.
\newblock {\em {M}athematical {T}opics in {F}luid {M}echanics: {V}olume 2:
  {C}ompressible {M}odels}, volume~2.
\newblock Oxford University Press, 1998.

\bibitem{novo2002}
S.~Novo and A.~Novotn{\' y}.
\newblock On the existence of weak solutions to the steady compressible
  {N}avier-{S}tokes equations when the density is not square integrable.
\newblock {\em Journal of Mathematics of Kyoto University}, 42(3):531--550,
  2002.

\bibitem{novotny2004}
A.~Novotn{\' y} and I.~Stra{\v s}kraba.
\newblock {\em Introduction to the mathematical theory of compressible flow},
  volume~27.
\newblock Oxford University Press on Demand, 2004.

\bibitem{ran-92-sim}
R.~Rannacher and S.~Turek.
\newblock Simple nonconforming quadrilateral {S}tokes element.
\newblock {\em Numerical Methods for Partial Differential Equations},
  8:97--111, 1992.

\bibitem{stum-80-bas}
F.~Stummel.
\newblock Basic compactness properties of nonconforming and hybrid finite
  element spaces.
\newblock {\em ESAIM: Mathematical Modelling and Numerical
  Analysis-Mod{\'e}lisation Math{\'e}matique et Analyse Num{\'e}rique},
  14(1):81--115, 1980.

\end{thebibliography}
\end{document}